\documentclass[11pt]{amsart}

\usepackage{graphicx} 

\usepackage[utf8]{inputenc}
\usepackage[T1]{fontenc}
\usepackage[english]{babel}
\usepackage{amsmath}
\usepackage{amsthm}
\usepackage{amssymb}
\usepackage{mathrsfs}
\usepackage{dsfont}
\usepackage{array}
\usepackage{graphicx} 
\usepackage[left=3cm,right=3cm,top=3cm,bottom=3cm]{geometry}
\usepackage{hyperref}

\usepackage{tikz-cd}
\usetikzlibrary{3d}
\usetikzlibrary{shapes}
\usetikzlibrary[patterns]
\usetikzlibrary{decorations.pathreplacing}
\usepackage{caption}
\usepackage{subcaption}
\usepackage{adjustbox}
 \usepackage{comment}
\usepackage{enumitem}
\setitemize{label=$\circ$}

\usepackage{eucal} 

\usepackage{mlmodern}

\newcommand{\bcd}{\begin{center}\begin{tikzcd}}
\newcommand{\ecd}{\end{tikzcd}\end{center}}

\newcommand{\vir}[1]{[#1]^\mathrm{vir}}

\newcommand{\ev}{\mathrm{ev}}
\newcommand{\pt}{\mathrm{pt}}

\newcommand{\m}[1]{\mfk\left[ #1 \right]}

\newcommand{\Jac}{\mathfrak{Jac}}
\newcommand{\Pic}{\operatorname{Pic}}

\newcommand{\DR}{\operatorname{DR}}
\newcommand{\DDR}{\operatorname{DDR}}

\newcommand{\DRL}{\operatorname{DRL}}
\newcommand{\DDRL}{\operatorname{DDRL}}

\newcommand{\LogPic}{\operatorname{LogPic}}

\newcommand{\Spec}{\operatorname{Spec}}

\newcommand{\Gm}{\mathbb G_m}
\newcommand{\Gmlog}{\mathbb G_{m,\rm{log}}}
\newcommand{\Gmtrop}{\mathbb G_{m,\rm{trop}}}

\newcommand{\gp}{\mathrm{gp}}

\newcommand{\Prim}{\operatorname{Prim}}

\newcommand{\mzero}{{\underline{0}}}

\newcommand{\PP}{\mathbb{P}}
\newcommand{\NN}{\mathbb{N}}
\newcommand{\ZZ}{\mathbb{Z}}
\newcommand{\RR}{\mathbb{R}}
\newcommand{\CC}{\mathbb{C}}

\newcommand{\QQ}{\mathbb{Q}}

\newcommand{\DDD}{\mathscr{D}}

\newcommand{\bfa}{\mathbf{a}}
\newcommand{\bfb}{\mathbf{b}}
\newcommand{\tbfa}{\widetilde{\mathbf{a}}}
\newcommand{\tbfb}{\widetilde{\mathbf{b}}}

\newcommand{\onedelta}{(\frac{1}{\delta},\frac{1}{\delta})}
\newcommand{\oner}{(\frac{1}{r},\frac{1}{r})}

\newcommand{\tS}{\widetilde{S}}
\newcommand{\tSigma}{\widetilde{\Sigma}}

\newcommand{\bfF}{ {\boldsymbol F} }
\newcommand{\bfG}{ {\boldsymbol G} }

\newcommand{\mfk}{\mathfrak{m}}

\renewcommand{\div}{\mathrm{div}}
\newcommand{\Div}{\operatorname{Div}}

\renewcommand{\P}{\mathcal{P}}
\renewcommand{\L}{\mathcal{L}}

\newcommand{\M}{\mathcal{M}}
\newcommand{\R}{\mathcal{R}}
\newcommand{\G}{\mathcal{G}}

\newcommand{\T}{\mathcal{T}}

\newcommand{\C}{\mathcal{C}}

\newcommand{\B}{\mathcal{B}}
\newcommand{\Q}{\mathcal{Q}}
\newcommand{\W}{\mathcal{W}}

\renewcommand{\O}{\mathcal{O}}

\newcommand{\U}{\mathcal{U}}

\newcommand{\Mgnbar}{\bar{\M}_{g,n}}
\newcommand{\Mgasigma}{\bar{\M}_{g,\bfa}^\sigma}

\renewcommand{\geq}{\geqslant}
\renewcommand{\tilde}{\widetilde}
\renewcommand{\bar}{\overline}

\newcommand{\gen}[1]{\langle #1 \rangle}

\newtheorem{theo}{Theorem}[section]
\newtheorem*{theom}{Theorem}
\newtheorem{prop}[theo]{Proposition}
\newtheorem{coro}[theo]{Corollary}
\newtheorem{lem}[theo]{Lemma}
\newtheorem{Notation}[theo]{Notation}

\newtheorem{conj}{Conjecture}[section]

\theoremstyle{definition}
\newtheorem{defi}[theo]{Definition}

\newtheorem{remark}[theo]{Remark}
\newenvironment{rem}[1]{
    \begin{remark}#1}{
    \xqed{\blacklozenge}\end{remark}
}

\theoremstyle{remark}
\newtheorem{mainexample}[theo]{Main Example}

\theoremstyle{remark}
\newtheorem{example}[theo]{Example}
\newenvironment{expl}[1]{
    \begin{example}#1}{
    \xqed{\lozenge}\end{example}
}

\newcommand{\xqed}[1]{
    \leavevmode\unskip\penalty9999 \hbox{}\nobreak\hfill
    \quad\hbox{\ensuremath{#1}}}

\def\floor (#1) at (#2,#3); {
    \node[draw,ellipse, minimum width=1cm, minimum height = 0.6 cm] (#1) at (#2,#3) {$\bullet$} ;
}

\def\ufloor (#1) at (#2,#3) (#4); {
    \node[draw,ellipse, minimum width=1cm, minimum height = 0.6 cm] (#1) at (#2,#3) {\scriptsize #4} ;
}

\def\marked (#1) to (#2) pos=#3 in=#4 out=#5; {
   \draw (#1) to[out=#5,in=#4] node[pos=#3] {$\bullet$} (#2) ;
}

\def\leftedge (#1) to (#2) pos=#3 in=#4 out=#5 w=#6; {
   \draw (#1) to[out=#5,in=#4] node[midway,left] {$#6$} (#2) ;
}
\def\leftmarked (#1) to (#2) pos=#3 in=#4 out=#5 w=#6; {
   \draw (#1) to[out=#5,in=#4] node[pos=#3] {$\bullet$} node[midway,left] {$#6$} (#2) ;
}

\def\wlmarked (#1) to (#2) pos=#3 in=#4 out=#5 w=#6; {
   \draw (#1) to[out=#5,in=#4] node[pos=#3] {$\bullet$} node[midway,left] {$#6$} (#2) ;
}

\def\rightedge (#1) to (#2) pos=#3 in=#4 out=#5 w=#6; {
   \draw (#1) to[out=#5,in=#4] node[midway,right] {$#6$} (#2) ;
}
\def\rightmarked (#1) to (#2) pos=#3 in=#4 out=#5 w=#6; {
   \draw (#1) to[out=#5,in=#4] node[pos=#3] {$\bullet$} node[midway,right] {$#6$} (#2) ;
}

\def\doublemarked (#1) to (#2) pos=#3 in=#4 out=#5; {
   \draw (#1) to[out=#5,in=#4] node[pos=#3] {$\bullet$} (#2) ;
}

\usepackage{todonotes}
\newcommand{\fadd}[1]{\textcolor{orange}{#1}}
\newcommand{\fnote}[1]{\todo[color=blue!50,inline]{{\scriptsize Francesca:} \scriptsize #1}}

\newcommand{\tnote}[1]{\todo[color=red!50, inline]{{\scriptsize Thomas :} \scriptsize #1}}
\newcommand{\tadd}[1]{\textcolor{red}{#1}}

\makeatletter
\def\l@subsection{\@tocline{2}{0pt}{2.5pc}{5pc}{}}
\makeatother

\makeatletter
\renewcommand{\l@section}{\@tocline{1}{0pt}{10pt}{1pc}{\bfseries}}
\makeatother

\title{A correlated refinement of the double double ramification cycle}
\author{Thomas Blomme, Francesca Carocci, Ajith Urundolil Kumaran}

\address{Thomas Blomme, Univ Toulouse, INUC, UT2J, INSA Toulouse, TSE, CNRS, IMT, Toulouse, France.}
\email{thomas.blomme@math.univ-toulouse.fr}

\address{Università di Roma Tor Vergata, Via della Ricerca Scientifica 1, Roma 00133, Italy}
\email{carocci@mat.uniroma2.it}

\address{Massachusetts Institute of Technology, Cambridge MA, United States of America}
\email{ajith270@mit.edu}

\begin{document}

\maketitle

\begin{abstract}
Given a family of semi-stable curves together with two degree $0$ line bundles, the double double ramification cycle  measures the locus where both line bundles are trivial on the fibers. When the two line bundles come equipped with natural roots, we provide a refinement of the DDR-class using the  Weil pairing of the roots. We prove that the refined classes satisfy a multiple cover formula analogous to the one for correlated invariants of projective bundles on elliptic curves proved in \cite{blomme2025multiple}. As a consequence, we prove that log-GW invariants of  toric surfaces can be refined taking into account the position of the points mapped to the toric boundary, and that these refined invariants also satisfy a multiple cover formula; the latter is as a variation of the N. Takahashi conjecture for genus zero maximal contact curves for $\mathbb P^2$ relative a smooth elliptic curve $E.$

\end{abstract}

\setcounter{tocdepth}{2}
\tableofcontents

\section{Introduction}

The moduli space of relative/logarithmic stable maps to a log Calabi-Yau pair $(X,D)$ for $X$ a smooth surface and $D$ an anti-canonical divisor (e.g $(\mathbb P^2, E)$) can be refined by keeping track of a torsion data for the points mapping to the boundary divisors (see Section~\ref{sec:generalizedtakahasi} for details). These kind of refinements have first been considered in \cite{gross2010tropical,bousseau2019takahashi} for the case of genus zero, maximal contact order, and more recently in \cite{blomme2024correlated,blommecarocci2025DR} for any genus and tangency profile  in the case of $\mathbb P^1$-bundles over elliptic curves.

In any of these geometric settings, keeping track of the torsion information allows to compute multiple cover contributions; in particular these refined invariants (named \emph{correlated} Gromov-Witten invariants by the first two authors) satisfy or are conjectured to satisfy \emph{multiple cover formulas} (MCF for short) \cite[Conjectures~6.3]{gross2010tropical}, \cite[Theorem~1.1]{bousseau2019takahashi}, \cite[Theorem~5.15]{blomme2025multiple}.

\medskip

The multiple cover formulas for these torsion-refined invariants might  be thought as a generalization to the log Calabi-Yau surface setting of the MCF conjectured (and now proved in several instances \cite{blomme2022abelian3,blomme2025multiple,klemm2010noether,pandharipande2016katz}) for the reduced Gromov-Witten theory of abelian and K3 surfaces.
For the case of $\mathbb P^1$-bundles on elliptic curve, the relation between reduced and correlated multiple cover formulas can be understood via degeneration formula; it is the  key to the proof of Oberdieck's  conjecture for abelian surfaces presented in \cite{blomme2025multiple}.

\medskip

The above suggests that there should be a broad geometric setting in which correlated invariants and their MCF phenomena should be defined and studied. We advance a proposal for what such a geometric setting should be and formulate a  precise conjecture in Section~\ref{sec:generalizedtakahasi}; the latter contains \cite[Conjectures~6.3]{gross2010tropical}, \cite[Theorem~1.1]{bousseau2019takahashi}, \cite[Theorem~5.15]{blomme2025multiple} as special cases.

\medskip

The present paper is dedicated to explore the parallel picture in the case of logarithmic Gromov -Witten invariants for toric surfaces relative to their toric boundary. 
In particular we will define the
\emph{correlated refinement} of log Gromov-Witten theory and prove that these refined invariants satisfy a multiple cover formula. 

For this geometry, the log GW invariants are controlled by the log double double ramification cycles \cite{ranganathan2024logarithmic} and the results are thus first presented in terms of the DDR class.

\subsection{Double ramification cycles}

Let $C\xrightarrow{\pi} S$ be a log smooth curve over a smooth algebraic stack with logarithmic structure and $\L$ a line bundle of degree zero on the fibers of $\pi$. In the open $S^{\circ}\subset S$ where $\pi$ is smooth, we can consider 
 the double ramification locus: the codimension $g$ substack defined by
\[\operatorname{DRL}(\L) = \{ s\in S^{\circ} \;\; |\;\; \L\rvert_{C_s}\simeq\O_{C_s}\}.\]

For example, for $S=\Mgnbar, \;S^{\circ}=\M_{g,n}$
and $\bfa=(a_1,\dots,a_n)\in\mathbb Z^n$ a vector satisfying $\sum_{i=1}^n a_i=0$ (we will call  $\bfa$ a \textit{ramification profile} or \emph{tangency profile}), we can take the line bundle $\O(\bfa)=\O(\sum_{i=1}^n a_ip_i)$. The double ramification locus is  defined by:
\[\operatorname{DRL}(\bfa) = \{ (C,p_1,\dots,p_n) \;\; |\;\; \O(\sum_{i=1}^n a_ip_i)\simeq\O\}.\]

Constructing a modular compactification of these double ramification loci and finding a tautological formula for their classes have been questions of central interest for over two decades
\cite{faber2005relative,janda2020double,holmes2021extending,marcuswiselog}.

In the most general formulation, these questions have been answered in \cite{bae2023pixton}, where the authors defined a cohomology class, morally the Poincar\'e dual of the closure of the zero section, admitting a \emph{universal Pixton formula} in the operational Chow ring of the universal Picard stack over prestable curves $\mathfrak{Pic}_{g,n,0}.$

As the data of $C\to S,\L$ as before defines a morphism $\varphi_{\L}\colon S\to \mathfrak{Pic}_{g,n,0},$ the so-called \emph{UniDR} formula induces by pull-back an expression for the DR class $\DR(\L)$ of the DR locus in all geometric situations. This recovers in particular the previously proved instances of Pixton's formula \cite{janda2017DRcurves,janda2020double}.


\medskip
Alongside with the double ramification cycle, it is natural to consider (compactifications) of intersections of DR loci: 
let $\L_1,\L_2$ two degree zero line bundles on the fibers of $C\xrightarrow{\pi} S$
the so-called \textit{double double ramification locus} (DDR for short) is defined by
\[\operatorname{DRL}(\L_1,\L_2)=\{s \in S^{\circ}\text{ s.t. }\L_1\rvert_{C_s}\simeq\L_2\rvert_{C_s}\simeq\O\}\subset S^{\circ}.\]
For example, if $\bfa$ and $\bfb$ are two ramification profiles,  
\[\operatorname{DRL}(\bfa,\bfb)=\{(C,p_1,\dots,p_n) \text{ s.t. }\O(\bfa)\simeq\O(\bfb)\simeq\O\}\subset\M_{g,n}.\]

Again, it is possible to construct a modular compactifcation of this locus, endowed with a (virtual) class defining a class  $\DDR(\L_1,\L_2)$ (see for example \cite{holmes2022intersection} or \cite[Section~1.6]{ranganathan2024logarithmic}) of the expected codimension in $\rm{CH}_*(S).$

\medskip

This was first considered in \cite{holmes2019multiplicativity}, where it was observed that the naive definition of DDR cycle as the product of $\DR(\L_1)$ and $\DR(\L_2)$ fails to have some desired invariant properties which are known to hold over the locus of compact type curves. The authors have the key insight that the multiplicativity properties for the DDR cycle can be restored working with classes in b-Chow (the b-Chow ring is defined as the colimit of
the Chow rings of all smooth blowups of $\Mgnbar$) and, at the end of the paper, \cite[Section~9]{holmes2019multiplicativity} the authors  outline how to approach the problem via logarithmic geometry.

\medskip

The logarithmic approach is then fully carried out in \cite{holmes2022intersection} working with the so called \emph{logarithmic double ramification cycle} (see also \cite{holmes2021extending,holmes2022intersection,holmes2025logDR}), a lift of the usual DR cycle to a class in $\rm{logCH}_*(\Mgnbar).$ Around the same time, Molcho-Ranganathan \cite{molcho2024case} develop an independent approach to intersection theory of higher DR loci, inspired by tropical compactifications and tropical intersections.

The upshot in both cases is that naive product formula remains true for logDR-cycles, as proven in \cite{holmes2022intersection}, allowing to see that also higher DR loci cycles are \emph{tautological}.

\medskip

In \cite{ranganathan2024logarithmic}, Ranganathan and the third named author developed a  technique that relates the logarithmic Gromov–Witten cycles of rigid and rubber geometries. Combined with  the aforementioned, they deduce that for maps to a toric variety relative to its full toric boundary, all logarithmic Gromov–Witten push-forwards lie in the tautological ring of the moduli space of curves. 
This result confirmed the importance of higher DR loci in enumerative geometry.

\subsection{Correlated refinement of DDR loci}

Let us momentarily go back to consider the DDR-locus
$\DDR(\L_1,\L_2)$ over the locus of smooth curves $S^{\circ}.$


\subsubsection{Correlation}

We assume that the previous $\L_1,\L_2$ come endowed with canonical $r$-roots. Equivalently, we rather work with $\L_1^r,\L_2^r$, so that $\L_1$ and $\L_2$ are natural roots. In this case, over the DDR-locus $\DDR(\L_1^r,\L_2^r)$, the line bundles $\L_1$ and $\L_2$ are $r$-torsion. In the example of $\Mgnbar$, we can consider the case when the ramification profiles have a non-trivial common divisor $r$.
 
The subgroup of $r$-torsion line bundles $\Pic^0_{ C/S^{\circ}}[r]$ (fiber-wise abstractly isomorphic to $\ZZ_r^{2g}$), is endowed with the Weil pairing \cite{milneAV}, a non-degenerate skew-symmetric form:
$$W_r\colon \Pic^0_{ C/S^{\circ}}[r]\times \Pic^0_{ C/S^{\circ}}[r]\to\mu_{r}.$$
We will call $\theta=W_r(\L_1,\L_2)$ the \emph{correlator} of the torsion line bundles.
Since the value of the correlator is locally constant on the moduli of pairs of roots, and the locus where $\L_i^r$ is trivial can be also described as the locus where $\L_i\cong \R_i$ for $\R_i$ the universal torsion bundles,  we see that the DDR-locus is a union of open and closed components, which we call \textit{correlated components}, where $\theta$ stays constant.


\subsubsection{Results}


Working with the logarithmic Jacobian \cite{molchowiselog} and using the logarithmic Weil pairing introduced in \cite{blommecarocci2025DR},
we extend the definition of correlators to compactifications of the DDR-locus over the boundary of $S$, and define a refinement of the DDR cycles\footnote{The reader is welcome to keep in mind the example  of $\Mgnbar$ at all time; this is ultimately the most geometrically meaningful example, especially for the application to log GW theory of toric varieties}:
\[\mathbf{DDR}(\L_1^r,\L_2^r) = \sum_{\theta}\DDR^\theta(\L_1^r,\L_2^r)\cdot (\theta)  \in \rm{CH}_*(S,\QQ)\otimes \QQ[\mu_r],\]
where the sum is over possible correlators $\theta\in\mu_{r}$. Using the canonical embedding $\mu_r\hookrightarrow\RR/\ZZ$, these can also be seen as elements in the group algebra $\QQ[\RR/\ZZ]$. 
We refer to Section \ref{sec-alg-recoll} for basic definitions and properties of the group algebra $\QQ[\RR/\ZZ]$, and to \cite{blomme2025multiple} for a compelling motivation to work with such gadget. 

\medskip

Combining the results about the multiplicativity of the \emph{logarithmic} double ramification cycles \cite{holmes2019multiplicativity}; the results of \cite{holmes2025logDR} telling us how to find logarithmic modification of $S$ over which the logDR class is represented by the UniDR formula for a certain line bundle; 
 the techniques of \cite{blommecarocci2025DR} (recalled in more details below), 
 one can show that the refined cycles $\DDR^\theta(\L_1^r,\L_2^r)$ 
are still tautological (see Corollary~\ref{cor:tautological}), and moreover they satisfy the following 
multiple cover formula (MCF):
\begin{theom}[\textbf{\ref{theo:MCF-FFR}}]
 The class $\DDR^\theta(\L_1^r,\L_2^r)$ only depends on the order of the correlator $\theta$. Furthermore, the correlated DDR-cycle satisfies the cycle theoretic $2g$-MCF:

\begin{equation}\label{eq:MCF1}
\operatorname{DDR}^\theta(\L_1^r,\L_2^r) = \sum_{k|r,\theta}k^{2g}\operatorname{DDR}^{\theta_\mathrm{prim}}(\L_1^{r/k},\L_2^{r/k}),
\end{equation}
where in the $k$-summand, $\theta_\mathrm{prim}$ is a primitive $(r/k)$-root of unity.
   \end{theom}

Contrarily to the $\PP^1$-bundle over elliptic base case \cite{blomme2024correlated}, the dependence on the order does not immediately follow from the definition. The choice of $\theta_\mathrm{prim}$ among primitive roots does not matter since they all have the same contribution by the first part of the theorem.

\medskip

In fact, we will see Section~\ref{sec:higher-DR-log} that the correlated refinement lifts to the logarithmic DDR class and in Theorem~\ref{theo:MCF-FFR} we actually prove that the multiple cover formula \eqref{eq:MCF1} holds replacing  $\mathrm{DDR}^\theta(\L_1^r,\L_2^r)$ with $\mathrm{logDDR}^\theta(\L_1^r,\L_2^r).$ The statement above then becomes an immediate corollary by \cite{holmes2022intersection}.




\medskip

The equation \eqref{eq:MCF1} above may then be reformulated using the formalism of the group algebra. Let us denote by  $T_k=\frac{1}{k}\sum_{k\theta\equiv 0}(\theta)$. A consequence of the fact that the correlated cycles only depends on the order of $\theta$, is that $\mathbf{DDR}(\L_1^r,\L_2^r)$ belongs to the span of the $T_k$ for $k|r$. If we denote by $\mathrm{Prim}_r$ the $T_r$-coefficient, then the multiple cover formula for the correlated DDR cycle is equivalent to the following functional equation:

$$\mathbf{DDR}(\L_1^r,\L_2^r) = \sum_{k|r}k^{2g}\Prim_{r/k}\mathbf{DDR}(\L_1^{r/k},\L_2^{r/k})T_{r/k}.$$


\subsection{Correlated refinement for log GW of toric surfaces}

Considering the refinement for the case $S=\Mgnbar$ and $\bfa,\bfb$ two ramification profiles with non trivial common divisor $r$, adapting 
to the correlated refinement  the arguments in \cite{ranganathan2024logarithmic} expressing logarithmic Gromov-Witten invariants in terms of products of logDR-loci, we prove a multiple cover formula for \emph{correlated log GW invariants of toric surfaces}, introduced in Section~\ref{sec:correlatedlogGW}.

\begin{theom}[\textbf{\ref{thm:toriccorrelated}}]
    Let $(X_{\Sigma},\partial X_{\Sigma})$ be a toric surface with its full toric boundary and $\bfa,\bfb$ tangency profiles supported on $\Sigma$.\footnote{In the context of logarithmic GW of toric surface the tangency profile is often given as $n$ vectors in $\mathbb Z^2$ rather than two vectors in $\mathbb Z^n.$ The data is obviously equivalent. Then being supported on $\Sigma$ means that the $n$ vectors $(a_i,b_i)$ are each parallel to one of the rays of the fan.}
    For fixed insertions $\alpha\in \rm{CH}^*(\Mgnbar),$  $\underline{\gamma}\in \rm{CH}^*(Ev_{(\bfa, \bfb)}(X))$ and fixed correlator $\theta,$ the correlated log GW invariants $\langle \alpha ; \underline{\gamma} \rangle_{g,(\bfa, \bfb)}^{\theta}$  satisfy the following MCF

\begin{equation}\label{eq:MCF2}
\langle \alpha ; \underline{\gamma} \rangle_{g,(\bfa, \bfb)}^{\theta}=\sum_{k|\bfa,\bfb,\theta}k^{3g-3+n-\deg\alpha} \langle \alpha ; \underline{\gamma} \rangle_{g,(\frac{\bfa}{k}, \frac{\bfb}{k})}^{\theta_{\rm{prim}}} .
\end{equation}
\end{theom}

\begin{expl}
    Take $X=\PP^2$ and consider curves of genus $g$ and degree $2d$ in $\PP^2$, which meet the toric boundary at $3d$ points with multiplicity $2$, i.e. the curve is tangent to the toric boundary at every points where they meet.
    In other words, we take the tangency profiles to be $\bfa=(-2,0,2)^d$ and $\bfb=(0,-2,2)^d$, where the $d$ exponent is a shorthand notation for two vectors in $\mathbb Z^{3d}$. The expected dimension of the space of such curves can be computed to be $3d+g-1$. 

    \medskip
    
    Then, fixing a configuration of $3d+g-1$ generic points in $\mathbb P^2$, we can ask what is the number of curves in $\PP^2$ as before passing through the fixed points. Formulating the question in the GW language leads to consider the log GW invariant $\gen{\pt^{3d+g-1}}_{g,(\bfa,\bfb)}$, which we denote by $N_{2,d,g}$.

    \medskip

    In this  example, $\gcd(\bfa,\bfb)=2$, so that the set of correlators is simply $\mu_2=\{\pm 1\}$
     and the only primitive correlator is $-1.$ The log GW  invariant can  be written as $N_{2,d,g}=N_{2,d,g}^++N_{2,d,g}^-$, where  the $+,-$ at the exponent denote the correlators. 

     \medskip
     
     
    In formulae, denoting by $N_{d,g}$ the number of degree $d$ and genus $g$ curves passing through $3d+g-1$ points in $\PP^2$, the MCF asserts that we have the relation
    $$N_{2,d,g}^+=N_{2,d,g}^-+2^{3g-3+(3d+3d-1+g)}N_{d,g},$$
    thus completely determining $N_{2,d,g}^\pm$ from $N_{2,d,g}$ and $N_{d,g}$.

    In the case of point insertions, this relation can also be proven using the tropical correspondence Theorem \cite{mikhalkin2005enumerative}, sorting tropical curves according to the g.c.d. of their edge weights, in the spirit of \cite{blomme2025short}.

    The approach of the current paper is much more general as it directly works with all possible insertions.
\end{expl}

\begin{rem}

Generalizing what happened in the example above we know that: the sum of correlated invariants over all possible correlators is the usual log GW invariant; each correlated invariant is expressed, thanks to the MCF, as a linear combination of  primitive correlated invariants. We thus get a triangular system of relations that we can invert to compute the correlated invariants from the standard log GW.

However, it is non-obvious from such description that the correlated invariants should be integers, even knowing that the uncorrelated ones are. In the case of points insertions, it is possible to adapt the tropical correspondence theorem from \cite{mikhalkin2005enumerative}, counting curves with a correlated multiplicity to compute the correlated invariants. The multiplicity of a curve $\Gamma$ with g.c.d. of its ends (alias unbounded edges) weights equal to $r$ and g.c.d. of all edge weights equal to $\delta$ can be computed to be $m_\Gamma\cdot T_{r/\delta}$ (see \cite[Section~6]{blomme2024correlated} for a similar computation in terms of Floor diagrams in the case of $\mathbb P^1$-bundles on elliptic curves). One can then check that for any toric surface and any tangency profile,  
\begin{equation}\label{eq:integrality}
 \gen{\pt^{3d+g-1}}_{g,(\bfa,\bfb)}^{\theta_{\mathrm{prim}}}\in\mathbb Z.
\end{equation}
The primitive correlated invariants should be thought as the analogue in this geometric setting of the so-called \emph{BPS count} in the sense of \cite[Section~6]{gross2010tropical}.
\end{rem}



\subsection{Ideas of proof for the MCF}\label{sec:ideaproof}
The idea to get a formula is similar to the one already used by the first two authors in \cite{blommecarocci2025DR}, modulo refining the strategy to get a formula for the logarithmic class.
\medskip

First, we consider a certain log modification of  the moduli space $S\oner$ parametrising log smooth curves with two (logarithmic) $r$-roots $\R_1,\R_2$ of our line bundle $\L_1^r,\L_2^r;$ we get log $r$-torsion line bundles $\mathcal T_1=\R_1\otimes \L_1^{-1}$ and $\mathcal T_2=\R_2\otimes \L_2^{-1}$. 


\medskip

The logarithmic Weil pairing defines a locally constant map from (a certain base change) the double spin space $S\oner$ with values in $\mu_{r},$ extending the usual Weil pairing on the locus of smooth curves.  We thus see that the latter decomposes in open and closed components
$S^\theta\oner$ on which the value $W_r(\T_1,\T_2)$ is constant.

\medskip

One can then apply the universal  DR-formula from \cite{bae2023pixton} for the line bundles  $\R_1,\R_2$ on each component $S^\theta\oner$. Since a (log) line bundle is trivial if an only if one of its root is, this DDR locus on the spin space push-forward to the cycle we are interested in. 
Up to working with an appropriate model of the double spin space (see Section~\ref{sec-LDDR}), and an appropriate twist for the line bundles, the classes one obtains applying the uni-DR formula are representative for the logarithmic spin DR classes  \cite{holmes2021extending,holmes2022intersection,holmes2025logDR} for the two logarithmic $r$-roots. In particular,  by \cite{holmes2022intersection}, we obtain a tautological formula for the logDDR class taking their product.

\medskip

Pushing forward from $S^\theta\oner$ to $S$ (or to be more precise to a suitable logarithmic blow-up) and keeping track of the degree on the various strata, we get a formula   for the (log) correlated classes $\DDR^\theta(\L_1^r,\L_2^r)$. 

\medskip


The MCF may then be proved looking at the explicit expression stratum by stratum, and follows from a degree computation.

\medskip


It is instructive to notice that in \cite{blomme2025multiple}, the MCF for the correlated DR-cycle boils down to combinatorics of torsion subgroups of an elliptic curve $E$.  In the present paper, the MCF for the correlated DDR-cycle reduces to counting pairs of vectors with a given pairing in a symplectic space over $\ZZ_r$, i.e. $\ZZ_r^{2g}$ with a non-degenerate skew-symmetric form. Both are actually elementary algebraic considerations completely independent of the GW-theory.

\subsection{Future directions}\label{sec:futuredir}






A broader setting unifying the definition of correlated logarithmic GW invariants of \cite{blomme2024correlated,blommecarocci2025DR} and the present one should come by considering log GW invariants of a smooth projective
 surface relative to an anticanonical divisor, which, by adjunction formula,  is a curve whose connected components have arithmetic genus $1.$ 

Indeed, in the case of a $\PP^1$-bundle over an elliptic curve, the union of the $0$ and $\infty$-sections is an anticanonical divisor, as so is the toric boundary of a toric surface. 



\subsubsection{Generalized Takahashi conjecture}\label{sec:generalizedtakahasi}

Let $X$ be a smooth del Pezzo surface and let $D\subseteq  X$ a smooth curve representing the anticanonical class. 
 Let further $\bfa=(a_1,\dots, a_n)$  be a tuple of non-negative integers with $\sum a_i=\beta\cdot[D]=-K_X\cdot\beta$.

We consider the moduli space of log-stable maps $\M_{g,n}(X|D,\beta,\bfa)$ parametrizing (at least in the nice locus) maps $f\colon(C,p_1,\dots,p_n)\to X$ such that the image curve meets $D$ at $f(p_i)$ with multiplicity $a_i$. Given such a map, in $\Pic^{\beta\cdot[D]}(D)$, we have
$$\O_D(\sum a_i f(p_i)) \cong \O_{X}(\beta)|_D,$$
where $\O_X(\beta)$ has a section cutting out the image curve; the notation emphasizes that it only depends on the curve class $\beta$ realized by the image curve. If we assume that the $a_i$ have a non-trivial common divisor $r$ that also divides $\beta$, for each $f$ as before we can consider the following $r$-torsion element in $\Pic^0(D)$:
\[\theta=\O_D(\sum \frac{a_i}{r} f(p_i))\otimes\O_{X}(-\beta/r)|_D.\]
The latter is locally constant, and thus the moduli space  $\M_{g,n}(X|D,\beta,\bfa)$ is a union of open and closed substacks where the \emph{correlator} $\theta$ is constant. 
For $\alpha\in H^*(\Mgnbar,\QQ)$ and $\underline{\gamma}$ a geometric insertion, we can thus define in the obvious way \emph{correlated invariants }
$\gen{\alpha;\underline{\gamma}}^\theta_{g,d,\bfa}$ for $(X,D),$ keeping track of the contribution to the invariants of the $\theta-$constant component. 

\begin{conj}\label{conj:generalizedTakahashi}
The following multiple cover formula holds:
\begin{equation}\label{eq:MCFTakahashi}
\langle \alpha ; \underline{\gamma} \rangle_{g,d,\bfa}^\theta=\sum_{k|\beta,\bfa,\theta}k^{3g-3+n-\deg\alpha} \langle \alpha ; \underline{\gamma} \rangle_{g,\beta/k,\bfa/k}^{\theta_{\rm{prim}}} .
\end{equation} 
Furthermore, for point insertions we have $\langle  \underline{[pt]} \rangle_{g,\beta,\bfa}^{\theta_{\rm{prim}}}$ are integers.
\end{conj}
For certain geometries the formula might need to be corrected by a sign depending  on $\beta\cdot K_X$ and on the order of the correlator.

 The case $X=\mathbb P^2$, $D$ smooth elliptic curve, $g=0$, and unique point of maximal tangency is the one considered in the standard Takahashi conjecture. 
 
\medskip

We in fact believe that it should be possible to define the correlated invariants for any nodal curve representing the anticanonical class, however it would take more care to do that in complete generality. 

\medskip

The present paper does so in the special  case of a toric surface with its full toric boundary and Theorem~\ref{thm:toriccorrelated} gives further evidence for our conjecture.

\medskip

More generally, we expect that there is a correlated refinement for log GW invariants of Looijenga pairs (essentially toric surfaces blown-up along some points in the interior of the toric boundary) and multiple cover formulae satisfied by such invariants. 

\medskip

 We expect that a proof of the conjecture above should be obtained applying a correlated refinement of the degeneration formula, in the spirit of the one proved and applied in \cite{blomme2024correlated}, once the case of Looijenga pairs has been understood.
 
   In particular, this would lead to a new proof of the Takahashi conjecture \cite{takahashi1996curves,takahashi2001log,bousseau2019takahashi} which does not require DT/GW correspondence results.


\subsubsection{Further future applications of correlated invariants}

 Finally, we believe that the study of correlated invariants should also be interesting and relevant in the following geometric situations

\begin{itemize}[leftmargin=0.4cm]
    \item There should be a reduced degeneration formula  for type III degeneration of an abelian surface, whose vertex contribution are precisely correlated GW invariants of toric surfaces as defined here. Then yet another proof of the MCF for abelian surfaces in the style of \cite{blomme2025short} should be deduced by the correlated MCF for toric invariants.
    \item More importantly, logarithmic invariants of toric surfaces are one of the two ingredients needed to tackle via reduced degeneration formula the MCF for K3 surfaces. The second ingredient will be precisely the study of correlated invariants and their multiple cover formulas for Looijenga pairs. The study of the correlated invariants in the setting of Looijenga pairs is the topic of the next paper in our series of works.
    
    \item The definition of correlated invariants naturally extends to higher dimensional toric varieties, via the different (rational) projections onto toric surfaces. 
    Unlike the case of dimension 2, where a toric surface relative to the toric boundary may sometimes be obtained degenerating a smooth anticanonical divisor (e.g. $\PP^2$ with a smooth cubic degenerating to the union of three lines), setting where we also expect as explained above a multiple cover formula phenomena, the higher dimensional situation does not admit a similar interpretation (for instance, $\PP^3$ with a family of quartics degenerating to the union of four planes).
    
     However, the existence of these refinements may hint toward multiple cover phenomena  for higher dimensional toric varieties which in turn might be useful to approach the  MCF for abelian threefolds conjectured in \cite{bryan2018curve}.
\end{itemize}

\subsection{Plan of the paper}

In Section \ref{sec-alg-recoll} we recall some basic algebraic  properties of functions with values in the group algebra of $\RR/\ZZ$, we introduce  algebraic multiple cover formulas for such functions and prove that certain functions defined from counts of pairs of vectors in a finite symplectic space do satisfy the algebraic MCF. 

In Section~\ref{sec:loggeometry} we (briefly) recall some of the necessary background in logarithmic geometry. In particular, we recall some properties of the logarithmic and tropical Picard group and some properties about moduli spaces of roots of line bundles in tropical and logarithmic geometry.

In Section~\ref{sec-LDDR}, we collect the definition of DR and logDR loci and recall the criteria under which the class obtained via the universal DR formula is a representative for the logarithmic class. Furthermore, we introduce the correlated refinement of DDR loci via logarithmic Weil pairing and via logarithmic reciprocity law.

In Section~\ref{sec-MCF-DDR}, we proceed as anticipated in Section~\ref{sec:ideaproof} to give a tautological formula for the log DDR cycle over a double spin moduli space and then study the degree of the push-forward  to $\Mgnbar$ on the various components and strata of this spin space to prove the multiple cover formula \eqref{eq:MCF1}.


Finally, in Section~\ref{sec:toric}, we prove the multiple cover formula for correlated log-GW invariants of toric surfaces relative to the toric boundary  \eqref{eq:MCF2}, adapting the arguments from \cite{ranganathan2024logarithmic}.

\medskip

\textit{Acknowledgements.} 
The authors would like to thank Sam Johnston, Thibault Poiret, Ilia Itenberg, Davesh Maulik, David Holmes, Sam Molcho, Rahul Pandharipande, and Dhruv Ranganathan for useful conversations about the topic of this paper. Special thanks go to Dhruv Ranganathan for several useful comments on a earlier draft of this paper.
Part of this work was conducted during visits of the authors at ETH Zurich and at the SwissMap research station in Les Diablerets; we thank both institutions for the hospitality and the stimulating research environment.
T.B. acknowledges the support of SNF grant 204125. F.C. is supported by the MIUR Excellence Department Project Mat-
Mod@TOV, CUP E83C23000330006, awarded to the Department of Mathematics,
University of
Rome Tor Vergata, and also acknowledges the support of the PRIN Project
``Moduli spaces and
birational geometry'' 2022L34E7W. A.U.K. is supported by SNF grant P500PT-225465.

\section{Algebraic recollections}
\label{sec-alg-recoll}

\subsection{Functions with values in the group algebra and refinement}

\subsubsection{$\NN$-module}

We start recalling some results from \cite[Section~2]{blomme2025multiple} about functions with values in a group algebra, useful to formulate and prove multiple cover formulas (MCF).
\begin{defi}
    A $\NN$-module is a set $X$ with a free action of $\NN$ such that for each $x\in X$  there is only a finite number of divisors, i.e. a finite number of $\widetilde{x}\in X$ such that $n\widetilde{x}=x$ for some $n.$ Furthermore $X$ is equipped with a \textit{norm} map $|\cdot|\colon X\to\NN$ compatible with the action. A morphism of $\NN$-modules is a map compatible with the action and preserving the norm.
\end{defi}

\begin{expl}
We have the following elementary examples.
\begin{itemize}[leftmargin=0.4cm]
    \item The trivial $\NN$-module $\NN$ with multiplication by itself and norm is the identity function.
    \item The $d$-dilated module $\NN$ with norm given by $d\cdot\mathrm{id}\colon\NN\to\NN$.
    \item The set $\DDD$ of families of $n$ non-zero vectors  $v_i\in\ZZ^2$ such that $\sum v_i=\underline{0}$ is a $\NN$-module whose norm is given by the gcd of all the entries of the $v_i$. 
    
    \item The set of parametrized tropical curves in $\RR^2$ with norm given by the gcd of the ends is a $\NN$-module with a morphism to the $\NN$-module $\DDD$ of degrees.
\end{itemize}
\end{expl}

\subsubsection{Group algebra}

Let $\G$ be the group algebra of $\RR/\ZZ$. We denote by $T_r$ the average of $r$-torsion elements, which is an idempotent element in $\G$:
$$T_r = \frac{1}{r}\sum_{r\theta\equiv 0}(\theta).$$
The subalgebra generated by $(T_k)_{k|r}$ is denoted by $\G_r$. The coefficient in front of $T_r$ is called \textit{primitive coefficient}.

\medskip

We consider functions defined on a $\NN$-module $X$ with values in the group algebra; we will write $\G^{\NN}$ for such functions in the case of the trivial $\NN$-module and $\G^X$ for a general $\NN$-module. The module $\G^{\NN}$ is endowed with two products: the usual term-by-term product but also the convolution product
$$(f\ast g)(r) = \sum_{kl=r}f(k)g(l).$$
The vector space $\G^X$ has a $\G^{\NN}$-module structure given by the following action:
$$f\ast F(x) = \sum_{k|x}f(k)F(x/k).$$
This operation is well-defined since by assumption any element in $X$ has a finite number of divisors.

\begin{defi}
    A function $f\colon \NN\to\G$ is of \emph{diagonal type} if $f(r)\in\G_r$. A function $F\colon X\to\G$ defined on a $\NN$-module $X$ is of \emph{diagonal type} if $F(x)\in\G_{|x|}$.
\end{defi}

\begin{rem}
    Notice that if $f\colon\mathbb N\to\G$ is of diagonal type then the coefficient $c_{\theta}$ of  a $r$ torsion element $\theta$ in $f(r)$ only depends on the order of $\theta.$ This is easily seen comparing the two expressions $\sum_{k|r}C_k\cdot T_k=f(r)=\sum_{r\theta=0} c_{\theta}\cdot (\theta).$
\end{rem}
\subsubsection{Multiple cover formulas}

Among diagonal type functions, we have specific functions satisfying a functional equation asserting that it is possible to recover the whole function from its primitive coefficients. We denote by $\epsilon_{\alpha}(n)=n^\alpha$ he $\alpha$-power function.

\begin{defi}
    A function $f\colon \NN\to\G$ satisfies the $\alpha$-MCF if $f=\epsilon_\alpha\ast(\Prim(f)T)$:
    $$f(r) = \sum_{kl=r}k^\alpha\Prim_l(f(l))T_l.$$
    A function $F\colon X\to\G$ satisfies the $\alpha$-MCF if $F=\epsilon_\alpha\ast(\Prim(F)T)$:
    $$F(x) = \sum_{k|x}k^\alpha\Prim_{|x/k|}(F(x/k))T_{|x/k|}.$$
\end{defi}

\begin{prop}\cite[Section~2]{blomme2025multiple}
The functions satisfying multiple cover formulas satisfy the following properties.
    \begin{enumerate}
        \item If $f$ satisfies the $\alpha$-MCF, then $\epsilon_r f$ satisfies the $(\alpha+r)$-MCF.
        \item If $f,g$ satisfy the $0$-MCF, then $fg$ satisfies the $0$-MCF.
        \item A diagonal type function $F\colon X\to\G$ satisfies the $\alpha$-MCF if and only if its restriction to each primitive orbit does. (i.e. the orbit of a primitive element)
        \item A diagonal function $f\colon\NN\to\G$ with norm $d\cdot\mathrm{id}$ satisfies the $\alpha$-MCF if and only if it is invariant by multiplication by $T_d$ and $\m{d}\circ f$ satisfies the usual $\alpha$-MCF (as function on the trivial $\NN$-module). We recall that $\m{d}\colon (\theta)\mapsto (d\theta)$ is an endomorphism of the group algebra.
    \end{enumerate}
\end{prop}

\subsection{Refined count of vector pairs}

Let us denote $\ZZ_r=\ZZ/r\ZZ$. We consider $V=\ZZ_r^{2g}$ a free $\ZZ_r$-module endowed with a \emph{non-degenerate skew-symmetric form} $\omega$. We consider the $\omega$-refined count of  pairs of vectors $(x,y)\in V\times V:$ 
$$\bfF_{2g}(r) = \sum_{x,y\in V}(\omega(x,y)) \in\QQ[\ZZ_r]\subset\QQ[\RR/\ZZ],$$
where $(\theta)$ denotes the group algebra element associated to $\theta\in\ZZ_r$.

\begin{prop}
    We have the explicit expression
    $$\bfF_{2g}(r) = r^{2g}\sum_{k|r}J_{2g}(k)T_k,$$
    where $J_{2g}$ is the $2g$-Jordan function. In particular, $\bfF_{2g}$ satisfies the $2g$-MCF.
\end{prop}

\begin{rem}
    The $m$-Jordan function counts the number of order $r$ elements in $\ZZ_r^m$. It is given by
    $$J_m(r) = r^m\prod_{p|r}\left(1-\frac{1}{p^m}\right),$$
    where the product is over prime numbers dividing $r$. In particular $J_1$ is also known as Euler function $\varphi$.
\end{rem}

\begin{proof}
Let $x\in V$ be of order $l$. In particular, we can write $x=\frac{r}{l}\widetilde{x}$ where $\widetilde{x}$ is some primitive element in $V$. The map $y\mapsto\omega(x,y)$ can be seen as the composition between $y\mapsto\omega(\widetilde{x},y)$ and the multiplication by $\frac{r}{l}$. By non-degeneracy, $y\mapsto\omega(\widetilde{x},y)$ is surjective, and the image of $y\mapsto\omega(x,y)$ is thus exactly $\ZZ_l\subset\ZZ_r$. We conclude that for every $x$ of order $l$, we have
$$\sum_{y\in V}(\omega(x,y)) = r^{2g}T_l.$$
Summing over $l$, given that there are exactly $J_{2g}(l)$ elements of order $l$ in $V\simeq\ZZ_r^{2g}$, we get
$$\sum_{x,y\in V}(\omega(x,y)) = \sum_{x\in V}r^{2g}T_l = r^{2g}\sum_{l|r}J_{2g}(l)T_l.$$
Concerning the $2g$-MCF, dividing by $r^{2g}$, we may restrict to showing that $r\mapsto\sum_{k|r}J_{2g}(k)T_k$ satisfies the $0$-MCF, which is obvious from its expression.
\end{proof}

We now consider the free $\ZZ_r$-module
$$V=(L\oplus L^*)\oplus W,$$
where $W$ is a symplectic lattice, $L$ a lattice, $L^*$ its dual, and $L\oplus L^*$ is endowed with its canonical symplectic form coming from the pairing $\langle\cdot,\cdot\rangle$ between $L$ and $L^*$. 
We thus have a symplectic form on $V,$ which we also denote by $\omega$
defined by
$$\omega(l_1+l_1^*+w_1,l_2+l_2^* +w_2)=\langle l_1,l_2^*\rangle \cdot \langle l_1^*,l_2\rangle \cdot\omega(w_1,w_2).$$
Let $b$ be the rank of $L$ and $2g-2b$ be the rank of $W$. Let $\lambda_1,\lambda_2\in L$ be two fixed elements of respective orders $\omega_1,\omega_2$.
$$\bfG_{b,2g}(r,\lambda_1,\lambda_2) = \sum_{\substack{x\in\lambda_1+L^*+W \\ y\in\lambda_2+L^*+W}}(\omega(x,y)).$$

\begin{expl}\label{ex:troPic}
    We have in mind the following example: $C$ is a nodal curve with normalization $\widetilde{C}$ and dual graph $\Gamma$ with $h_1(\Gamma,\mathbb Z)=b;$ $V=\LogPic^0_C[r],$ $W=\Pic^0_{\widetilde{C}}[r]$ with the symplectic structure induced by the classical Weil pairing for abelian varieties under the self duality isomorphism; $L^*=H^1(\Gamma,\mu_r)$ and $L=\mathrm{TroPic}^0_C[r].$ It is proved in \cite[Proposition~2.13]{blommecarocci2025DR} that the logarithmic Weil pairing is non-degenerate, it induces the standard pairing between $L$ and $L^*$.
\end{expl}

\begin{prop}\label{prop:expression_G}
    The value of $\bfG_{b,2g}(r,\omega_1,\omega_2)$ only depends on $\lambda_1,\lambda_2$ through their orders $\omega_1$ and $\omega_2$. Furthermore, we have the explicit formula:
    $$\bfG_{b,2g}(r,\omega_1,\omega_2) = r^{2g}T_{\omega_1}T_{\omega_2}\sum_{k|r}J_{2g-2b}(k)T_k.$$
    Extending the function by $0$ when $r$ is not divisible by $\omega_1$ and $\omega_2$, it satisfies the $2g$-MCF.
\end{prop}

\begin{proof}
    The sum factors as follows:
    $$\sum_{\substack{x\in\lambda_1+L^*+W \\ y\in\lambda_2+L^*+W}}(\omega(x,y)) = \left(\sum_{\mu\in L^*}(\gen{\lambda_1,\mu})\right)\left(\sum_{\mu\in L^*}(\gen{\mu,\lambda_2})\right)\left(\sum_{x,y\in W}(\omega(x,y)\right).$$
    The third sum has already been computed to be $r^{2g-2b}\sum_{k|r}J_{2g-2b}(k)T_k$. For the first (resp. second) sum, one just needs to compute the image of $\mu\mapsto\gen{\lambda_1,\mu}$, which is $\ZZ_{\omega_1}\subset\ZZ_r$ where $\omega_1$ is the order of $\lambda_1$. We thus get $r^bT_{\omega_1}$. In the end, we get
    $$r^bT_{\omega_1}\cdot r^bT_{\omega_2} \cdot r^{2g-2b}\sum_{k|r}J_{2g}(k)T_k,$$
    yielding the announced expression. The $2g$-MCF follows from the following lemma, applied to $\delta=\mathrm{lcm}(\omega_1,\omega_2)$.
\end{proof}

\begin{lem}
    Let $\bfG\colon\NN\to\NN$ be a function satisfying the $0$-MCF and $\delta\geqslant 1$. We set
    $$\bfG_\delta(r) = \left\{ \begin{array}{l}
        T_\delta\cdot\bfG(r) \text{ if }\delta|r,  \\
        0 \text{ else.} 
    \end{array}\right.$$
    Then $\bfG_\delta$ satisfies the $0$-MCF.
\end{lem}

\begin{proof}
    We consider the $\delta$-dilated $\NN$-module $\NN_\delta=\NN$ with $|1|=\delta$ and the diagonal type function $\widetilde{\bfG}\colon r\mapsto T_\delta\cdot\bfG(\delta r)$. Let $\iota\colon\NN_\delta\hookrightarrow\NN$ be the multiplication by $\delta$, where $\NN$ is the trivial $\NN$-module, where $|1|=1$, so that $\iota$ is a morphism of $\NN$-modules. Notice that $\bfG_\delta$ is the push-forward of $\widetilde{\bfG}$ since the push-forward just extends by $0$ the value for non-multiples of $\delta$:
    $$\bfG_\delta(r)=\iota_*\widetilde{\bfG}(r).$$
    Therefore, to prove the MCF, by \cite[Lemma 2.22]{blomme2025multiple}, it suffices to prove the MCF for $\widetilde{\bfG}$ on $\NN_\delta$. By \cite[Lemma 2.19]{blomme2025multiple}, we can apply $\m{\delta}$ to check the MCF with the trivial norm, for the function $r\mapsto \m{\delta}\bfG(\delta r)$, which is ensured by \cite[Lemma 2.23]{blomme2025multiple}.
\end{proof}

\section{Log-geometry background}\label{sec:loggeometry}
We assume the reader to be already familiar with the standard definitions and results about logarithmic schemes/logarithmic algebraic stacks. In particular, we assume familiarity with the notion of \emph{tropicalization} of a logarithmic algebraic scheme/stack, intended both as a cone stack and as an Artin fan (see \cite{cavalieri2020moduli}  and references). 
We refer for example to \cite{ogus,olsson2003logarithmic,KatoF,cavalieri2020moduli,abramovich2014stable} for details. An account of the logarithmic background sufficient for the scopes of this paper can be found for example in \cite{molchowiselog,marcuswiselog,holmes2025logDR}. 

We furthermore assume some familiarity with the notion of tropical and logarithmic line bundles in the sense of \cite{molchowiselog}.

In order to fix notation we recall some of the previously mentioned background material.

\subsection{Sheaves of groups on logarithmic schemes}
We denote by $\mathrm{LogSch^{\mathrm{fs}}}$ the category of fine and saturated (fs) logarithmic schemes. The logarithmic schemes we consider are all fine and saturated unless otherwise stated and thus we often do not repeat fs explicitly.

 The category of fs-logarithmic schemes $\mathrm{LogSch}^{\mathrm{fs}}$ is fibered over $\mathrm{Sch};$ the projection forgets the logarithmic structure. 
Logarithmic structures on algebraic spaces and algebraic stacks are defined as  sections of this projection; see \cite[Section~5]{olsson2003logarithmic}.

\subsubsection{}  On a logarithmic scheme $(S,M_S)$ we have 
 an exact sequence of sheaves of abelian groups:
\begin{equation}\label{eq:fundamentallogses}
    0\to\mathcal O_S^\ast\to M_S^{\mathrm{gp}}\to\bar{M}_S^{\mathrm{gp}}\to 0 .
\end{equation}
For each local section $\alpha\in \bar M_S^{\mathrm{gp}},$ the fiber of $ M_S^{\mathrm{gp}}\to\bar{M}_S^{\mathrm{gp}}$ is an $\mathcal O_S^\ast$-torsor; we denote by $\mathcal O_S(\alpha)$ the associated line bundle.  


 
\subsubsection{}  We denote by 
\[\Gmlog,\Gmtrop\colon\operatorname{LogSch^{fs}}\to\mathrm{AbGrp},\]

the abelian group objects  whose value at $(S, M_S)$ are respectively  
\[H^0(S, M_S^{\gp}) \text{ and } H^0(S, \overline{M}_S^{\gp}).\]

As shown for example in \cite[Section~2]{molchowiselog}, these are \emph{not} representable by an algebraic stack with logarithmic structure, but we will see below that they admit a log \'etale map from a scheme (algebraic stack respectively) with log structure.
\begin{Notation}
    From now on we simply write $S$ rather than $(S,M_S)$ for a logarithmic scheme unless we need to remark the choice of logarithmic structure. 
\end{Notation}


\subsection{Artin fans, tropicalization and logarithmic modifications}

To any finitely generated, integral and saturated monoid $P$ is associated a dual rational polyhedral cone $\sigma$. We denote by $U_{\sigma}$ the corresponding affine toric variety and by 

\[\mathcal A_{\sigma}:=[U_{\sigma}/T_{\sigma}],\]

the corresponding (so called) \emph{Artin cone}; see \cite{abramovich2014comparison, cavalieri2020moduli} and references therein. 

For $S$ an fs log scheme, choosing charts in the sense of \cite[II.2]{ogus} we get, \'etale locally, morphisms to Artin cones; these glue together to a morphism to an algebraic stack with logarithmic structure, in fact an Artin fan 
\[t\colon S\to\mathcal A_{\Sigma_S},\] 
which we call \emph{tropicalization} of $S$. 

The stack $\mathcal A_{\Sigma_S}$ is an \emph{Artin fan} in the sense of \cite{abramovich2014comparison, cavalieri2020moduli,olsson2003logarithmic}; namely an algebraic stack with logarithmic structure which an \'etale cover given by Artin cones. 
By \cite{cavalieri2020moduli}, the category of Artin fans is equivalent to the category of \emph{cone stacks}, these are a generalization  of cone complexes where multiple face maps between two cones are allowed and self maps other than the identity are allowed.

Because of this result, in what follows we  \emph{do not distinguish}
the algebraic stack $\mathcal A_{\Sigma_S}$ and the combinatorial object
$\Sigma_S$ and \emph{refer to both as the tropicalization } of $S$ and often write 
\[S\to\Sigma\]
for a logarithmic map from $S$ to the Artin fan associated with the cone stack $\Sigma.$
\begin{rem}
We defined the tropicalization starting from charts of the logarithmic scheme as described above. We refer to \cite[Remark 5]{holmes2025logDR} and references therein for a discussion about this dependence.
\end{rem}

\subsubsection{} Given a cone stack $\Sigma$ we say that $\widetilde{\Sigma}\to\Sigma$ is a subdivision if for each cone $\sigma\to\Sigma$ we have that $\widetilde{\Sigma}\times_{\Sigma}\sigma\to\sigma$ is a subdivision in the classical sense of toric geometry (see for example \cite{cox2024toric}).

We say that a morphism of logarithmic schemes $\widetilde S\xrightarrow{p} S$ is a \emph{subdivision} or \emph{logarithmic modification} if 
$\widetilde S=S\times_{\Sigma }\widetilde{\Sigma}.$ We also say that $\widetilde S\xrightarrow{p} S$ is a logarithmic modification if $\widetilde S$ is obtained  from $S$ by a root construction, see for example \cite[Section~2]{holmes2023root}. Locally at the level of cones, this correspond to a refinement of the lattice structure.

\medskip
A logarithmic modification $p$ has the following properties (see \cite[Lemma~13]{holmes2025logDR}) and references therein:
\begin{itemize}[leftmargin=0.4cm]
    \item $p$ is proper and representable (representable by a Deligne-Mumford stack in the case of roots)  and it is a \emph{monomorphism of logarithmic algebraic stacks};
    \item $p$ is \emph{logarithmically \'etale} (see for example \cite{ogus} for definition and characterization of log \'etale and log smooth morphisms);
    \item if $S$ is logarithmically smooth and thus $S\to \Sigma_S$ is smooth in the usual sense, then $p$ is birational.
\end{itemize}
    \begin{rem}
Though not representable, $\Gmlog$ and $\Gmtrop$ admit subdivisions which are representable by logarithmic algebraic stacks with log structure. Explicitly, $\mathbb P^1\to\Gmlog$ and $[\PP^1/\mathbb G_m]\to\Gmtrop$ are such subdivisions.
(See \cite[Section~2]{molchowiselog}).
    \end{rem}

\subsection{Logarithmic Chow ring for a log smooth algebraic stacks}

Given $S$ a logarithmic smooth scheme or algebraic stack, one defines as in \cite{holmes2022intersection} (see also \cite{molcho2023hodge,pandharipande2025logarithmic} and \cite[Section~1.3,1.4,2.5]{holmes2025logDR}) the log-Chow ring:
\[\mathrm{logCH}(S)=\varinjlim_{\widetilde S\to S} \mathrm{CH}^*_{\mathrm{op}}(\widetilde S,\mathbb Q).\]
There exists a natural injective ring homomorphism $ \mathrm{CH}^*_{\mathrm{op}}( S)\hookrightarrow \mathrm{logCH}(S)$ as well as a \emph{group homomorphism}
$\mathrm{logCH}(S)\to \mathrm{CH}^*_{\mathrm{op}}( S)$ defined via proper push-forward.

\medskip

A \emph{piece-wise polynomial} $f\in\mathrm{PP(\Sigma)}$ on a cone stack $\Sigma$ is a continuous function $f\colon |\Sigma|\to\mathbb R$ (where $|\Sigma|$ denotes the support of $\Sigma$) such that there exists a subdivision $\widetilde{\Sigma}\to\Sigma$ with $f$ polynomial on each cone of $\widetilde \Sigma.$ If no subdivision is necessary we say that $f$ is strict and write $f\in \mathrm{sPP(\Sigma)}.$

As explained in \cite{holmes2022intersection,molcho2023hodge}, if $\widetilde\Sigma$ is the tropicalization of a log smooth scheme $\widetilde S$, there is a ring homomorphism
\[\Phi\colon \mathrm{sPP(\widetilde \Sigma)}\to \mathrm{CH}^*_{\mathrm{op}}(\widetilde S,\mathbb Q),\]
inducing a ring homomorphism
\[\Phi\colon \mathrm{PP(\Sigma)}\to \mathrm{logCH}(S).\]
\begin{rem}
\label{rem:toricpp}
When $S=X_\Sigma$ is a smooth toric variety with fan $\Sigma$, the ring homomorphism $\Phi:\mathrm{sPP(\Sigma)}\to \mathrm{CH}^*_{\mathrm{op}}(S,\mathbb Q)$ is surjective. A proof of this fact is outlined in \cite{katz2008piecewise}[following Theorem 1.4].
\end{rem}
  \subsection{Log smooth curves and their tropicalization}
  \label{subsec:logcurve}

A log curve over $S$ is a proper, integral and log-smooth morphism $\pi\colon C\to S$ with connected, reduced one-dimensional (geometric) fibers. We refer to \cite{KatoF} for a complete characterization of how the logarithmic morphism $C\xrightarrow{\pi} S$ looks like \'etale locally.

 \medskip
 

If there are no markings with logarithmic structure we will say that $C\to S$ is \emph{vertical} (see \cite[Definition~1.2.8]{ogus}). Given any log curve $C\to S$ it is possible to consider the associated \emph{vertical log curve} has the same underlying scheme as $C$ but different log structure $M_C^V\subseteq M_C$ as in \cite[Definition~1.2.8]{ogus}.

\medskip

\paragraph{\bf{Warning}} 
We will often be interested in smooth logarithmic curves $C\to S$ with markings $p_1,\dots,p_n\colon S\to C$ carrying non trivial logarithmic structure. Nonetheless, when we talk about the logarithmic Picard group $\LogPic_{ C/S}$ we always consider the curve with its \emph{vertical log structure}.
We will not remark this further in what follows.

 \medskip

Given $C\to S$ we denote by $\Gamma\to S$ its tropicalization. This can be thought of as follows: for each  geometric point $s\in S$ a metrization of $\Gamma_s$ (the dual complex of the curve underlying  $C_s$) with values in $\bar{M}_{S,s}$; for each geometric specialization $t\rightsquigarrow  s$ an edge contraction $\Gamma_s\to\Gamma_t$ where the contracted edges correspond to the  nodes of $C_s$ smoothed in $C_t$ and  are precisely the edges whose length $\ell(e)\in \bar{M}_{S,s}$ goes to zero under the morphism $\bar{M}_{S,s}\to \bar{M}_{S,t}.$

 
\medskip

 Let $\Gamma$ be a tropical curve metrized in a monoid $\bar{M}_S$ (e.g. if we are in the constant degeneration case). In this case the curve can be endowed with the \emph{tropical topology} (as defined in \cite[Section~3]{molchowiselog}) whose fundamental opens are the stars of the vertices and two such opens intersect  in a disjoint union of open intervals.

In this topology: 
     \begin{itemize}[leftmargin=0.4cm]
 \item $\mathcal{PL}$ denotes the sheaf of strict piecewise linear functions whose sections over an open set $U$ are the data of: a function $\alpha$ on the vertices of $U$ with values in $\bar{M}_S^{\gp}$  a collection $(w_e)$ of slopes for edges in $U$ such that $\frac{\alpha(v)-\alpha(w)}{\ell(e)}=w_e$ if $v,w,e\in U$,
  \item  $\bar{M}_{S}^\gp$ denotes the constant sheaf on $\Gamma$,
        \item $\B\L\hookrightarrow\mathcal P\mathcal L$ denotes the subsheaf of 
\emph{balanced} strict piece-wise linear function, i.e the sum of the outgoing slopes at each vertex is zero.
    \end{itemize}

These sheaves fit together (see \cite[Section~3.4]{molchowiselog}) in the following commutative diagram:

\begin{equation}\label{eq:tropicalsheaves}
    \begin{tikzcd}
& & 0\ar[d] & 0\ar[d]\\
0\ar[r] & \bar{M}_{S,s}^{\gp}\ar[d,"="]\ar[r]&\B\L\ar[d]\ar[r] &\mathcal H\ar[d]\ar[r] & 0\\
0\ar[r] & \bar{M}_{S,s}^{\gp}\ar[r]&\mathcal P\mathcal L\ar[d]\ar[r] &\mathcal E\ar[d]\ar[r] & 0,\\
& & \mathcal V\ar[r,"="] \ar[d]& \mathcal V\ar[d]\\
& & 0 & 0
\end{tikzcd}
\end{equation}
where $\mathcal E$ and $\mathcal V$ are the constant groups sheaves freely generated by the edges and the vertices respectively and the map between them is the homology boundary map.
 These sequences induce long exact sequence in cohomology, meaning Čech cohomology with respect to an acyclic cover in the tropical topology. From \cite[Section~3.4]{molchowiselog} we have that:
\begin{itemize}[leftmargin=0.4cm]
\item For  $\mathcal{PL}$ we have, simply unraveling definitions:
\[H^0(\Gamma,\mathcal P\mathcal L)=H^0(C,\bar{M}_C^\gp),\qquad H^1(\Gamma,\mathcal P\mathcal L)=H^1(C,\bar{M}_C^\gp).\]
\item  $H^0(\Gamma,\mathcal V)$ is the set of divisors supported on the vertices of $\Gamma$. 
\item $H^i(\Gamma,\mathcal V)$ and $H^i(\Gamma,\mathcal E)$ are zero for $i>0;$ 
\item  $H^0(\Gamma,\mathcal H)$ is identified with the first homology group $H_1(\Gamma,\mathbb Z)$ of the graph underlying $\Gamma$ and  $H^1(\Gamma,\mathcal H)$ with $H_0(\Gamma,\mathbb Z);$
\item there is a natural isomorphism 
\[H^1(\Gamma,\bar{M}_{S}^{\gp})\cong\operatorname{Hom}(H_1(\Gamma,\mathbb Z),\bar{M}_{S}^{\gp}).\]
\end{itemize}
If $\Gamma\xrightarrow{\pi} S$ is a tropical curve over a general fs logarithmic scheme, we can still define these sheaves and the maps between them by  considering the definitions we just gave on geometric points and imposing compatibility along the maps induced by specialization $t\rightsquigarrow s$ (see \cite[Section~3.7]{molchowiselog} ).


\subsection{Logarithmic and Tropical Picard group for log smooth curves}\label{sec:logpic}

In \cite[Section~3]{molchowiselog} the  authors define the tropical Jacobian  as a group object
 \[\mathrm{TropPic}^0_{\Gamma/S}\colon \mathrm{LogSch^{fs}_S}\to \mathrm{AbGrp},\]
on the category of fs log schemes.

     For $(\Spec k, Q)\to S$ a log point
     \[\mathrm{TropPic}^0_{\Gamma/S}((\Spec k, Q))=\frac{\mathrm{Hom}(H_1(\Gamma_{k},\ZZ),Q^\mathrm{gp})^\dagger}{H_1(\Gamma_{k},\mathbb Z)},\]
     where the $\dagger$ indicates the restriction to \emph{bounded monodromy} homomorphisms as defined in \cite{molchowiselog} (see also \cite{kajiwara2008logabelian2}). The latter is a technical condition to ensure algebraization of deformations of tropical line bundle, the object parametrized by the tropical Jacobian. 

         When $Q=\mathbb R_{\geq 0}$ the $\mathrm{TropPic}^0_{\Gamma/S}((\Spec k, Q))$ is a real torus of dimension $b_1(\Gamma)$ and it is in fact precisely the tropical Jacobian of a classical tropical curve as defined in \cite{mikhalkin2008tropical}.

     \subsubsection{} From the  long exact sequences induced by the diagram of sheaves \eqref{eq:tropicalsheaves}, we have that:
\[\mathrm{TropPic}^0_{\Gamma/S}((\Spec k, Q))= \mathrm{ker}( H^1(\Gamma,\mathcal B\mathcal L)\xrightarrow{\deg}H_0(\Gamma,\mathbb Z)\cong\mathbb Z);\]
in particular we can define the torsors $\mathrm{TropPic}^d_{\Gamma/S}((\Spec k, Q))$ for each $d\in\mathbb Z.$
     
     \subsubsection{}  $\mathrm{TropPic}^d_{\Gamma/S}$ admits a logarithmically smooth cover by an algebraic stack with logarithmic structure. In fact,  $\mathrm{TropPic}^d_{\Gamma/S}$ can be thought of as the limit over all subdivisions $\widetilde{C}\to C\times_S\widetilde S$ for $\widetilde{S}\to S$ a logarithmic modification or root of $S$ of rational equivalence classes of degree $d$ divisors on  the tropical curve:
     \[\mathrm{Div}^d(\widetilde{\Gamma}/\widetilde{S})/\mathrm{sPL}(\widetilde{\Gamma}),\]
     where a divisor on $\Gamma_s$ is a formal linear combination $\sum_{v\in V(\Gamma_s)} a_v [v]$ with $a_v\in\mathbb Z;$  for $\Gamma/S$ we have as usual such a data on each logarithmic stratum plus the compatibility along generalization maps.

     This means in particular that any \emph{tropical line bundle} can be represented by a degree $d$ divisor on some subdivision of $\Gamma/S.$

\subsubsection{} Let $C\xrightarrow{\pi}S$ be a proper, vertical, log curve. A \emph{logarithmic line bundle}
on $C$ is a $\Gmlog$-torsor on $C$ in the strict \'etale topology whose fibers over $S$ have bounded monodromy 
\cite{molchowiselog}, \cite{kajiwara2008logabelian2}.
     The logarithmic Picard group of $C\xrightarrow{\pi} S$ is the sheaf of  abelian groups over log schemes parametrizing isomorphism classes of logarithmic line bundles, i.e:
\[\LogPic_{C/S}( T)= (R^1\pi_* M_{C_T}^{\gp})^\dagger,\] 
the sub-sheaf of $R^1\pi_* M_{C_T}^{\gp}$  where bounded monodromy is satisfied. The latter is in fact a condition on the induced $\bar{M}_{C_T}^{\gp}$-torsor.

\subsubsection{} Given $C\to(\Spec k, Q)$ there is a well-defined degree map 
\[\deg\colon\LogPic_C \to\ZZ.\]
Over a  more general logarithmic  scheme the degree is locally constant as proved in \cite[Proposition~4.5.2]{molchowiselog}.
We then define the logarithmic Jacobian to be the kernel of the degree morphism. As proved in \cite[Section~4]{molchowiselog} there is a short exact sequence of sheaves of groups over $\mathrm{LogSch^{fs}}:$

\[0\to \Pic^{[0]}_{C\slash S}\to \LogPic^0_{ C/S}\to\mathrm{TroPic}^0_{C/S}\to 0,\]
and $\LogPic^0_{ C/S}$ is a logarithmic abelian variety in the sense of \cite{kajiwara2008logabelian2}.

\subsubsection{} As explained in \cite[Section~4.3]{molchowiselog}, a  $M_C^\gp$ -torsor $L$ has bounded monodromy along the fibers if and only if there is a logarithmic modification (log blowups and root construction) $\widetilde{S}\to S$ possibly  followed by a subdivision
    $\widetilde{C}\to C\times_{S}\widetilde{S} $ such that the pull-back $\widetilde{L}$  on $\widetilde{C}$ is represented by a line bundle, i.e. $\widetilde{L}$ is in the image  of
    \[\Pic(\widetilde{C}\slash\widetilde{S})=H^0(\widetilde{S},R^1\pi_*\mathcal O_{\widetilde{C}}^{\times})\to \LogPic_{\widetilde{C}\slash\widetilde{S}}(\widetilde{S}), \]
    or more precisely
    \[\Pic(\widetilde{C}\slash\widetilde{S})=H^0(\widetilde{S},R^1\pi_*\mathcal O_{\widetilde{C}}^{\times})\to \LogPic_{C/S}(\widetilde{S})\cong\LogPic_{\widetilde{C}\slash\widetilde{S}}(\widetilde{S})\subseteq H^0(\widetilde{S},R^1\pi_*M_{\widetilde{C}}^{\mathrm{gp}}). \]
    The map is the natural one in by the long exact sequence induced by \eqref{eq:fundamentallogses}.

    \subsubsection{}\label{sec:fixmultidegree} The model over which $L$ can be represented by a line bundle can be understood tropically: given $L$
 we have an associated $\mathrm{Trop}(L)\in \mathrm{TroPic}_{\Gamma/S}\cong (R^1\pi_*\B\L)^\dagger$ given as explained in \cite[Section~4.1]{molchowiselog}. By construction, $\mathrm{Trop}(L)$ is a lift of the $\bar{M}^{\mathrm{gp}}_C$ torsor $\bar{L}\in (R^1\pi_*\mathcal P\mathcal L)^\dagger\cong (R^1\pi_*\bar{M}^{\mathrm{gp}}_C)^\dagger$ associated to $L.$

    From the long exact sequences induced by diagram~\eqref{eq:tropicalsheaves} we see that the latter is zero only if there exists a divisor $D$ supported on the vertices of $\Gamma$ such that $\mathrm{Trop}(L)=[D]$ under the coboundary map $H^0(\Gamma,\mathcal V)=\mathrm{Div}(\Gamma/S)\to H^1(\Gamma,\B\L).$

Let $\widetilde{\Gamma}/\widetilde{S}$ be the modification over which $\mathrm{Trop}(L)$ is represented by $D$ and let $\widetilde{C}\to \widetilde{S}$ the induced log curve. 

It follows  from the exact sequence
\[H^0(\widetilde{S},\pi_* \bar{M}_{\widetilde{C}}^{\gp,\dagger})/\bar{M}_S^\mathrm{gp}\to  \Pic_{\widetilde{C}/ \widetilde{S}}(\widetilde{S})\to  \LogPic_{\widetilde{C}/ \widetilde{S}}(\widetilde{S})\to  H^0(\widetilde{S}, (R^1\pi_*\bar{M}_{\widetilde{C}}^\mathrm{gp})^{\dagger})\]
that once $D$ is fixed there is a unique line bundle $L_D$, up to line bundles pulled-back from the base, representing the log line bundle $L.$



\subsubsection{} By \cite[Corollary 4.11.4]{molchowiselog}, $\LogPic^0_{ C/S}$ admits a logarithmically smooth cover by a logarithmic smooth scheme.  More concretely, given a subdivision $\mathfrak P^{\sigma}\to \mathrm{TroPic}^0_{C/S}$ such that $\mathfrak P^{\sigma}$ is represented by a cone stack in the sense of \cite{cavalieri2020moduli}, we have that the induce subdivision $\P^{\sigma}\to\LogPic^0_{ C/S}$ where $\P^{\sigma}$ is a an algebraic stack with log structure which is furthermore log smooth. We call $\P^{\sigma}$ a \emph{model} for the logarithmic Jacobian.

One most important examples of $\P^{\sigma}$ to keep in mind are the ones coming from \emph{fine compatified Jacobians} with respect to \emph{small, non degenerate stability conditions} \cite{melo2015compactifications}, \cite[Section~4]{holmes2025logDR} and references therein.

\subsection{Logarithmic and tropical roots of a line bundle}\label{sec:logroots}
We conclude this section recalling some results about compactification of moduli of roots of a line bundle on a family of pre-stable curves. The material is mostly taken from \cite{holmes2023root}, but we use some results from \cite{abreu2023moduli} to describe suitable models on which the universal logarithmic root is representable.

\subsubsection{Moduli of log-roots} Let $S$ an algebraic stack with logarithmic structure, $C\to S$ be a log smooth curve, and let $\mathcal L$ a line bundle $C$ of degree $d$ along the fibers. Given $r \;|\; d$ we 
 denote by $ S(\mathcal L^{1/r})$ the moduli space of log-$r$-roots of $\mathcal L$. This is constructed, as explained in \cite{holmes2023root}, as the fiber product in the category of stacks over log schemes:
\bcd
 S(\mathcal L^{1/r}) \ar[r]\ar[d]  &\LogPic_{C/S} \ar[d,"r\cdot"] \\
S  \ar[r,"\mathcal L"] & \LogPic_{C/S},
\ecd
where $r\cdot$ denotes the multiplication map. 

\medskip

 The stack $S(\mathcal L^{1/r})$ comes equipped with a universal curve $C\times_S S(\mathcal L^{1/r})$ and universal log-line bundle $\mathcal T$ such that
\[\mathcal T^{\otimes r}\equiv \mathcal L \in \LogPic_{C/S}( S(\mathcal L^{1/r})),\]
where $\equiv$ here and in what follows will denote the equivalence in the sense of log-line bundles.
 
\begin{prop}\cite[Corollary~3.12,Section~4.1.3]{holmes2023root}
   The moduli space  $S(\mathcal L^{1/r})$  is a proper, finite, flat scheme over $S$ of degree $r^{2g};$ if $r$ is invertible on $S$ then is log \'etale (but not \'etale in general) over $S.$ 
\end{prop}
Notice over the locus where $C\to S$ is smooth, the map $S(\mathcal L^{1/r})\to S $ is \'etale; for $\mathcal L=\mathcal O_C$  we have that $S(\mathcal O_C^{1/r})$ is a group scheme over $S$ and more generally $S(\mathcal L^{1/r})$ is a torsor under the latter.
By considering the base change along a root stack $\widetilde{S}\to S$ which we now explain, we can regain these properties. 
\subsubsection{Root stack construction.}
 In the notation of \cite{chiodo2008stable}, $D_{i,(A|B)}$ the boundary divisor in $\mathfrak{M}_{g,n}$ of curves with one separating node where:  the genus splits as $i, g-i;$ the subsets $A, B \subseteq \left\{1,\dots,n\right\}$ prescribe the distribution of the markings on the two components and let $D_{\mathrm{irr}}$ the boundary divisor of curves with a non separating node. Consider the root-stack (see for example \cite{abramovich2008gromov,cadman2007using} and also \cite[Section~2]{chiodo2008stable}:
\[\widetilde{\mathfrak{M}}_{g,n}=\mathfrak{M}_{g,n}\left[\sum_{i,A,B} D_{i,(A|B)}/r +D_{\mathrm{irr}}\slash r \right].\]
The latter  admits a modular interpretation: it is the moduli stack $\mathfrak{M}_{g,n}(\vec{r})$ of \emph{twisted} pre-stable curves in the sense of \cite{chiodo2008stable} with stabilizer of order $r$ at the nodes \cite[Theorem~4.5]{chiodo2008stable}.

\begin{rem}\label{rem:rootstackmonoids}
The root construction can  also be described in terms of
 logarithmic structures, as done in \cite[Section~2]{holmes2023root}.  At the level of monoids,  the upshot of the root  construction is the following: given closed point $s$, the monoid $\overline{M}_{\widetilde{\mathfrak{M}}_{g,n},s}$ is the saturation of the image of $\overline{M}_{\mathfrak{M}_{g,n},s}$ inside $(\mathbb Z^{|E(\Gamma_s)|}\oplus \mathbb Z^{|E(\Gamma_s)|})/(\ell_e-r\ell'_{e})$, i.e.
 $\overline{M}_{\widetilde{\mathfrak{M}}_{g,n},s}=\frac{1}{r} \overline{M}_{\mathfrak{M}_{g,n},s}.$

\medskip

    Holmes and Orecchia \cite{holmes2023root} consider a considerably more  efficient root stack which only extracts roots for those boundary components of $\mathfrak{M}_{g,n}$ corresponding to non separating nodes, namely for the components $\mathfrak{M}_{\Gamma}$ for which $b_1(\Gamma)\neq 0;$ even for those the roots are not extracted separately for each non separating edge. 
\end{rem}

There are three universal curves on $\widetilde{\mathfrak{M}}_{g,n}:$ 
\begin{itemize}
    \item the universal twisted curve $\mathfrak{C}^{1/r,\mathrm{tw}}$ with stacky nodes \cite{chiodo2008stable}; 
   \item  the corresponding coarse curve $\mathfrak{C}$, which is simply the pull-back from $\mathfrak{M}_{g,n}$ and it has singular total space due to  the base change; 
   \item the resolution $\widetilde{\mathfrak{C}}$ of the latter, whose fibers are obtained inserting a chain of $r-1$ unstable $\PP^1$ at each node.
\end{itemize}

A log smooth $C/S$ determines a log  morphism 
$S\to \mathfrak M_{g,n}$ where the latter is endowed with the natural logarithmic structure defined by the boundary divisor.
Then we define 
\[\widetilde{S}:=S\times_{\mathfrak M_{g,n}}\widetilde{\mathfrak M}_{g,n},\]
where the fiber product is in the category of  fs logarithmic algebraic stacks. 

By pull-back, on  $\widetilde{S}$ we have a twisted curve $C^{1/r},$ its coarse moduli space $C=C\times_S\widetilde{S}$ and the resolution of the latter, which we denote by $\widetilde{C}.$

If $\Gamma$ and $\widetilde{\Gamma}$ denote the tropicalization of $C_s,$ $\widetilde{C}_s$ respectively on some geometric point $s\in\widetilde{S},$ then $\widetilde{\Gamma}$ is obtained from $\Gamma$ subdividing each edge $e\in E(\Gamma)$ is $r$ edges $e_r\in E(\widetilde{\Gamma})$ of length $\ell(e_r)=\frac{\ell(e)}{r}.$ 
We have the following:
\begin{theo}\cite[Theorem~4.3]{holmes2023root},\cite[Section~3.1]{chiodo2008stable}\label{thm:logrotsafterbasechange}
\begin{itemize}
    \item There is a strict logarithmic morphism $\widetilde{S}(\mathcal L^{1/r}):=S(\mathcal L^{1/r})\times_S\widetilde{S}\to \widetilde{S}$, where the fiber product is in the category of fs logarithmic schemes.
    \item The underlying stack of $\widetilde{S}(\mathcal O^{1/r})$ is a finite flat group scheme over $\widetilde{S},$ \'etale if $r$ is invertible on $S$.
    \item The underlying stack of $\widetilde{S}(\mathcal L^{1/r})$ is a finite \'etale torsor under $\widetilde{S}(\mathcal O^{1/r})$ over $\widetilde{S},$ for $r$ is invertible on $S$.
\end{itemize}
\end{theo}
Chiodo \cite{chiodo2008stable} interprets $\widetilde{S}(\mathcal L^{1/r})$ as moduli of roots of a line bundle on the \emph{twisted curve} $C^{1/r}.$ Since we will not need this point of view, we will refrain from explaining it.

Furthermore, the following is true: on the universal curve $\widetilde{C}$ over $\widetilde{S}(\mathcal L^{1/r})$ 
 the universal log root $\mathcal R$ can be represented by an honest line bundle, which we still denote by $\mathcal R\in \Pic_{\widetilde{C}/\widetilde{S}(\mathcal L^{1/r})}$. We refer the reader to \cite[Lemma~2.2.5]{chiodo2008towards} or \cite[Section~2.2]{chiodo2024double}. The line bundle $\mathcal R$ is only a $r$-root  in the  logarithmic sense, namely we only know that fiberwise on $\widetilde{C}\to\widetilde{S}(\mathcal L^{1/r})$

 \begin{equation}\label{eq:isolinebundles}
    \mathcal \R^{\otimes r}=\L_{\widetilde{C}}(\alpha),
\end{equation}

where $\alpha$ is some strict piecewise linear function on the universal tropical curve $\widetilde{\Gamma},$ namely a strict piecewise linear function on $\widetilde{\Gamma}_s$ with values in $\bar{M}^\mathrm{gp}_{\widetilde{S}(\mathcal L^{1/r}),s}$ for each $s\in \widetilde{S}(\mathcal L^{1/r})$ compatible under generalization, i.e. edge contractions.

\subsubsection{Tropical torsion divisor and choices of $\mathcal R$}
 In this subsection we use \cite[Section~4,5]{abreu2023moduli} to give an explicit description of a pair $(\mathcal D,\alpha),$ 
 where $\mathcal D$ is a tropical torsion divisor on $\widetilde{\Gamma}\to \widetilde{S}(\mathcal L^{1/r}) $ determining the lift $\mathcal R_{\mathcal D}$ and $\alpha$ is the sPL function on $\widetilde{\Gamma}$  satisfying \eqref{eq:isolinebundles}.

\begin{rem}
  The existence and properties of this line bundle representative $\mathcal R_{\mathcal D}$ can be also understood in terms of the  \emph{ universal limit root } over  the compactifications of moduli of roots   constructed  by Caporaso-Casagrande-Cornalba \cite{caporaso2007moduli} and Jarvis \cite{jarvis1998torsion,jarvis2000geometry}, since the moduli space $\widetilde{S}(\mathcal L^{1/r})$ maps to all of these compactifications (See \cite[Section~6]{abreu2023moduli},\cite[Section~4.3]{chiodo2008stable}).  
\end{rem}
 
Let $\widetilde{\Gamma}\to \widetilde{S}(\mathcal L^{1/r})$ the universal tropical curve, and let $\underline{\deg}(\mathcal L)$ the degree $d$ tropical divisor on $\widetilde{\Gamma}$ of which we want to take roots.
The latter is simply defined by

\[\underline{\deg}(\mathcal L)\rvert_{\widetilde{\Gamma}_s}=\sum_{v\in V(\widetilde{\Gamma}_s)} \deg(\mathcal L\rvert_{\widetilde{C}_v})[v].\]
The divisor $\underline{\deg}(\mathcal L)$ is in fact induced by pull-back by the divisor $\Gamma,$ the universal tropical curve pulled back from $S,$ defined the same way.


\begin{defi}\label{def:weighting}
    A \emph{flow modulo r} of $\Gamma$ is an element $\phi\in C_1(\Gamma,\ZZ_r)$. We say it is compatible with $\underline{\deg}(\mathcal L)$ if $\partial\phi=\underline{\deg}(\mathcal L)$, seen as an element of $C_0(\Gamma,\ZZ_r)$, where $\partial$ denotes the boundary operator.
\end{defi}

Notice that the set of flows compatible with $\underline{\deg}(\mathcal L)$ is non-empty if and only if $r$ divides $d$. In the latter case, we see that the set of flows compatible with $\underline{\deg}(\mathcal L)$ is a torsor under $H_1(\Gamma,\ZZ_r)=\ker\partial$ and has cardinality $r^{b_1(\Gamma)}$.

\medskip

Given $\phi$ a flow modulo $r$, we can define a function $w\colon H(\Gamma)\to \{0,\dots, r-1\}$ on the set of half-edges of $\Gamma$, usually called a \emph{weighting modulo $r$}. To each half-edge $h$ belonging to an edge $e=\{h,h'\}$, we associate the unique lift in $\{0,\dots, r-1\}$ of the coefficient of $\phi$ along $e$, oriented from $h$ to $h'$. In particular, we have the following:
\begin{enumerate}
    \item for each edge $e=\{h,h'\}$ we have $w(h)+ w(h')\equiv 0 \mod r$;
    \item for each vertex $v\in V(\Gamma)$ \[\sum_{h\vdash v}w(h)\equiv \underline{\deg}(\mathcal L)(v) \mod r,\]
    where $h\vdash v$ means that $v$ is the base vertex of the half edge $h$.
\end{enumerate}
The data of the function $w$ is easily seen to be equivalent to the data of $\phi$, and we will use the terminology ``weighting modulo $r$''.

\begin{rem}
We refer to \cite[Section~2,3]{abreu2023moduli} for more complete definitions of flows on graphs and associated divisors. Notice that in the notation of  \cite[Section~4]{abreu2023moduli} we have
\[ \partial^{-1}(\underline{\deg}(\mathcal L)\;\;\mathrm{mod}\; r)=\mathcal F(\Gamma, \overline{\deg(\mathcal L)})\]
where $\overline{\deg(\mathcal L)}=\underline{\deg}(\mathcal L)\;\;\mathrm{mod} \; r$.
\end{rem}

 We denote by \[\mathrm{Root}^{r,\mathrm{trop}}(\Gamma,\underline{\deg}(\mathcal L))\subseteq \mathrm{TroPic}^d(\Gamma)\] the $\mathrm{TroPic}^0(\Gamma)[r]$-torsor parametrizing $r$-roots of $\underline{\deg}(\mathcal L).$

 \medskip

 As recalled above, the tropical Picard group can be described as the inverse limit of the set of divisors up to linear equivalence on subdivision  $\Gamma'\to \Gamma\times_S S'$ for $S'\to S$ some log modification, where in our definition divisors are supported on vertices.

 \begin{defi}\label{def:torsiontropical}

 Given a weighting $\phi$ compatible with $\underline{\deg}(\L)$ and $w$, the associated weight function on half-edges, we define a divisor $D_{w}$ supported on the vertices of $\widetilde{\Gamma}$ such that we have $rD_w\sim \underline{\deg}(\mathcal L)$. In other words, $D_{w}$ is a representative of one of the tropical roots of $\underline{\deg}(\mathcal L)$. The divisor is defined as follows:
 \begin{itemize}[leftmargin=0.4cm]
     \item For each $v\in V(\Gamma)$, we set
     \[D_w(v)=\frac{1}{r}(\underline{\deg}(\mathcal L)(v) -\sum_{h\vdash v}w(h)) .\]
Notice that the above expression is well-defined, since by condition $(2)$ in Definition~\ref{def:weighting} the quantity in brackets is divisible by $r$.
\item For each $e=\{h,h'\}$ with $w(h)\neq 0$, let us denote $u$ the vertex on $\widetilde{\Gamma}$ at distance $\frac{w(h)}{r}\ell(e)$ from the vertex $v$ base of the half edge $h$ ($u$ is also  at distance $\frac{w(h')}{r}\ell(e)$  from the vertex $v',$  base of the half edge $h'$), then we take $D_w(u)=1$.
 \end{itemize}
 \end{defi}

 \begin{lem}\label{lem:plfunction}
     For each weighting $\phi$ compatible with $\underline{\deg}(\L)$, the divisor $D_w$ defined above is a root of $\underline{\deg}(\mathcal L)$, i.e. there is a strict piece-wise linear function $\alpha_D$ on $\widetilde{\Gamma},$ with values in $\bar{M}^\mathrm{gp}_{\widetilde{S}(\mathcal L^{1/r}),s}$ such that
     \begin{equation}\label{eq:lin}
         \underline{\deg}(\mathcal L)-rD_w=\mathrm{div}(\alpha_D).
     \end{equation}
     Furthermore we can choose $\alpha_D$ to take value $0$ on each vertex 
     $v\in V(\Gamma)\subseteq V(\widetilde{\Gamma}).$
 \end{lem}

 \begin{proof}
     If it exists, the function $\alpha_D$ is unique up to the addition of a constant. 
     We define explicitly a function $\alpha_D$ and check that it satisfies the requirements. 

     \begin{equation}\label{eq:sPL}
       \alpha_D(v)=\begin{cases}
           0\;\;\;\forall\; v\in V(\Gamma),\\
          \frac{ w(h')w(h)}{r}\ell(e)\;\; \mathrm{for}\;\; v=u \text{ a vertex where }D_w(u)=1. 
       \end{cases}
        \end{equation}
        In particular we have that $\alpha_D$ is linear over each edge $e_h, e_{h'}$ obtained subdividing $e=\{h,h'\}$ at $u$, with slope  $w(h')$ along $e_h$ from $v$ to $u$, and slope $w(h)$ along $e_h'$ from $v'$ to $u$, whenever $w(h)$ is non-zero. 
        We now check that equation~\eqref{eq:lin} is satisfied.
        \begin{itemize}
            \item For  $v\in V(\Gamma)$ we have:
            \begin{align*}
            rD_w+\mathrm{div}(\alpha_D)(v)= \left(\underline{\deg}(\mathcal L)(v) -\sum_{h\vdash v}w(h)\right) + \sum_{h\vdash v} w(h),
        \end{align*}
     where the second sum is by the definition given just above the sum of the outgoing slopes of $\alpha_D$ at $v.$
     \item for $u$ an exceptional vertex we have:
     $$rD_w+\mathrm{div}(\alpha_D)(u)=r-w(h)-w(h')=0,$$
     where the second equality follows from Definition~\ref{def:weighting}.
     \end{itemize}        
 \end{proof}

\begin{rem}
    The root representative of $D_w$ we define here differs from the one defined in \cite[Definition~4.3]{abreu2023moduli}, but one can explicitly define a strict piece-wise linear function on $\widetilde{\Gamma}$ whose associated divisor is the difference of the two.
    The reason for our slightly different point of view is that our definition of $D_w$ does not require a choice of orientation on $\Gamma.$
\end{rem}
 We collect in the Proposition below the properties we need about the construction above; all the proofs can be found in \cite[Section~4,5]{abreu2023moduli}.
 \begin{prop}\label{prop:constructroots}
 Denoting by $\W_r(\Gamma,D)$ the set of weightings modulo $r$ compatible with $D=\underline{\deg}(\L)$, we have the following.
     \begin{enumerate}
         \item The definition of weightings modulo $r$ and of $D_w$ is compatible with edge contraction 
         (see \cite[Corollary~4.7]{abreu2023moduli}). 
         \item We have a bijection (see \cite[Proposition~4.12]{abreu2023moduli}):
         \begin{align*}
          D\colon \W_r(\Gamma,D)&\longrightarrow   \mathrm{Root}^{r,\mathrm{trop}}(\Gamma,\underline{\deg}(\mathcal L))\subseteq\mathrm{TroPic}(\Gamma).\\
          w&\longmapsto [D_w]
         \end{align*}
         \item If $\underline{\deg}(\mathcal L)=\underline{0}$, $\W_r(\Gamma,D)$ corresponds to flows on $\Gamma$, i.e. elements in $H_1(\Gamma,\ZZ_r)$. In particular, $w=0$ is a weighting compatible with $\mzero$, and $D_0$ is the trivial tropical divisor.
     \end{enumerate}
 \end{prop}
In particular, choosing a consistent representatives divisor for each weighting as done above, we obtain $\mathcal D$ a tropical divisor on the universal tropical curve $\widetilde{\Gamma}\to\widetilde{S}(\mathcal L^{1/r})$ representing the universal tropical $r$-roots. 

\medskip

We recall that for each logarithmic strata of $T\subseteq S$ with corresponding tropicalization $\sigma_T$ there are $r^{b_1(\Gamma_T)}$ strata in $\widetilde{S}(\mathcal L^{1/r})$ mapping on the latter, indexed by the tropical $r$-roots of $\mathcal L.$ We refer to \cite{holmes2023root} (see also \cite{blommecarocci2025DR}) for further details on this remark.
 
\begin{defi}\label{rem:rootrep}
    We define $\mathcal R\in\Pic_{\widetilde{C}/\widetilde{S}(\mathcal L^{1/r})}$ to be the unique (see \eqref{sec:fixmultidegree} )representative of the universal logarithmic root $\mathcal L^{1/r}$ whose associated multidegree $\underline{\deg}(\mathcal R\rvert_{\widetilde{C_s}})=\mathcal D\in\mathrm{Div}(\widetilde{\Gamma}_s)$ is the preferred one in its linear equivalence class, given by the above construction. 
By definition, $\mathcal T$ satisfies

\[\mathcal \R^{\otimes r}=\mathcal L(\alpha_D),\]
 for $\alpha_D$ the PL function defined in Lemma~\ref{lem:plfunction}. Notice that this is in particular one of the piece-wise linear function considered in \cite[Section~6]{holmes2023root} to compute the spin double ramification cycle.
\end{defi}

\section{(Logarithmic) Double (Double) Ramification Cycles}
\label{sec-LDDR}

\subsection{Double ramification cycle}

We assume the reader to be already familiar with the definition of double ramification cycles and refer to the vast literature for a detailed account  of the many results on the  topic  \cite{janda2017DRcurves,janda2020double,bae2023pixton,holmes2021extending,holmes2025logDR}.

 In order to fix notation and for later use we recall the definition obtained via Marcus-Wise logarithmic compactification of the Abel-Jacobi map \cite{marcuswiselog}, as the latter is the one used to obtain the universal DR-cycle formula of \cite{bae2023pixton}.

\begin{defi}\cite{marcuswiselog}
    We denote by  $\Div_{g}$ the stack over fs log schemes whose $S$ points parametrize: a log smooth curve $\pi\colon C\to S;$ a section $\alpha\in H^0(S,\pi_*\bar{M}_C^\gp/\bar{M}_S^\gp)$, i.e. $\alpha$ is the data of a strict piece-wise linear function on $\Gamma_s$ with fixed slope $0$ along the markings,  well-defined up to translation by a constant and compatible with edge contraction.
    This is the same as a section of $\pi_*\bar{M}_C^{V,\gp}/\bar{M}_S^\gp$ where $M_C^{V,\gp}$ is the vertical part of the logarithmic structure. 
\end{defi}
By \cite[Theorem~4.2.4]{marcuswiselog}, $\Div_g$  is actually representable by a locally of finite type algebraic stack with log structure; in fact it is logarithmically \'etale on $\mathfrak{M}_{g,n}^{\mathrm{log}}.$

\medskip

\subsubsection{} Let $C/S$ be a log smooth curve over a log algebraic stack  $S$ and let $\L$ a line bundle on $C$ whose restriction to each fiber has  total degree $0$; this data defines a morphism from $S$ to the universal Picard stack $ \mathfrak{Pic}_{g,n,0}.$ 

We call \emph{double ramification locus} $\DRL(\L)$  the following fiber product in the category of  log schemes (\cite[Definition~4.5.2]{marcuswiselog}):
\bcd 
\DRL(\L) \ar[r]\ar[d,"\epsilon"] & \Div_g \ar[d,"\mathrm{aj}"] \\
S \ar[r,"\varphi_{\L}"] & \mathfrak{Pic}_{g,n,0}
\ecd
where $\mathrm{aj}$ is defined by sending $\alpha\mapsto \mathcal O_C(\alpha).$  
To be precise, the right vertical morphism takes value in the relative Picard space $\mathfrak{Pic}^{\rm{rel}}_{g,n,0}$ parametrizing isomorphism classes of line bundles. This can however be lifted by pull-back to a morphism with values in the Picard stack as explained in \cite[Definition 31]{bae2023pixton}. 
By a slight abuse of notation we keep denoting by $\Div_g \xrightarrow{\mathrm{aj}}
\mathfrak{Pic}_{g,n,0}$ the lifted map.

The fiber product $\DRL(\L)$ is  a modular compactification of locus $\DRL(\L)^\circ\subseteq S^\circ\subseteq S,$ where $S^\circ$ is the open in $S $ where the logarithmic structure is trivial (i.e. over which the log smooth curve is actually smooth) cut out by the condition: $\L$ restricts to the trivial line bundle fiberwise.


\begin{rem}
Given a vector  $\bfa=(a_1,\dots,a_n)$, it is possible to define $\operatorname{Div}_{g,\bfa}$ with a map to $\mathfrak{Pic}_{g,n,d}$, where $d=\sum a_i$, looking at sections of $\pi_*\bar{M}_C^\gp/\bar{M}_S^\gp$ with fixed slope $a_i$ along the $i$-th unbounded edge corresponding to  the $i$-th marked point. Taking the same fiber product as before, we get a compactification of the locus cut out by the relation $\L=\O(\bfa)$ on each fiber.
We can however always reduce to the previous situation considering $\L(-\bfa)$ instead. 
It is thus not restrictive to assume that $d=0$ and consider $\operatorname{Div}_g$. We refer to \cite[Section~7]{bae2023pixton} for further details on the equivalence between these two points of view.
\end{rem}

\subsubsection{} As explained in \cite[Section 4.6]{marcuswiselog}, the above fiber product is proper over $S$ and it is naturally endowed with an obstruction theory induced by pull-back from the obstruction theory for the Abel-Jacobi map $\mathrm{aj}.$ 

In particular, if $S$ has a fundamental class or a virtual fundamental class, by virtual pull-back we get a
 class $\vir{\DRL(\L)}$ whose push-forward to $S$ is the so-called DR-cycle \[\DR(\L)\in\operatorname{CH}_{\dim S-g}(S).\]
 We remark that, implicit in the notation, the class lives in the expected codimension $g;$ $\dim$ should be replaced by $\mathrm{vdim}$ when we work with virtual classes.

The main result of \cite{bae2023pixton} furthermore tells us that 
this class can be computed via the universal DR-cycle formula, i.e.
$$\epsilon_*\vir{\DRL(\L)} = \varphi^*_{\L}\DR\cap [S]^{(\mathrm{vir})}$$
where $\DR\in \operatorname{CH}^{\mathrm{op},g}( \mathfrak{Pic}_{g,n,0}) $  admits a universal  Pixton type formula
\cite[Section~0.3.5]{bae2023pixton}.


\begin{expl}
Let $\Mgnbar$  denote the moduli space of stable curves; this has a natural logarithmic structure induced by the normal crossing boundary divisor. Let $\bfa=(a_1,\dots,a_n)$ be a $n$-tuple of weights satisfying $\sum a_i=0$. Then the line bundle $\O(\bfa) = \O(\sum a_ip_i)$ has total degree zero on the fibers of the universal curve and thus defines a map to the relative Picard stack.

The DR-cycle $\DR(\bfa):=\epsilon_*[\DRL(\O(\bfa))]\in\operatorname{CH}_*(\Mgnbar)$ is the classical double ramification cycle. In this case, applying the universal DR-cycle formula \cite{bae2023pixton} recovers the more classical Pixton's formula first proved in \cite{janda2017DRcurves}.
\end{expl}

\subsection{Double double ramification cycle}

Let $C/S$ be a log smooth curve over a smooth log algebraic stack  $S$ and let $\L_1,\L_2$ two degree $0$ line bundles on the fibers of $C/S$. This data defines a morphism 

\[S\xrightarrow{(\varphi_{\L_1},\varphi_{\L_2})}\mathfrak{Pic}_{g,n,0}\times_{\mathfrak{M}_{g,n}} \mathfrak{Pic}_{g,n,0}. \]

We define $\DDRL(\L_1,\L_2)$  as the following fiber product in the category of  log schemes:
\begin{equation}\label{eq:diagramDDR}
    \begin{tikzcd}
    \DDRL(\L_1,\L_2) \ar[r]\ar[d] & \Div_g \times^{fs}_ {\mathfrak{M}_{g,n}}\Div_g\ar[d,"\mathrm{aj}\times\mathrm{aj}"]\\
S \ar[r] & \mathfrak{Pic}_{g,n,0}\times_{\mathfrak{M}_{g,n}} \mathfrak{Pic}_{g,n,0} .
 \end{tikzcd}
\end{equation}

The fiber product $\Div_g \times^{fs}_ {\mathfrak{M}_{g,n}}\Div_g$, taken in the category of fs log stacks, is also denoted by $\mathfrak{M}^{\mathrm{rub}}_{\Lambda}(\Gmtrop^2)$ in \cite[Section~6]{ranganathan2024logarithmic}, where $\Lambda$ denotes the discrete data. 
Since  the morphism $\Div_g\to\mathfrak{M}_{g,n}$ is not strict nor weakly semistable, this is an example where the underlying stack of the fs fiber product is \emph{not} the fiber product of the underlying stack. This difference is the reason for the failure of multiplicativity of the DR class. We refer to \cite{molcho2024case} for a more detailed discussion.

Also in this case one can consider the variation  $\Div_{g,\bfa} \times^{fs}_ {\mathfrak{M}_{g,n}}\Div_{g,\bfb}.$ 

It is deduced from the previous case that this fs fiber product is in fact a locally finite type algebraic stack with logarithmic structure, log \'etale over $\mathfrak{M}_{g,n}$.

\medskip

Seen as a stack over $\mathrm{LogSch}^{\rm{fs}},$ the fiber product  $\DDRL(\L_1,\L_2)$ parametrizes: 
\begin{itemize}
    \item a log smooth curve $C\to T$ for $T$ a log scheme over $S;$
    \item two logarithmic trivializations $\widetilde{\alpha}_i$ is a section of  $\pi_* \L_i^\times\otimes_{\O_{C_T}^\star}M_{C_T}^{\gp}/M_T^{\gp},$ for  $i=1,2.$
\end{itemize}

As explained in the proof of \cite[Proposition~4.5.3]{marcuswiselog} the logarithmic trivializations induce sections $\alpha_1,\alpha_2\in H^0(T, \pi_*\bar{M}_{C_T}^{\gp}/\bar{M}_T^{\gp})$ such that there are (fiberwise) isomorphisms 
\begin{equation}\label{eq:logtriv}
    \L_1\rvert{C_T}\simeq\O_{C_T}(\alpha_1),\;\;\ \; \L_2\rvert{C_T}\simeq \O_{C_T}(\alpha_2).
\end{equation}




Notice that the right vertical $\mathrm{aj}\times\mathrm{aj}$ map in \eqref{eq:diagramDDR} is endowed with a perfect obstruction theory since it is a section of the smooth fibration $\mathfrak{Pic}_{g,n,0}\times_{\mathfrak{M}_{g,n}} \mathfrak{Pic}_{g,n,0}\to \mathfrak{M}_{g,n}$ over  $\Div_{g} \times^{fs}_ {\mathfrak{M}_{g,n}}\Div_{g}$ (see also \cite[Section~4.5]{molcho2024case}). So if $S$ is endowed with a fundamental (virtual class) we obtain by Gysin pull-back a class

\[[\rm{DDRL}(\L_1,\L_2)]^{\rm{vir}}= (\mathrm{aj}\times\mathrm{aj})^![S],\]
and thus by pushforward a class of (virtual) codimension $2g,$
 the so called  DDR-cycle in $\mathrm{CH}_{\mathrm{vdim}-2g}(S):$ 
\[\DDR(\L_1,\L_2):=\epsilon_* [\DDRL(\L_1,\L_2) ]^{\mathrm{vir}}.\]

For the same reason as before, $S\to \mathfrak{Pic}_{g,n,0}\times_{\mathfrak{M}_{g,n}}\mathfrak{Pic}_{g,n,0}$  admits an obstruction theory (see also \cite[Section~3.3.3]{ranganathan2024logarithmic}), and thus 
since $\mathfrak{M}^{\mathrm{rub}}_{\Lambda}(\Gmtrop^2)$ is log \'etale on $\mathfrak{M}_{g,n}$, and so it is in particular endowed with a fundamental cycle,  we can  also  
define a virtual class on the DDR-locus by Gysin pull-back along the horizontal map in \eqref{eq:diagramDDR}. 
By functoriality of the virtual pull-back \cite{manolache2012virtual}, this coincides with the one defined just above. 
\begin{expl}
\label{ex:rublog}
Let us take $S=\Mgnbar$  and  $\bfa,\bfb\in\mathbb Z^n$ whose entries sum to zero. This yields the DDR-cycle $\DDR(\bfa,\bfb)=\DDR(\O(\bfa),\O(\bfb))\in\operatorname{CH}(\Mgnbar)$, which can also be defined using the moduli space of rubber logarithmic maps to $\Gmlog^2$ with fixed tangency profile considered in \cite{ranganathan2024logarithmic,molcho2024case}. In the first reference these are denoted as $R_{\Lambda}(\Gmlog^2)$ where $\Lambda$ is the collection of the discrete  data $(g,\bfa,\bfb);$ in the second reference these are called toric contact loci.

\medskip

We will go back to this example in the last section.
\end{expl}

\subsection{Correlated refinement of the DDR-locus}\label{sec:correlatedDDR}

\subsubsection{Logarithmic Weil Pairing}

As recalled in Section~\ref{sec:logpic}  and references therein, for
$C/S$ log smooth, $\LogPic^0_{C/S}$ is a logarithmic abelian variety in the sense of \cite{kajiwara2008logabelian2}. 

Furthermore, it is self-dual and initial for morphism from $C/S$ to logarithmic abelian varieties $\mathcal A\to S$ (\cite{delignepair}, \cite[Section~2,3]{blommecarocci2025DR} for a more detailed discussion). 


Following \cite[Section~2,3]{blommecarocci2025DR}, one can define on the subgroup $\LogPic^0_{C/S}[r]$ of $r$-torsion logarithmic line bundles a \emph{non-degenerate} pairing, called \emph{logarithmic Weil Pairing}:
\begin{equation}
    W_r(\cdot,\cdot)\colon \LogPic^0_{C/S}[r]\times \LogPic^0_{ C/S}[r]\to \mu_r.
\end{equation}

The Weil pairing on torsion points of the Jacobian of a smooth curve can be understood in terms of the Weil reciprocity law. In the case of a smooth curve, it is given as follows.
\begin{prop}\cite[p.242]{griffiths2014principles}\cite[p.44]{serre1959groupes}\label{prop:weilclassic}

    Let $C$ be a smooth curve and consider two meromorphic functions $g_1,g_2\colon C\dashrightarrow\CC^*$ (so $\div(g_i)$ are principal) with $p_1,\dots,p_n$ their common poles and zeros. We have
    $$\prod_i (-1)^{\nu_{p_i}(g_1)\nu_{p_i}(g_2)} \left.\frac{g_1^{\nu_{p_i}(g_2)}}{g_2^{\nu_{p_i}(g_1)}}\right|_{p_i}=1.$$
\end{prop}
If $D_1$ and $D_2$ are two $r$-torsion divisors with $rD_j=\div(g_j)$, the Weil reciprocity law ensures that the following quantity only depends on the linear equivalence class of $D_1$ and $D_2$, and is a root of unity:
\begin{align}
        W_r(D_1,D_2)& :=
        \prod_p (-1)^{D_1(p)D_2(p)} \left.\frac{g_1^{D_2(p)}}{g_2^{D_1(p)}}\right|_{p} \in \operatorname{Tor}_{r}(\mathbb G_m)=\mu_{r}.
     \end{align}

  This can be extended to the logarithmic Weil pairing, by virtue of the following proposition.

  \begin{prop}\label{prop:weil-reciprocity-law}
Let $C\to S$ be a log smooth curve and let $\alpha,\beta\in H^0(S,M_C^{\gp})$ sections of the logarithmic structure (which in particular induce logarithmic trivializations).
Then 
 $$\prod_i (-1)^{s_{\bar{\alpha}(p_i)}s_{\bar{\beta}(p_i)}} \left.\frac{\alpha^{s_{\bar{\beta}(p_i)}}}{\beta^{s_{\bar{\alpha}(p_i)}}}\right|_{p_i}=1\in H^0(S,\pi_*\O_C^\star)$$
 where the product is taken over the markings $p_i\colon S\to C$ carrying logarithmic structure and $s_{\bar{\alpha}(p_i)},s_{\bar{\beta}(p_i)}$ denote the slopes of $\bar{\alpha},\bar{\beta}\in  H^0(S,\pi_*\bar{M}_C^{\gp})$  along the corresponding unbounded edges.
\end{prop}

\begin{proof}
For smooth curves, $\alpha,\beta$ are invertible functions on $C\setminus\{p_i\}$, and $\bar{M}_C^{\gp}\rvert_{p_i}=\mathbb Z$ so $s_{\bar{\alpha}(p_i)}$ (resp. $s_{\bar{\beta}(p_i)}$) is simply the order of $p_i$ as a pole/zero of $\alpha$ (resp. $\beta$), and the statement clearly reduces to the Proposition above.

\medskip

In general, the product takes values in $\Gmlog$. First, show that under the natural map $\Gmlog\to\Gmtrop$, this product is zero. 
For each marking $p_i$ we have that:

\[\bar{\alpha}(p_i)=(s_{\bar{\alpha}(p_i)}, \bar{\alpha}(v))\in \bar{M}_C^{\gp}\rvert_{p_i}=\mathbb Z\times \bar{M}_S^{\gp},\]
where the component $C_v$ indexed by the vertex $v$ is the one to which the marking $p_i$ is attached. 

We now take the image of $\frac{\alpha^{s_{\bar{\beta}(p_i)}}}{\beta^{s_{\bar{\alpha}(p_i)}}}$ in $\Gmtrop$. It is clear that for each $p_i$ the slope ($\ZZ$-component) in $\mathbb Z\times \bar{M}_S^{\gp}$ is zero;  we only need to argue that \[\sum^n _{i=1} s_{\bar{\beta}(p_i)}\bar{\alpha}(v)-s_{\bar{\alpha}(p_i)}\bar{\beta}(v)=0.\]
This is precisely the sum of the \emph{momentums} of the unbounded edges and the vanishing is the tropical Menelaus theorem proved for example in \cite[Proposition~39]{mikhalkin2017quantum}.

    \medskip
    
In order to show that this is $1$, we rewrite:
\[\prod_i  (-1)^{s_{\bar{\alpha}(p_i)}s_{\bar{\beta}(p_i)}}\left.\frac{\alpha^{s_{\bar{\beta}(p_i)}}}{\beta^{s_{\bar{\alpha}(p_i)}}}\right|_{p_i}=\prod_v\prod_{h\vdash v} (-1)^{s_{\bar{\alpha}(q_h)}s_{\bar{\beta}(q_h)}} \left.\frac{\alpha^{s_{\bar{\beta}(q_h)}}}{\beta^{s_{\bar{\alpha}(q_h)}}}\right|_{q_{h}},\]
where on the right-hand-side the first product is over the vertices $v\in V(\Gamma)$ of the tropicalization of the log curve and the second is over all half-edges $h$ incident to $v$, corresponding to either pre-images of nodes $q_e$ on the component $C_v\subseteq C $ or markings  $p_i\in C_v.$

For any node $q_e$ with $e=\{h,h'\}$, we denote by  $s_{\bar{\alpha}}(q_h),s_{\bar{\beta}}(q_h)\in\mathbb Z$  the slopes along $e$, outgoing from $v$ to which the half edge $h$ is incident.

   The equality simply follows from the fact that in  the product on the right-hand-side, for each edge $e=\{h,h'\}$ we have two contributions that cancels out since the slopes of $\bar{\alpha},\bar{\beta}$ have opposite sign at $q_h$ and $q_{h'}.$

   To conclude, we notice that looking at a fixed vertex $v$,  $\alpha,\beta$ define meromorphic functions which are invertible on $C_v\setminus\{q_h\}_{h\vdash v}$ with order of zeros and poles given by the collection of integral vectors $(s_{\bar{\alpha}(q_h)},s_{\bar{\beta}(q_h)})_{h\vdash v}.$ 
   Then by the Proposition recalled above,
   \[\prod_{h\vdash v} (-1)^{s_{\bar{\alpha}(q_h)}s_{\bar{\beta}(q_h)}} \left.\frac{\alpha^{s_{\bar{\beta}(q_h)}}}{\beta^{s_{\bar{\alpha}(q_h)}}}\right|_{q_{h}}=1.\]
\end{proof}

The proposition above gives us an effective tool to compute the logarithmic Weil pairing.

\subsubsection{Correlators for higher DR loci}
\label{subsub:corrhdr}

Let $C\to S$ be a log smooth curve and  $\L_1,\L_2$ two degree $0$ line bundles. For any $r\geqslant 1$, we can consider their powers $\L_1^{\otimes r}$ and $\L_2^{\otimes r}$. From now on, we drop the $\otimes$ to easen the notation. 

 Alternatively, one can think of having two line bundles, each endowed  with a natural root $r$th-root. Below, the main example of this set-up that we have in mind.

\begin{expl}
Consider $S=\Mgnbar$ or a logarithmic modification of it. We consider the following line bundles on the universal curve: 
$\L_1=\O(\tbfa)$ and $\L_2=\O(\tbfb)$, for $\tbfa,\tbfb\in\mathbb Z^n$ vectors whose entries sum to zero; their $r$ powers are $\O(r\tbfa)$ and $\O(r\tbfb)$.
On the other hand, given $\bfa,\bfb$ such that $r$ divides the $\gcd(a_i,b_i)_{i=1,\dots n}$, $\O(\frac{\bfa}{r})$ and $\O(\frac{\bfb}{r})$ are natural $r$th-roots of $\O(\bfa)$ and $\O(\bfb)$.
\end{expl}

\begin{expl}
    The context of \cite{blommecarocci2025DR} may be seen as a variation of this setting where $S=\bar{\M}_{g,n}(X,\beta)$, $\L_1=\O(\tbfa)$ and $\L_2=f^*L$ for $L$ some degree $0$ line bundle on $X$.
\end{expl}

 The DDR locus  $\DDRL(\L_1^r,\L_2^r)\to S$ defined above parametrizes
 log smooth curves $C_T\to T$ (for $T$ a log scheme over $S$) together with two logarithmic trivializations $\L_i^r\cong \O(\alpha_i)$ (see \eqref{eq:logtriv}).

 It follows by the modular interpretation that on the universal curve over $\DDRL(\L_1^r,\L_2^r),$ the line bundles
 $\L_1$ and $\L_2$ are logarithmically $r$-torsion, and thus define a section of $ \LogPic^0_{C/S}[r]\times \LogPic^0_{ C/S}[r]$ over the DDR locus.

From the above observation, it follows that we have a morphism:
\begin{equation}\label{eq:correlator}
\begin{array}{rcl}
   W_r\colon \DDRL(\L_1^r,\L_2^r) & \longrightarrow & \mu_r \\
    (C,\alpha_1,\alpha_2) & \longmapsto & W_r(\L_1,\L_2).
\end{array}
\end{equation}
This implies that the DDR locus is a union of open and closed sub-stack over which the value of the logarithmic Weil pairing is constant: 
\[  \DDRL(\L_1^r,\L_2^r)=\bigsqcup_{\theta\in\mu_r} \DDRL^\theta(\L_1^r,\L_2^r).\]

\begin{defi}
    We call the \emph{correlated} DDR-cycle the class in $\operatorname{CH}_*(S)\otimes \mathbb Q[\mu_r]$ given by
    \[\mathbf{DDR}(\L_1^r,\L_2^r)=\sum_{\theta\in\mu_r} \DDR^\theta(\L_1^r,\L_2^r)\cdot (\theta),\]
    where $\DDR^\theta(\L_1^r,\L_2^r)=\varepsilon_*[\DDRL^\theta(\L_1^r,\L_2^r)]$, and $(\theta)\in\QQ[\mu_r]$.
\end{defi}

\begin{expl}
In the case  $S=\Mgnbar,$ $\L_1=\O(\tbfa)$ and $\L_2=\O(\tbfb)$, we obtain a refinement of the ``standard'' DDR-cycle $\DDR(\O(\bfa),\O(\bfb))$, where $\bfa=r\tbfa$ and $\bfb=r\tbfb$.

As explained in Example~\ref{ex:rublog}, the latter can be identified with the moduli space of rubber logarithmic maps $R_{\Lambda:=(g,\bfa,\bfb)}(\Gmlog^2)$ considered in \cite{ranganathan2024logarithmic} Thus, looking at the composition  \[\M_{\Lambda}(X_{\Sigma})\to \M_{\Lambda}(\Gmlog^2)\to R_{\Lambda}(\Gmlog^2),\] 
we obtain a decomposition of the moduli stack of logarithmic stable maps to a toric surface: 
\[ \M_{\Lambda}(X_{\Sigma})=\bigsqcup_{\theta\in\mu_r}\M^\theta_{\Lambda}(X_{\Sigma}).\]
We will go back to log GW invariants of toric surfaces in the last section, where we define (in the obvious way) correlated log GW invariants and prove that these satisfy a multiple cover formula.
\end{expl}

\subsection{Logarithmic double ramification cycles}\label{sec:logDR}

\subsubsection{}The definition of double ramification cycles provided by Holmes  \cite{holmes2021extending} via resolutions of the indeterminacy of the Abel-Jacobi map  naturally yields DR classes on logarithmic blow-ups of $S$ which push-forward to $\DR(\L)\in\mathrm{CH}_*(S)$ as defined above.

\medskip

The construction goes as follow:
the map $\varphi_\L\colon S\to\mathfrak{Pic}_{g,n,0}$ ``factors'' rationally through the Jacobian $\mathfrak{Jac}$, the multidegree $\mzero$ part of $\mathfrak{Pic}_{g,n,0}$:
\[\varphi_\L\colon S\dashrightarrow\mathfrak{Jac}.\]

If we consider (the $\mathbb G_m$ rigidification), $\mathfrak{Jac}$  is a separated, smooth algebraic space (not proper)  over $S$ whose fibers are semi-abelian varieties. In particular the unit section $e\to\mathfrak{Jac}$ is a closed regular embedding and we have a well-defined class $[e]\in \mathrm{CH}^g(\mathfrak{Jac})$.

\medskip

The map $\varphi_\L$ 
is only defined on the open $\U\subseteq S$ where $\L$ has multidegree $\mzero$ on each fiber, so that we can consider the fiber product with the $0$-section $e$ of $\mathfrak{Jac}$:
\bcd 
 & \U\times_\mathfrak{Jac}e \ar[dl]\ar[r]\ar["\epsilon",d] & e \ar[d] \\
S & \U \ar[l]\ar["\varphi_\L",r] & \mathfrak{Jac}.
\ecd
However, the map $\U\times_\mathfrak{Jac}e\to S$ is rarely proper, and does not usually provide a class representing the double ramification cycle as defined before.

\medskip

By \cite{holmes2021extending}, it is possible to resolve the Abel-Jacobi map by taking an iterated logarithmic blow-up $S_\L\to S$  in such a way that the fiber product  $\widetilde{\U}\times_{\mathfrak{Jac}} e$ (where $\widetilde{\U}$ is the open in $S_{\L}$ over which the Abel-Jacobi map extend) is proper in  $S_\L$, and thus  over $S$.

\medskip

In particular, we get a class of the expected codimension in the Chow group of $S_\L$ and
 thus a class \[\operatorname{logDR}(\L)=\varphi_{\L}^{!}[e]\in \operatorname{logCH}(S).\] This class pushes forward to usual DR-class in the Chow homology of $S$.
There are  in fact a lot of choices to choose a ``good enough'' modification $S_{\L}\to S$ \cite[Section~1.5]{chiodo2024hodge} such that \[\mathrm{logDRL}(\L)=\widetilde{\U}\times_{\mathfrak{Jac}}e \hookrightarrow
S_{\L}\] is a closed embedding. We describe below a criterion and construction for such  models. 
On the other hand, we  will see that 
the open $\widetilde{\U}$ can be named without making choices.

\subsubsection{} It is natural to wonder if there exists a lift of Pixton's formula for the logarithmic DR cycle, enabling its computation.
A positive answer to this question is provided in \cite{holmes2025logDR}, using the fact that in suitable blow-ups it is possible to represent the logDR-cycle via the universal DR-cycle formula of \cite{bae2023pixton}.

\medskip

A criterion for when such a representation is possible is given in \cite{holmes2022intersection} and \cite[Proposition~47]{holmes2025logDR} using so-called \textit{almost twistable families}, recalled in Section \ref{sec-almost-twistability}.

\medskip

One way is to use stability conditions and compactified Jacobians, whose existence is guaranteed as soon as there is a section \cite[Section~4]{holmes2025logDR} (see also, for example, \cite{melo2015compactifications} and references therein). One can construct explicitly some $S_\L^\sigma$, equipped with a family of (quasi-stable) curves $C_\L^\sigma$ and a line bundle $\L^\sigma=\L(\alpha^\sigma)$ such that the universal DR-cycle for $\L^\sigma$ provides a representative of $\operatorname{logDR}(\L)$. We recall this construction below in Section \ref{sec:logblowupfromstability}. 

\subsubsection{}\label{sec-almost-twistability}

We closely follow \cite[Section 5.2]{holmes2025logDR}, 
and recall the \textit{(almost) twistability criterion}.

Let $C/S$ be a log smooth curve over a log smooth algebraic stack  endowed with a degree $0$ line bundle $\L$ and let $\mathfrak{Jac}\to S$ denote the multidegree $\mzero$ part of the relative Picard space; in particular $\mathfrak{Jac}\to S$ is finite type, smooth and separated, but not proper.

\begin{defi}\cite[Section~1.6]{holmes2022intersection}
    A family $(C/S,\L)$ is said \textit{twistable} if there exists a strict  PL function $\alpha$ on the tropicalization $\Gamma\to S$, i.e. $\alpha\in \rm H^0(S,\pi_*\bar{M}_C^{\gp})$ such that $\L(\alpha)$ has multidegree $\mzero$.
\end{defi}

By \cite[Lemma 4.8]{holmes2022intersection}, if a family is twistable, the twisted line bundle $\L(\alpha)$ provides a section $S\to\mathfrak{Jac}$, and in this situation, 
\[\DR(\L(\alpha))=\operatorname{logDR}(\L)\in\mathrm{logCH}(S).\]

However, as the authors of $[ibid]$ themselves observe, the notion of twistability is far too restrictive and does not provide a way in general to determine the $\mathrm{logDR}$ cycle using $\mathrm{DR}.$
Because of this, the weaker notion of \textit{almost twistability} is introduced:

\begin{defi}\cite[Definition 4.10]{holmes2022intersection}
    A family $(C/S,\L)$  is \textit{almost twistable} if there exists a dense open set $i\colon U\hookrightarrow S$ and a strict PL function $\alpha$ such that $\L(\alpha)$ has multidegree $\mzero$ on $U$ and such that the section $U\xrightarrow{i\times\varphi_{\L(\alpha)}}S\times_{\mathfrak{M}}\mathfrak{Jac}$ given by $\L(\alpha)$  is a closed immersion.
\end{defi}


The importance of almost twistability is explained by  the following Proposition:

\begin{prop}{\cite[Prop 44]{holmes2025logDR}}
    If $(C/S,\L)$  is almost twistable, then the DR-cycle $\DR(\L(\alpha))$ is a representative of $\operatorname{logDR}(\L)$. In particular, we can compute the cycle applying the universal DR-cycle formula: $\operatorname{logDR}(\L)=\varphi_{\L(\alpha)}^*\DR$
\end{prop}

Furthermore, as explained in \cite[Section~4.5]{holmes2022intersection} starting with any family $(C/S,\L)$ it is always possible to find a log alteration $\widetilde{S}\to S$ and a subdivision $\widetilde{C}\xrightarrow{\psi}C\times_S\widetilde{S}$ such that $(\widetilde{C}/\widetilde{S},\psi^*\L)$ is almost twistable.

\medskip

Before recalling how to explicitly construct the almost twistable model $(\widetilde{C}/\widetilde{S},\psi^*\L)$ using the theory of stability conditions, we recall a useful criterion proved in
\cite{holmes2025logDR}, providing a sufficient condition for a family to be almost twistable. 

\medskip

Let $(C/S,\L)$ be as before, We recall the definition of categories (fibered over $\mathrm{LogSch}/S$) of twists:
\begin{itemize}
    \item $\mathbf{Twist}(S,\L)=\{(T/S,\widehat{C}\to C_T,\alpha)\}$, where $T\to S$ is a morphism of log schemes, $\widehat{C}$ is a destabilization of $C_T$ and $\alpha$ a sPL (strict piece-wise linear) function on $\widehat{C}$ vanishing on the vertex carrying the first marking\footnote{This condition is imposed to fix the translation by $\bar{M}_S^{\rm{gp}}$ and thus get a representable algebraic stack};
    \item $\mathbf{Twist}^\mzero(S,\L)$ is the subcategory of \textit{trivial} twists where $\widehat{C}\to C_T$ is an isomorphism (i.e. no destabilization) and $\underline{\deg}\ \L(\alpha)=\mzero$.
\end{itemize}

\begin{rem}
    The stack over log schemes $\mathbf{Twist}(S,\L)$ generalizes the stack $\Div_g$ of Marcus-Wise, defined for $S=\mathfrak{M}_{g,n}.$ In particular a word by word repetition of \cite[Theorem~4.2.4]{marcuswiselog} proves that it is representable by an algebraic stack and log \'etale over over $S$. The same holds for  $\mathbf{Twist}^\mzero(S,\L)$ which is an open in the latter.
\end{rem}

As  explained in \cite[Section~1.4]{chiodo2024hodge}, one could define $\mathrm{LogDRL}(\L)$ as the stack over Log schemes parametrizing $(C_T/T\to S,\beta, \mathcal F\in\Pic(T))$ such that $\L(\beta)\cong\pi^*\mathcal F,$ i.e. $\L(\beta)$ is fiberwise trivial\footnote{To be precise, in \cite[Section~1.4]{chiodo2024hodge} the authors require the sPL function $\beta$ to be reduced in the sense of $[ibid,\text {Section~1.3.2}]$. We explain at the end of Remark~\ref{rem:tropDR} the difference and why they do not matter for the class.}.

Then its clear from this description that we have a closed embedding
\[\mathrm{LogDRL}(\L)\hookrightarrow \mathbf{Twist}^\mzero(S,\L), \]
and that this is in fact the pull-back of the unit section $e\hookrightarrow\mathfrak{Jac}_{C/S}$ under the natural map 
\[\mathrm{aj}\colon \mathbf{Twist}^\mzero(S,\L)\to \mathfrak{Jac}_{C/S}.\]

In other words, $\mathbf{Twist}^\mzero(S,\L),$ is the open $\widetilde{\mathcal  U}$ of the previous section.

From this description is clear that $\mathrm{LogDRL}(\L)$ is endowed, by Gysin pull-back, with a natural virtual class which we will simply denote by $\mathrm{logDR}(\L).$


\medskip

Proposition~\ref{prop:twists} give  us a really explicit way to check when some compact model of $\mathbf{Twist}^\mzero(S,\L),$ is ``good enough'' to contain $\mathrm{LogDRL}(\L)$ as a closely embedded substack. However, for some of our applications below it is sufficient the following slightly weaker criterion.

\begin{lem}\cite[Section~1.5]{chiodo2024hodge}\label{lem:fineenoughsubdivision}
   Let $\widetilde{S}\to S$ be a logarithmic modification such that the proper (by definition) morphism $p\colon \mathrm{LogDRL}(\L)\to S$ factors as $i\colon \mathrm{LogDRL}(\L)\to \widetilde{S}$ then if $i$ is \emph{strict} then it is a closed embedding and the following cycle represents the log DR class:
   \[i_*\mathrm{logDR}(\L)\in\mathrm{CH}_*(\widetilde{S}).\]
\end{lem}
Since the condition of being strict can be read off from the induced map on the tropicalization, this give a tropical point of view on the ``good enough'' model needed and yet another way to describe the open $\widetilde{\mathcal U}.$ 
In Remark~\ref{rem:tropDR} we explain this in the case of $\Mgnbar.$ The general case can be deduced by this by base change.

\begin{prop}\cite[Prop 47]{holmes2025logDR}\label{prop:twists}
    If $X\hookrightarrow\mathbf{Twist}(S,\L)$ is an open immersion containing $\mathbf{Twist}^\mzero(S,\L)$ and $X$ is separated, then $(X, \widehat{C}_X,\L(\alpha))$ is an almost twistable family, where $\widehat{C}_X$ and $\alpha$ are determined by the map from $X$ to the category of twists.
\end{prop}

In Sections \ref{sec:logblowupfromstability} and \ref{sec:logblowupfromstability2}, we provide explicit applications of this criterion.

\medskip

We will say that a logarithmic modification of $\widetilde{S}\to S$ \emph{supports a representative} for  $\mathrm{logDR}(\L)$ when one of the above criteria is satisfied. When $\widetilde{S}$ comes equipped with an explicit almost twistable family, $(\widetilde{C}/\widetilde{S},\L,\alpha)$ we will make use of the identification $\mathrm{logDR}(\L)=\mathrm{DR}(\L(\alpha)).$

\subsubsection{}\label{sec:logblowupfromstability} We consider the case of $\C_{g,n}\to\Mgnbar$ with the line bundle $\O(\bfa)$. We follow \cite[Section~4]{holmes2025logDR}. Given a \emph{small non degenerate stability condition} $\sigma$ \cite[Definition~28]{holmes2025logDR}, we get a \emph{compactified Jacobian} $\mathcal P^{\sigma}$ parametrizing quasi-stable models $C^\sigma\to C$ together with a $\sigma$-stable line bundle \cite[Definition~30]{holmes2025logDR} on $C^{\sigma}$. The latter is a logarithmically smooth algebraic stack with log structure over $\Mgnbar$ and admits a log \'etale map 
    \[\mathcal P^{\sigma}\to \LogPic^0_{\C_{g,n}/\Mgnbar}.\]
Then, we take $\bar{\M}^\sigma_{g,\bfa}$ to be the following fs fiber product:
\bcd
\bar{\M}^\sigma_{g,\bfa}\ar[d,"\rho"]\ar[r] & \mathcal P^{\sigma}\ar[d]\\
\Mgnbar \ar[r, "\varphi_{\O(\bfa)}"] & \LogPic^0_{\C_{g,n}/\Mgnbar}.
\ecd
      
From such description, $\bar{\M}^\sigma_{g,\bfa}$ comes with a universal curve $\C^\sigma_{g,\bfa}$ which is a destabilization of $\rho^*\C_{g,n}$ and a sPL-function $\alpha^\sigma$ such that the line bundle $\L^{\sigma}_{\bfa}=\O(\bfa)(\alpha^\sigma)$ is $\sigma$-stable. In particular, we have a natural map
$$\Mgasigma\to\mathbf{Twist}(\Mgnbar,\O(\bfa)),$$
determined by the quasi-stable model $\C^\sigma_{g,\bfa}\to\C_{g,n}$ and the sPL function  $\alpha^\sigma$.

We sketch the proof given in \cite[Proposition~48]{holmes2025logDR} that $\bar{\M}^\sigma_{g,\bfa}$  contains $\mathbf{Twist}^\mzero(\Mgnbar,\O(\bfa))$. Assume $\O(\bfa)(\alpha)$ has multidegree $\mzero$ for some $\alpha$. Since $\O(\bfa)(\alpha^\sigma)$ is $\sigma$-stable, it is also multidegree $\mzero$. 
Hence, by uniqueness of the $\sigma$-stable limit, $\alpha$ is actually given by $\alpha^\sigma$. 

Moreover, by [\textit{ibid}, Theorem~35], $\Mgasigma$ is separated. 
In fact, the theorem states that $\Mgasigma\to\Mgnbar$ is a log modification obtained from a subdivision. This implies that it is proper, log \'etale, relatively representable and birational since $\Mgnbar$ is smooth.

\medskip

We can apply \cite[Prop 47]{holmes2025logDR}, to conclude that
\[\varphi^*_{\L^{\sigma}_{\bfa}}\DR = \operatorname{logDR}(\L^\sigma_{\bfa}).\]
It then follows from invariance properties of the logDR-cycle that
\[\operatorname{logDR}(\L^{\sigma}_{\bfa})=\operatorname{logDR}(\O(\bfa)),\]
thus providing a formula for $\operatorname{logDR}(\bfa)$ using \cite[Prop 44]{holmes2025logDR}.


\subsubsection{}\label{sec:logblowupfromstability2}  Consider now $(\Mgasigma,\L_{\bfa}^\sigma)$. We have a map
$$\Mgasigma\to\mathbf{Twist}(\Mgasigma,\L_{\bfa}^\sigma),$$
given by the choice of the trivial sPL function. It also contains $\mathbf{Twist}^\mzero(\Mgasigma,\L_{\bfa}^\sigma)$. Indeed, assume that $\L^\sigma(\alpha)$ has multidegree $\mzero$. In particular, it is $\sigma$-stable. Since $\L^\sigma$ is already $\sigma$-stable, the divisor of $\alpha$ is the trivial one and $\alpha$ is thus constant.

\begin{rem}
\label{rem:tropDR}
 One can think of the subdivision $\Mgasigma$ as a choice for a proper model for (in the notation of \cite{molcho2024case}, it would be $\mathbf{Twist}(\Mgnbar,\L_{\bfa})$ in our notation above )
 
 \[\operatorname{tDR}_{g,\bfa} := \operatorname{DR}^{\rm{trop}}_{g,\bfa} \times_{\M_{g,n}^{\mathrm{trop}}}\Mgnbar, \]
 where $\M_{g,n}^{\mathrm{trop}}$ is the Artin stack of tropical curve constructed in \cite{cavalieri2020moduli}, which is also the tropicalization of $\Mgnbar;$ and $ \operatorname{tropDR}_{g,\bfa} $ is the tropicalization of $\DRL(\mathcal O(\bfa)).$

 By the construction as fiber product recalled above, the cones of the tropicalization of the latter are indexed by:
 \begin{itemize}
     \item the dual graph $\Gamma$ of a stable curve;
     \item a piece-wise linear function $\alpha$ vanishing on the vertex supporting the first marking, such that $\mathrm{div}(\alpha)-\bfa=\underline{0}$. 
 \end{itemize}
 Notice that the existence of such a function $\alpha$ imposes restriction of the edge lengths and thus the $\mathrm{tropDR}_{g,\bfa}$ is not a sub-complex of $\M_{g,n}^{\mathrm{trop}}$ but rather of some subdivision.

It follows from the description of the cones of the tropicalization of
$\Mgasigma$  given in \cite[Section~4]{holmes2025logDR} that the latter is such a subdivision. Indeed, the cones of the tropicalization of $\Mgasigma$ are indexed by:
\begin{itemize}
    \item a quasi stable curve $\Gamma^{\sigma}_{\bfa}\to\Gamma;$
    \item a $\sigma$-stable divisor $D_{\bfa};$
    \item a piece-wise linear function $\alpha^{\sigma}$ (vanishing on the vertex supporting the first marking) realizing the equivalence between $D_{\bfa}$ and $\underline{\deg}(\mathcal O(\bfa)).$
    The tropical $\operatorname{tropDR}_{g,\bfa}$ is the sub-complex corresponding to the cones where $\Gamma$ is not subdivided and the divisor  $D_{\bfa}$ is $\underline{0}.$
\end{itemize}
To be precise, on \cite{molcho2024case} the authors consider a refined cone structure on $ \mathrm{tropDR}_{g,\bfa}$ guaranteeing non-singularity of its Artin fan, thus the latter is rather the tropicalization of the moduli space $\widetilde{\mathrm{Rub}}(\mathcal O(\bfa))$ defined in \cite[Section~6]{bae2023pixton}.
In this case the piece-wise linear function $\alpha$ induces a total order on the vertices of $\Gamma;$ we can assume that the minimal value is zero and $\alpha$ is asked to be \emph{reduced}. As explained in \cite[Section~6]{bae2023pixton}, the classes induced b $\widetilde{\mathrm{Rub}}_g\subseteq \mathrm{Rub}_g\subseteq\Div_g$ coincide, we thus do not pay attention to the difference.
\end{rem}

\subsubsection{}\label{sec:logblowupgeneral} The construction of the logarithmic modification recalled in Section~\ref{sec:logblowupfromstability} can be generalized to any family $(C/S,\L)$ of log smooth curves over a log smooth algebraic stack endowed with at least a section $x\colon S\to C;$ this condition is necessary to ensure the existence of a small non degenerate stability condition $\sigma.$ It is in fact in this  generality that the authors in \cite[Section~4, Theorem~35]{holmes2025logDR} work.

Also Remark~\ref{rem:tropDR} can be generalised to any base $S$ substituting the tropicalization $\Sigma_S$ of $S$ to $\mathcal M_{g,n}^{\rm{trop}}$ and adjusting accordingly the combinatorial data indexing the cones. In this case $\rm{tDR}(\L)$ nothing but $\mathbf{Twist}^{\mzero}(S,\L).$
\subsection{Higher logarithmic DR}
\label{sec:higher-DR-log}

The following proposition is one of the reasons why it is preferable to work with the logDR-cycles rather than the usual DR-cycles.

\begin{prop}\cite{holmes2022intersection}
\label{prop:productformula}
    The DDR-cycle $\DDR(\L_1,\L_2)\in \operatorname{CH}(\Mgnbar)$ is the push-forward of the following cycle:
    $$\operatorname{logDDR}(\L_1,\L_2) := \operatorname{logDR}(\L_1)\cdot \operatorname{logDR}(\L_2) \in \operatorname{logCH}(\Mgnbar).$$
\end{prop}

Let us  recall how to define a locus whose (virtual) class represents the logDDR-cycle, as for the rank 1 case \cite{holmes2021extending}. 

Given $(C/S,\L_1,\L_2)$ a log smooth curve over a smooth algebraic stack with logarithmic structure, and two line bundles of total degree zero, we have a rational map
\[(\varphi_{\L_1}, \varphi_{\L_2})\colon S\dashrightarrow\mathfrak{Jac}\times_{S}\mathfrak{Jac}.\]
Applying the construction \cite[Section~3.5,5.1]{holmes2022intersection}, there is a logarithmic blow-up $S_{\L_1,\L_2}$ such that the following map is proper:
\[e_2\times_{\mathfrak{Jac}\times_S\mathfrak{Jac}} \widetilde{\U}\to S_{\L_1,\L_2} ,\]
where $\widetilde{\U}$ is the biggest open in the logarithmic modification where the map extends, and $e_2\hookrightarrow\mathfrak{Jac}\times_S\mathfrak{Jac} $ the closed embedding obtained taking the fiber product of the unit section with itself. 

\medskip

Concretely, the space $S_{\L_1,\L_2}$ can be chosen to be a common refinement of the subdivisions $S^\sigma_{\L_1}\to S,$ and $S^\sigma_{\L_2}\to S$.
The Gysin pull-back yields a class 
on $e_2\times_{\mathfrak{Jac}\times_S\mathfrak{Jac}} \widetilde{\U}$ and thus a class in $\operatorname{CH}_*(S_{\L_1,\L_2}).$ This is a representative of the logDDR-class.

\medskip

This refinement comes with two quasi-stable curves $C^\sigma_{\L_1}$ and $C^\sigma_{\L_2}$ pulled back from the respective $S^\sigma_{\L_i};$  each one  of them comes equipped with its $\sigma$-stable line bundle $\L_1^\sigma=\L_1(\alpha_1^\sigma)$ (resp. $\L_2^\sigma=\L_2(\alpha_2^\sigma)$).

\medskip

Possibly after passing to a finer logarithmic subdivision, there is a curve  $C^\sigma_{1,2}$ which is a subdivision of both $C^\sigma_1$ and $C^\sigma_2$.   It is quasi-stable in the following sense: on each fiber there can be at most two unstable $\PP^1$ adjacent to each other. Furthermore on each unstable component, the pull-back of  $\L^\sigma_1$ or $\L^\sigma_2$ has degree $1$. 

\medskip

Below, we prefer to work with the two curves $C^\sigma_1$ and $C^\sigma_2$ rather than this common model. This is not a problem since the two models have the same Jacobian, which is also isomorphic to the Jacobian of the underlying stable curve.

\medskip

Let $\DRL(\L_i^{\sigma})=S_{\L_1,\L_2}\times_{\mathfrak{Pic}_{g,n,0}}\Div_g$. We have that the fs fiber product is also the usual fiber product (see \cite[Theorem~4.6.3]{molcho2024case}):

\bcd
\DDRL(\L_1^{\sigma},\L_2^{\sigma})\ar[d]\ar[r] & \DRL(\L_1^{\sigma})\times \DRL(\L_2^{\sigma})\ar[d]\\
S_{\L_1,\L_2}\ar[r,"\Delta"]  & S_{\L_1,\L_2}\times S_{\L_1,\L_2}.
\ecd
It is endowed with the natural virtual class induced by Gysin pull-back along the diagonal and represents the logarithmic double double ramification cycle:

\[\rm{logDDR}(\L_1,\L_2)=[\DDRL(\L_1^{\sigma},\L_2^{\sigma})]^{\rm{vir}}=\Delta^!(\DR(\L1^{\sigma})\times\DR(\L_2^{\sigma}))\]

where $\DR(\L_i^{\sigma})=[S_{\L_1,\L_2}]\cap \varphi_{\L_i^{\sigma}}^*\DR^{op}.$

\medskip

Using the Weil pairing exactly as in Section~\ref{sec:correlatedDDR}, we can define a \emph{correlated} refinement of the logarithmic double ramification cycle:

\[\mathbf{logDDR}(\L^r_1,\L^r_2)=\sum_{\theta\in\mu_r}\operatorname{logDDR}^{\theta}(\L^r_1,\L^r_2)\cdot(\theta)\in\mathrm{logCH}_*(S,\mathbb Q)\otimes\mathbb Q[\mu_r].\]

\section{Multiple Cover Formula for DDR cycle}
\label{sec-MCF-DDR}
This section is dedicated to the proof of the main result of the paper.
\begin{theo}\label{theo:MCF-FFR}
    The correlated logDDR-cycle belongs to the span of the $(T_k)_{k|r}$, and satisfies the $2g$-MCF:
    $$\mathbf{logDDR}(\L_1^r,\L_2^r) = \sum_{k|r}k^{2g}\Prim_{r/k}\mathbf{logDDR}(\L_1^{r/k},\L_2^{r/k})\cdot T_{r/k}.$$
\end{theo}

Alternatively, we have this formulation free of the group algebra formalism.

\begin{coro}
    The correlated class $\operatorname{logDDR}^\theta(\L_1^r,\L_2^r)$ only depends on the order of $\theta$, and they satisfy the following multiple cover formula
    $$\operatorname{logDDR}^\theta(\L_1^r,\L_2^r) = \sum_{k|r,\theta}k^{2g}\operatorname{logDDR}^{\theta_\mathrm{prim}}(\L_1^{r/k},\L_2^{r/k}).$$
\end{coro}

\begin{proof}
    One gets the class formulation from the group algebra formulation taking the $(\theta)$-coefficient in the group algebra. Saying that the $(\theta)$-coefficient only depends on the order is equivalent to saying that the group algebra valued logDDR belongs to the span of the $(T_k)_{k|r}$. We also have that
    $$\Prim_r\mathbf{logDDR}(\L_1^r,\L_2^r) = r^2\cdot\operatorname{logDDR}^{\theta_\mathrm{prim}}(\L_1^r,\L_2^r).$$
    Conversely, the group algebra formulation is obtained from the class formulation summing over all $\theta$ and grouping them by order.
\end{proof}

Pushing forward, we also get the same relation between the DDR-cycles.

\begin{coro}\label{theo:MCF-FFR}
    The correlated DDR-cycle satisfies the $2g$-MCF:
    $$\mathbf{DDR}(\L_1^r,\L_2^r) = \sum_{k|r}k^{2g}\Prim_{r/k}\mathbf{DDR}(\L_1^{r/k},\L_2^{r/k})T_{r/k}.$$
\end{coro}

\subsection{Moduli of roots}

Let $C/ S$ be a smooth log curve over smooth algebraic stack with log structure with at least one section; let $\L_1$ and $\L_2$ be two degree $0$ line bundles, $\sigma$ a small non-degenerate stability condition, and $r\geqslant 1$. 

\medskip

As explained in Section~\ref{sec:higher-DR-log}, we can find a logarithmic blow-up \[S_{\L_1^r,\L_2^r}\to S,\]
depending on the line bundles $\L_1^r$ and $\L_2^r$ and on the choice of stability $\sigma,$ such that on $S_{\L_1^r,\L_2^r}$, we have:  two quasi-stable curves $C_{r,1}^\sigma$ and $C_{r,2}^\sigma$; the curve $C_{r,j}^\sigma$ comes equipped with a $\sigma$-stable line bundle $(\L_j^r)^\sigma=\L_j^r(\alpha_j^\sigma)$. The class $\mathrm{logDR}(\L_j^r)\in\mathrm{logCH}_*(S)$ is represented by $\mathrm{DR}(\L_j^r(\alpha_j^\sigma))\in \mathrm{CH}_*(S_{\L_1^r,\L_2^r})$ and the latter can be computed applying the universal DR-formula to the line bundle
 $(\L_j^r)^\sigma.$

 \medskip

To lighten the notation, from now on, we denote
\[S_r:= S_{\L_1^r,\L_2^r}.\]


Following \cite{holmes2023root}, as recalled in \ref{sec:logroots}, we can construct the moduli space $S_r\oner$ parametrizing log-$r$-roots of $(\L_1^r)^\sigma$ and $(\L_2^r)^\sigma$ as the following fiber product over fs log schemes:
\bcd
S_r\oner \ar[r]\ar[d] & \LogPic(C^\sigma_{r,1}/S_r)\times_{S_r}\LogPic(C^\sigma_{r,2}/S_r) \ar["r\cdot{,}r\cdot",d] \\
S_r \ar[r] & \LogPic(C^\sigma_{r,1}/S_r)\times_{S_r}\LogPic(C^\sigma_{r,2}/S_r),
\ecd
where the bottom map is the one given by  $(\varphi_{(\L_1^r)^\sigma},\varphi_{(\L_2^r)^\sigma})$.

\medskip

The \emph{spin moduli space} $S_r\oner$ comes endowed with the two universal log-roots $\R_1$ and $\R_2$ satisfying the following equality of logarithmic line bundles:
\begin{equation}\label{eq:logrootsrelation}
\R_1^{\otimes r}\equiv(\L_1^r)^\sigma\equiv\L_1^r,\quad\text{and}\quad \R_2^{\otimes r}\equiv(\L_2^r)^\sigma\equiv\L_2^r.
\end{equation}

Following \cite[Section~4]{holmes2023root}, as recalled in Theorem~\ref{thm:logrotsafterbasechange} after taking a base change to a certain root stack along the boundary
$\tS_r\to S_r$ we have that:

\begin{itemize}[leftmargin=0.4cm]
    \item The stack $\tS_r\oner$ is a (\'etale finite) torsor under the \'etale 
    group scheme $\tS_r(\O^{1/r})\times_{\tS_r} \tS_r(\O^{1/r})$,
    with a natural section $(\L_1,\L_2)$ trivializing the torsor.
    \item After resolving the singularities of the total space  of the pull-back along the root stack of the quasi-stable curves\footnote{Recall this  is obtained inserting a length $r-1$ chain of exceptional components $\mathbb P^1$ at each node, i.e., tropically, subdividing each edge $\delta$ times.} $C_{r,j}^\sigma$,
    there exists an honest line bundle on the subdivided curve $\widetilde{C}^\sigma_{r,j}$ representing the universal logarithmic root, also denoted by $\R_j$.
\end{itemize}


The line bundle $\R_j$ depends on the choice of a  tropical $r$-torsion divisor on each cone of the tropicalization of $\widetilde{C}^\sigma_{r,j}\to\tS_r\oner$, compatible under generalization maps, i.e. edge contractions.

We choose the representatives induced by the tropical torsion divisors constructed in Definition~\ref{def:torsiontropical}. The multidegree of $\R_j$ then satisfies
\[r\cdot\underline{\deg}\ \R_j=\underline{\deg}\ (\L_j^r)^\sigma+\mathrm{div}(\tau_j),\]
where $\tau_j$ is the sPL function on the tropicalization $\widetilde{\Gamma}_{r,j}$ of  $\widetilde{C}^\sigma_{r,j}$ with value $0$ at the vertices of $\Gamma_j$, the tropicalization of $C^\sigma_{r,j}$,  constructed in Lemma~\ref{lem:plfunction}.

\begin{rem}
Recall that the slope of $\tau_j$ along the two half edges $e_h, e_{h'}$ obtained subdividing of the edge $e=\{h,h'\}$
are opposite modulo $r$ and that when the slope is non-zero, $\tau_j$  has a unique bending point $u$ inside the edge with $\div(\tau_j)(u)=r$.
\end{rem}

Equation~\eqref{eq:logrootsrelation} for the log roots becomes the following identity of line bundles on $\widetilde{C}_{r,j}$:
\[\R_j^r = (\L_j^r)^\sigma(\tau_j) = \L_j^r(\alpha_j+\tau_j).\]

\begin{rem}
Notice that our representative $\R_j$ is actually pulled back from a coarser subdivision of $C_{r,j}^\sigma$, where along each edge there is at most one exceptional curve, corresponding to the vertex $[u]$. From now on, we denote by $\widetilde{C}_{r,j}^\sigma$ this minimal subdivision over which $\R_j$ is defined. 
\end{rem}


    \subsubsection{Almost Twistability}


We denote by $\DR(\R_j)$ the class in $\mathrm{CH}_*(\tS_r\oner,\QQ) $ obtained applying the universal DR cycle formula \cite{bae2023pixton} for the morphism

\[ \widetilde{S}_r\oner\xrightarrow{\varphi_{\R_j}} \mathfrak{Pic}_{g,n,0},\]
and $\operatorname{logDR}(\R_j)\in  \mathrm{logCH}_*(\widetilde{S}_r\oner,\QQ) $  the logDR-class defined as recalled in Section \ref{sec:logDR}.
Notice that by our choice of representative $\R_j,$ $\DR(\R_j)$ is the class of the \emph{spin DR-cycle} as introduced in \cite{holmes2023root}.

\begin{prop}\label{prop:almosttwistable}
    The family $(\widetilde{C}_{r,j}^\sigma\to \widetilde{S}_r\oner,\R_j,0)$ is almost twistable. In particular, $\DR(\R_j)$ provides a representative for $\operatorname{logDR}(\R_j)$.
\end{prop}

\begin{proof}
    We use the almost twistability criterion: the family  $(\widetilde{C}_{r,j}\to \tS_r\oner,\R_j,0)$ is the data of a morphism
    \[\tS_r\oner\to\mathbf{Twist}(\tS_r\oner,\R_j).\]

    Furthermore,  assuming that $S$ is a separated  algebraic stack with log structure (e.g. for the case of $\Mgnbar$), so is $\tS_r\oner$ by construction. 

To show that the family is almost twistable, we only need to argue  that  $\mathbf{Twist}^{\mzero}(\tS_r\oner,\R_j)$ is contained in $\widetilde{S}_r$ as an open.  
A $T$-point of $\mathbf{Twist}^\mzero(\widetilde{S}_r\oner,\R_j)$ is the data of a strict piece-wise linear function $\alpha_T$ on  $(\widetilde{C}_{r,j}^\sigma)_T$ (the pull-back to $T$ of the minimal subdivision over which the line bundle $\R_j$ is defined) satisfying $\underline{\deg}\ \R_j(\alpha_T)=\mzero$ on each fiber of $(\widetilde{C}_{r,j}^\sigma)_T\to T.$

    If this is the case, the associated tropical line bundle is trivial:
    \[\mathrm{trop}( \R_j(\alpha_T))= \mathrm{trop}( \R_j)= 0_{\widetilde{S}_r}\in\mathrm{TroPic}^0_{\widetilde{C}_{r,j}^\sigma/\widetilde{S}_r}(T).\]

    
    Denoting by $[-]$ the equivalence class of a tropical divisor, we have that  the  $r$th-power $\R_j^r$ satisfies
    \[0_{\widetilde{S}_r} = \mathrm{trop}(\R_j^r)=\mathrm{trop}((\L_j^r)^\sigma)= [\underline{\deg}(\L_j^r)^\sigma]\in\mathrm{TroPic}^0_{\widetilde{C}_{r,j}^\sigma/\widetilde{S}_r}(T).\]

    In particular, this means that there is a sPL function $\gamma_T$ on $(\widetilde{C}_{r,j}^\sigma)_T$ such that
     $(\L_j^r)^{\sigma}(\gamma_T)$ has multi-degree $\mzero$. 
     
     Since 
    $(\L_j^r)^\sigma$ is pulled back from $(C_{r,j}^\sigma)_T,$  the sPL $\gamma_T$ does not bend along the edges and thus it is also pulled-back from a sPL function on $(C_{r,j}^\sigma)_T$.

Furthermore, as $(\L_j^r)^\sigma$  is $\sigma$-stable for a small non degenerate stability condition, the uniqueness of the stable limit guarantees, precisely as in \cite[Proposition~48]{holmes2025logDR} that $\gamma_T=0$ and $(\L_j^r)^\sigma$ has multidegree $\mzero$.

Finally, by Proposition~\ref{prop:constructroots},  the representative $\R_j$  associated to the trivial tropical line bundle $0_{\widetilde{S}_r}$ is the unique one induced by the tropical divisor $\mathcal D_0$ for the trivial weighting: 
this is supported on $(C_{r,j}^\sigma)_T$ (not on a subdivision) and  is already multi-degree $\mzero$.

Remembering that the only sPL function on a curve with vertical log structure which are balanced and thus preserve the multi-degree are the constants we conclude $\alpha_T$ is indeed zero.
    

    
\end{proof}

    \subsubsection{Splitting $\widetilde{S}_r\oner$  according to order}

Twisting the universal log-root $\R_j$ of $\L_j^r$ by the preferred root $\L_j$, we get two $r$-torsion log-line bundles:
\[\T_1=\R_1\otimes\L_1^{-1},\quad \text{and}\quad \T_2=\R_2\otimes\L_2^{-1}.\]

These defines sections of the finite \'etale group scheme (see Theorem~\ref{thm:logrotsafterbasechange}) $\widetilde{S}_r(\O^{1/r})\times_{\widetilde{S}_r} \widetilde{S}_r(\O^{1/r})$ over 
 $\widetilde{S}_r\oner$. 

\medskip

Given $k|r$ we consider the locus in $\widetilde{S}_r\oner$, 
where $\mathcal T_1,\mathcal T_2$ have order at most $\frac{r}{k}$, namely the locus where as sections of  $\widetilde{S}_r(\O^{1/r})\times_{\widetilde{S}_r} \widetilde{S}(\O^{1/r})$,
\[(\T_1^{r/k})\equiv (\T_2^{r/k})\equiv\O.\]
Since we are looking at sections of an \'etale group over  $\widetilde{S}_r\oner$, the locus where a finite number of them agree is open and closed.

We denote by 
\[\widetilde{S}^k_r\oner\subseteq \widetilde{S}_r\oner\]
the open and closed sub-stack where $\T_1$ and $\T_2$ have order at most $\frac{r}{k}$ in the sense above.




\begin{prop}
\label{prop:change-order}
    The open and closed sub-stack $\widetilde{S}_r^k\oner$  is isomorphic to the (base change along the root stack $\widetilde{S}_r\to S_r$) of the $(r/k)$ spin moduli space $S_{r}(\frac{1}{r/k}, \frac{1}{r/k})$ parametrizing log-$(r/k)$-roots of  the line bundles $(\L_1^{r/k},\L_2^{r/k})$,  over $C_{r,1}^\sigma\to S_r$ and $C_{r,2}^\sigma\to S_r$ respectively. 
    
\end{prop}

\begin{proof}
We have the following:
    \begin{itemize}
        \item On $ \tS_r^k\oner$ we have that the restriction of the universal roots satisfy  $\R_j^{r/k}\equiv\L_j^{r/k}$ and thus determine a morphism $\tS_r^k\oner\to \tS_r\times_{\tS_r}S_{r}(\frac{1}{r/k}, \frac{1}{r/k})$

        \item On the other hand, on the universal curve  (does not matter if the singular one of its resolution for the log Picard group) over $\tS_r\times_{\tS_r}S_{r}(\frac{1}{r/k}, \frac{1}{r/k})$ the universal $r/k$ root $\R'_j$  satisfies $(\R'_j)^{r/k}\equiv\L_j^{r/k}$ which implies $(\R'_j)^{r}\equiv\L_j^{r}$ and thus we have a morphism $\tS_r\times_{\tS_r}S_{r}(\frac{1}{r/k}, \frac{1}{r/k})\to \tS_r^k\oner.$
    \end{itemize}
    These are clearly mutually inverse.
\end{proof}
\begin{defi}
    We will denote by \[\mathrm{logDDR}^k(\R_1,\R_2)=\DDR^k(\R_1,\R_2)\] the restriction of the logarithmic Spin DDR  class to the open and closed sub-stack $\widetilde{S}_r^k\oner.$
 \end{defi}

Since almost twistability is preserved by restricting to a open and closed sub-stack, by Proposition~\ref{prop:almosttwistable} and \ref{prop:change-order},  the class $\mathrm{logDDR}^k(\R_1,\R_2)$ is a representative for $\mathrm{logDDR}(\R'_1,\R'_2)$ where $\R'_1,\R'_2$ are the universal log-$(r/k)$-roots of  the line bundles $(\L_1^{r/k},\L_2^{r/k})$, hence the notation.

\medskip

We will show below that via the forgetful map $\tS_r\oner\to \tS_r$ the cycle
$\mathrm{logDR}^k(\R_1,\R_2)$ pushes forward to  $\mathrm{logDR}^k(\L^{r/k}_1,\L^{r/k}_2).$

\begin{rem}
    Notice that, following \cite{holmes2023root,chiodo2008towards}, it would be more natural to take the $(r/k)$-root stack to get a model over which the $(r/k)$ log-roots are representable and the ramification of $\LogPic^0[r/k]\to S$ have been resolved. For our purposes, it is however more convenient to work over the $r$-root stack.
\end{rem}


\subsubsection{Splitting by Weil pairing}

On each $\tS^k_r\oner$ we can consider the logarithmic Weil pairing $W_{r/k}$ naturally defined on the group of $(r/k)$-torsion log-line bundles. 

We then see that the moduli space is a union of open and closed components indexed by the value of $W_{r/k}(\T_1,\T_2)$. We write:
\[\tS^k_r\oner=\bigsqcup \tS_r^{k,\theta}\oner\]
where the union is over $\theta\in\mu_{r/k}$ and $\tS_r^{k,\theta}\oner$ is the component over which the $W_{r/k}$ is constant. In particular, the moduli stack $\tS_r\oner$ splits in open and closed sub-stacks indexed by the values of $W_r(\T_1,\T_2)\in\mu_r$.

\medskip

If $k>1$ divides $r$, notice that, it follows from the definition of the Weil pairing, that $W_{r/k}$ is not the restriction of  $W_r$ since 
via the canonical embedding $\mu_r\hookrightarrow\RR/\ZZ$,
we have the following relation in $\RR/\ZZ$:
    $$W_r(\T_1,\T_2) = k\cdot W_{r/k}(\T_1,\T_2).$$
We notice that $W_r$ fully determines $W_{r/k}$ since the subgroup of $(r/k)$-torsion elements of $\mathbb Q/\mathbb Z$ is isomorphic to a quotient of $r$-torsion via the multiplication map by $k$. More concretely, 
we can write a $(r/k)$-torsion element $\T_j$ as $\T_j=\U_j^k$ for some $r$-torsion log-line bundle $\U_j;$  then we have
$$W_{r/k}(\T_1,\T_2) = k\cdot W_r(\U_1,\U_2) \in (\RR/\ZZ)[r/k]\simeq\mu_{r/k}.$$
The right hand-side is indeed well-defined: $\U_1$ and $\U_2$ are well defined up to $k$-torsion, and thus $kW_r(-,-)$ vanishes whenever one of the two log-line bundles is of $k$-torsion.

\subsubsection{Stratification and degree computation}
    \label{sec-stratification-degree-computation}

The proof of the multiple cover formula will essentially follow from computing the degree of the projection from $\widetilde{S}_r(\frac{1}{r},\frac{1}{r})\to S_r$ on each strata of the logarithmic  stratification of the stack $S_r$, i.e. the strata where the fiber of the sheaf of monoids is constant.
The latter correspond to cones in the Artin fan (see for example \cite{carocci2026tropical} or \cite{pandharipande2025logarithmic}) and can thus be described tropically. 
\begin{rem}
    Since all the stacks $\mathcal S$ with log structure we are working with are logarithmically smooth over the trivial log point, and thus smooth over their Artin fan $\mathcal S\to\Sigma_{\mathcal S},$ \cite[Theorem~4.6]{olsson2003logarithmic} the closure of each logarithmic strata $\mathcal S_{\tau}$ (for $\tau\to\Sigma_{\mathcal S}$ a cone) is equidimensional of dimension $\dim \mathcal S-\text{rk}(\bar{M}^{\gp}_{\mathcal S,\eta})$ for $\eta$ a generic point of $\mathcal S_{\tau}.$ We  will denote by $[\mathcal S_{\tau}]\in \mathrm{CH}_*(\mathcal S)$ the corresponding class.
\end{rem}

\medskip

\paragraph{\bf{Cones of the tropicalization and strata of $S_r$}}
Following the construction of the log modification $S_{\L}^\sigma$ given in \cite[Section~4]{holmes2025logDR} and in particular the proof of [\textit{ibid}, Theorem~35], we see that
 the cones  of  $S_r$ are indexed by the following data:
\begin{itemize}
    \item two quasi stable subdivisions $\Gamma^{\sigma}_{r,j}\to\Gamma,\; j=1,2$
    \item two $\sigma$-stable divisors $D_{\L^r_j}\; j=1,2$ supported on the vertices of $\Gamma^{\sigma}_{r,j};$
    \item the slopes of the two strict piece-wise linear functions $\alpha_j^\sigma,$ realizing the linear equivalence  between $D_{\L^r_j}$ and $\underline{\deg}(\L_j^r).$
\end{itemize}
To simplify the notation, we will denote by $\Lambda$ the collection of all these discrete data, and  by $\tau_{\Lambda}$ the corresponding cones. The associated closed stratum is denoted by $S_{r,\Lambda}$. The cone $\tau_\Lambda$ is the sub-cone of the cone $\mathbb R_{\geq 0}^{E(\Gamma^{\sigma}_{r,1})}\times_{\mathbb R_{\geq 0}^{E(\Gamma) }}\mathbb  R_{\geq 0}^{E(\Gamma^{\sigma}_{r,2})}\times_{\mathbb R_{\geq 0}^{E(\Gamma) }}\sigma_{\Gamma}$
cut out by the condition ensuring that $\alpha_j^\sigma,$ are strict piece-wise linear function on the quasi stable model. The maps $\mathbb R_{\geq 0}^{E(\Gamma^{\sigma}_{r,j})}\to\mathbb R_{\geq 0}^{E(\Gamma) }$ are the natural ones defined by $\ell(e)=\widetilde{\ell}(e_1)+ \widetilde{\ell}(e_2)$ when $e$ is subdivided.


\medskip

\paragraph{\bf{Cones of the tropicalization and strata of the root stack $\tS_r$}}
Starting from the description of the cones for the Artin fan $\Sigma_r$ of $S_r$ we can also obtain a description of the logarithmic strata for the root stack $\tS_r$, as done in \cite{holmes2023root}:
if  $\tSigma_r$ denotes the tropicalization,  its cones are indexed by the same combinatorial data $\Lambda$, but they come with a different lattice structure since having taken a root stack along the boundary, each integral generator admits an $r$-root. The corresponding stratum is denoted by $\tS_{r,\Lambda}$.

\medskip

\paragraph{\bf{Cones of the tropicalization and strata of $\tS_r\oner$}}

Finally, we describe the cones 
$\tSigma_r\oner$ the tropicalization of $\tS_r\oner$. Following \cite{holmes2023root}, its cones (corresponding to the logarithmic strata of $\tS_r\oner$) are indexed by:


\begin{itemize}
    \item The combinatorial data $\Lambda$ from before.
    \item Two weightings modulo $r,$ $\phi_1$ and $\phi_2$, with $\phi_j$ compatible with $\underline{\deg}(\L_j^r)^\sigma \mod r$ (see Section~\ref{sec:logroots}). These correspond to the choice of line bundle representing the logarithmic root as recalled above.
   Given the weighting $\phi_j$, 
    we speak of the associated torsion tropical divisor, which  is the multidegree of $\T_j$ and is given by
$$T_j=\frac{\underline{\deg}\ (\L_j^r)^\sigma +\div\phi_j}{r}-\underline{\deg}\ \L_j.$$

\end{itemize}


Notice that over a fixed cone of $\tSigma_r$ we have exactly $(r^{b_1(\Gamma)})^2=r^{2b_1(\Gamma)}$ cones of $\tSigma_r$ , indexed by the choices of $r$-tropical roots. In other words, the restriction of $\tS_r\oner\to \tS_r$ over each stratum splits according to the possible tropicalizations of the roots. 

We denote the closed stratum corresponding to $(\Lambda,\phi_1,\phi_2)$ by $\tS_{r,\Lambda,\phi_1,\phi_2}\oner.$

\begin{prop}
    We have the following degree computations:
\begin{enumerate}
    \item The degree of the map $\tS_r\oner\to \tS_r$ is $(r^{2g})^2$.
    \item The degree of the map $\tS_{r,\Lambda,\phi_1,\phi_2}\oner\to \tS_{r,\Lambda}$ is $(r^{2g-b_1(\Gamma)})^2$.
\end{enumerate}
\end{prop}

\begin{proof}
\begin{enumerate}
    \item There are exactly $r^{2g}$ (log-)roots to a given (log-)line bundle; in fact, since we are working on a suitable root stack, the morphism is in fact \'etale of the given degree (see Theorem~\ref{thm:logrotsafterbasechange}).
    
    \item There is a short exact sequence of groups over Log schemes over $\tS_r\oner$:
    \[1\to\Pic^{[0]}_{\widetilde{C}_r/\tS_r\oner}[r]\to\LogPic_{\widetilde{C}_r/\tS_r\oner}[r]\to \mathrm{TroPic}^0_{\widetilde{C}_r/\tS_r\oner}[r]\to 1.\]
    Thus, the log-roots are spread uniformly among the possible tropicalizations.
   This implies that  in particular that the degree splits uniformly among the $r^{2b_1(\Gamma)}$ strata $\tS_{r,\Lambda,\phi_1,\phi_2}\oner$ over a given  $\tS_{r,\Lambda}.$
\end{enumerate}
\end{proof}

The stratification of $\tS_r\oner$ induces, simply by restriction along the open immersion,  a stratification of each sub-stack $\tS_r^{k,\theta}\oner$. 

However, the degrees of projection restricted to a component  $\tS_r^{k,\theta}\oner$ now depend on the correlator and thus needs to be computed with a more fine analysis.
That is the content of following Proposition~\ref{prop:degreecomputation} below. 

\begin{defi}
Let $d_{r,\Lambda,\phi_1,\phi_2}^{k,\theta}$ be the degree of the map $\tS_{r,\Lambda,\phi_1,\phi_2}^{k,\theta}\oner\to \tS_{r,\Lambda}$. We define the correlated degree as the following group algebra element:
$$\mathbf{d}_{r,\Lambda,\phi_1,\phi_2}^k = \sum_{(r/k)\theta\equiv 0} d_{r,\Lambda,\phi_1,\phi_2}^{k,\theta}\cdot(\theta).$$
\end{defi}

\begin{prop}\label{prop:degreecomputation}
    The correlated degree has the following expression:
    $$\mathbf{d}_{r,\Lambda,\phi_1,\phi_2}^k = \bfG_{b_1(\Gamma),2g}(\frac{r}{k},\omega_1,\omega_2),$$
    where $\omega_j$ is the order of $T_j$ (the multidegree of $\T_j$), and $b_1(\Gamma)$ the genus of the graph associated to the stratum.
\end{prop}

\begin{proof}

    Up to changing the torsion order we consider, we can assume that $k=1$.  
     We have the two following short exact sequences:
    \[1\to\Pic^{[0]}_{\widetilde{C}_r/\tS_r\oner}[r]\to\LogPic_{\widetilde{C}_r/\tS_r\oner}[r]\to \operatorname{TroPic}^0_{\widetilde{C}_r/\tS_r\oner}[r]\to 1,\]
    \[1\to H^1(\Gamma,\mu_r)\to\Pic^{[0]}_{\widetilde{C}_r/\tS_{r,\Lambda,\phi_1,\phi_2}\oner}[r]\to \prod_v\operatorname{Pic}^0_{(\widetilde{C}_r)_v}[r]\to 1,\]
    where $(\widetilde{C}_r)_v$ are the components of the simultaneus normalization of $\widetilde{C}_r$ over the stratum.
    

   Notice that $\tS_{r,\Lambda,\phi_1,\phi_2}^{k,\theta}\oner\to \tS_{r,\Lambda}$ is in fact \'etale since it is obtained by base change from the morphism

   \[\tS_{r}^{k,\theta}\oner\hookrightarrow\tS_{r}^{k}\oner\hookrightarrow \tS_{r}\oner\to\tS_{r};\]
   we explained before in the section that the first two morphisms are open immersion and the last one is in fact \'etale.
   
   Therefore, to compute the degree of this flat morphism, we can work locally around a geometric point of the given strata.

In particular, since we are dealing with $r$-torsion groups, both short exact sequences split and we get
\[\LogPic_{\widetilde{C}_r}[r]\cong \operatorname{TroPic}^0_{\widetilde{C}_r}[r]\oplus H^1(\Gamma,\mu_r)\oplus\prod_v\Pic^0_{(\widetilde{C}_r)_v}[r].\]
Furthermore, this splitting is compatible with the Weil pairing in the sense that the first two factors are dual with each other with respect to the pairing, and the last factor is orthogonal to the two first, with its induced Weil pairing, which is the classical for Jacobians (see \cite[Prop 2.13]{blommecarocci2025DR}).

Having reduced the degree computation around a point, the computation of the  degree amounts to count pairs of torsion log-line bundles with fixed tropicalization that have a prescribed Weil pairing, which is exactly what has been computed in Proposition \ref{prop:expression_G}, taking $L= H^1(\Gamma,\mu_r),$  $L^*=\operatorname{TroPic}^0_{\widetilde{C}_r}[r]$ and $W=\prod_v\Pic^0_{(\widetilde{C}_r)_v}[r].$
\end{proof}

In particular, the correlated degree over the stratum of smooth curves, i.e. the family of degrees of $\tS^\theta_r\oner\to \tS_r$ is given by $\bfF_{2g}(r)$.

\subsection{Relation between SpinDR and DR}
Recall that the almost twistability of $\L_1^r$ and $\L_2^r$ on the stack $\tS_r$ implies that $\tS_r$ supports representatives of the cycles $\operatorname{logDR}(\L_1^r)$ and $\operatorname{logDR}(\L_2^r)$; in fact these can be computed explicitly using the twisting functions. We now want to argue that for every $k$ dividing $r$ the stack $\tS_r$ supports representatives of the cycles $\operatorname{logDR}(\L_1^{r/k})$ and $\operatorname{logDR}(\L_2^{r/k}).$ To do so, we employ the weaker criterion of Lemma \ref{lem:fineenoughsubdivision}: Consider the following factorization:
\[
 \begin{tikzcd}
 \mathbf{Twist}^\mzero(\tS_r,\L_i^{r/k})\arrow[r,"j_k"] \arrow[rd, "g" ] & \tS_r\\    
 & \mathbf{Twist}^\mzero(\tS_r,\L_i^r) \arrow[u,"j_1"],
 \end{tikzcd}
\]
where the map $g$ just scales the PL-function by the factor $k$. We notice that both $j_1$ and $g$ are strict, thus so is the composition. Indeed, the tropicalization of the moduli stack of degree zero twists is by definition the tropical DR considered in Remark~\ref{rem:tropDR}, and from the description of the cone structure for the tropicalization of $\tS_r$ the fact that $\Sigma_{\mathbf{Twist}^\mzero(\tS_r,\L_i^r)}$ is a sub-complex is immediate; same for  $\Sigma_{\mathbf{Twist}^\mzero(\tS_r,\L_i^{r/k})}\subseteq  \Sigma_{\mathbf{Twist}^\mzero(\tS_r,\L_i^r)}$

Therefore we conclude that the pushforward
\[
i_*\vir{\operatorname{logDRL}(\L_i^{r/k})}
\]
is a representative of $\operatorname{logDR}(\L_i^{r/k})$.
\begin{remark}
Note that its not clear to the authors if $\L_i^{r/k}$ is almost-twistable for $k>1$ without modifying $\tS_r$ further, namely it is not clear whether it is possible to extend the twisting function on the tropical curve over $\Sigma_{\mathbf{Twist}^\mzero(\tS_r,\L_i^{r/k})}$ to a suitable destabilization of the tropical curve over $\Sigma_{\tS_r}.$
\end{remark}
We will write  $\mathrm{logDDR}(\L_1^{r/k},\L_2^{r/k})$ for the product of the two aforementioned representatives of $\mathrm{logDR}(\L_i^{r/k})$. Recall by Proposition \ref{prop:productformula}
\[
\mathrm{logDDR}(\L_1^{r/k},\L_2^{r/k})=i_*\vir{\operatorname{logDRL}(\L_1^{r/k},\L_2^{r/k})}
\]
The next proposition we explain how these cycles are related



\medskip

\begin{prop}\label{prop:push-forward-compatibility}
    Under the map $\nu\colon \tS_r\oner\to\tS_r$, we have the following:
    \begin{enumerate}
        \item The cycle $\operatorname{logDDR}^k(\R_{1},\R_{2})$ pushes forward to $\operatorname{logDDR}(\L_1^{r/k},\L_2^{r/k})$.
        \item The cycle $\operatorname{logDDR}^{k,\theta}(\R_{1},\R_{2})$ pushes forward to $\operatorname{logDDR}^\theta(\L_1^{r/k},\L_2^{{r/k}})$.
    \end{enumerate}
\end{prop}

\begin{proof}

The rank $1$ analogue of (1) in the case $k=1$ is due to  \cite{chiodo2024double}.

Let $U_k=\mathbf{Twist}^\mzero(\tS_r,\L_1^{r/k})\cap\mathbf{Twist}^\mzero(\tS_r,\L_2^{r/k})$ be the open where both $\L_1^{r/k}$ and $\L_2^{r/k}$ are tropically trivial. 
Let $\alpha_j$ denote the PL function on the restriction of $C^\sigma_{r,j}$ to $U_k$ such that $\L_j^{r/k}(\alpha_j)$ is multidegree $0$.

\medskip

Let also $U^\mathrm{spin}_k=\mathbf{Twist}^\mzero(\tS^k_r\oner,\R_1)\cap\mathbf{Twist}^\mzero(\tS^k_r\oner,\R_2)$ be the open subset of $\tS_r^k(\frac{1}{r},\frac{1}{r})$ where both $\R_1$ and $\R_2$ are tropically trivial.

\medskip

We use once again that when $(\L_j^r)^\sigma$ has multidegree $\mzero$, the construction of the representatives $\R_j$ for the log-roots has been made in such a way that when also the tropicalization of  $\R_j$ is trivial, then the $\R_j$ are exactly the roots of $(\L_j^r)^\sigma$ as a honest line bundle and no subdivision of the quasi-stable model as been performed. 

\medskip

From the above we can then deduce the following after restricting to the open and closed component $\tS_r^k(\frac{1}{r},\frac{1}{r})$: When $\L_j^{{r/k}}$ has trivial tropicalization, then the $\R_j$ with trivial tropicalization are exactly the $k$-th roots of $\L_j^{{r/k}}(\alpha_j)$ as a honest line bundle.

\medskip

We thus get that the front face of the following commutative diagram is a cartesian square. 
Let $C_j$ denote the restriction of $C^\sigma_{r,j}$ to $U_k$ and similarly we write $\widetilde{C}_j$ for the restriction of $\widetilde{C}^\sigma_{r,j}$ to $U^{\mathrm{spin}}_k$.

\bcd
 & U_k  \arrow[ld,"e"] \arrow[dd, equal] &   & \operatorname{logDRL}^k(\R_{1},\R_{2}) \arrow[ld] \arrow[ll] \arrow[dd] \\
\Jac_{C_1 / U_k}\times\Jac_{C_2 / U_k}
\arrow[dd,"\otimes \frac{r}{k}"]  &                                     & U^\mathrm{spin}_k \arrow[dd, "p", pos=0.2] \arrow[ll, "\tilde{\mathtt{aj}}_{\R}", pos=0.2,crossing over]  & \\
& U_k  \arrow[ld,"e"]            &  & \operatorname{logDRL}(\L_1^{r/k},\L_2^{r/k}), \arrow[ll] \arrow[ld]            \\
\Jac_{C_1 / U_k}\times\Jac_{C_2 / U_k}            &         & U_k \arrow[ll,"\mathtt{aj}_{\L}"] \arrow[from=uu,crossing over]           &        
\ecd
Here we denoted by $e$ the $0$-sections; 
by $\mathtt{aj}_{\L}$ the section determined by $\L_1^{r/k}(\alpha_1)$ and $\L_2^{r/k}(\alpha_2)$. Also by $\tilde{\mathtt{aj}}_{\R}$ we denote the section determined by $\R_{1},\R_{2}.$ 

We notice that the vertical maps in the  cartesian square are \'etale, non proper morphisms. 

Then the logDR loci $\operatorname{logDRL}^k(\R_1,\R_2)$ and $\operatorname{logDRL}(\L_1^{r/k}, \L_2^{r/k})$ are defined as the pullbacks of $0$-sections from the Jacobians via the Abel-Jacobi sections.

The Abel Jacobi section $\mathtt{aj}_{\L}$ is lci, which implies by flat base change that $\tilde{\mathtt{aj}}_{\R}$ is also lci.
We can thus endow $\operatorname{logDRL}^k(\R_1,\R_2)$ and $\operatorname{logDRL}(\L_1^{r/k}, \L_2^{r/k})$ with virtual classes via refined Gysin pull back defined from the top and bottom cartesian squares:
\begin{align*}
\vir{\operatorname{logDRL}^k(\R_1,\R_2)}:=&\tilde{\mathtt{aj}}_{\R}^{!}([U]) , \\
\vir{\operatorname{logDRL}(\L_1^{r/k}, \L_2^{r/k})}:=&\mathtt{aj}_{\L}^{!}([U]).
\end{align*}
Even though the morphism $p$ is non proper, its restriction to the log DDR loci is. This is due to the fact that the back square is cartesian. The  proper push-forward  \cite[Theorem~6.2]{fulton2013intersection} identifies the virtual classes. Note the refined gysin pullback used here to endow $\operatorname{logDRL}^k(\R_1,\R_2)$ with a virtual class is different from the discussion in \cite[Section 1.5]{chiodo2024hodge}. The difference comes from the target of the abel jacobi section induced by $\R_j$ being $\Jac_{C_j / U_k}$ vs. $\Jac_{p^*C_j/U^{spin}_k}$. Nevertheless they both lead to the same virtual class since the natural map $\Jac_{p^*C_j/U^{spin}_k}\to \Jac_{C_j / U_k}$ is étale.

\medskip

We now proceed to prove (2). We have constructed above a map, 
$$\mathrm{logDRL}^k(\R_1,\R_2)\to\mathrm{logDRL}(\L_1^{r/k},\L_2^{r/k})$$
compatible with the virtual classes. The locus $\mathrm{logDRL}(\L_1^{r/k},\L_2^{r/k})$ splits according to the Weil pairing $W_{r/k}(\L_1,\L_2)$, while the locus $\mathrm{logDRL}^k(\R_1,\R_2)$ splits according to the Weil pairing $W_{r/k}(\T_1,\T_2)$. However, on the DR-locus, we have $\R_1\equiv\R_2\equiv\O$. Therefore, we also have $\T_1\equiv\L_1$ and $\T_2\equiv\L_2$, so that the two refinements agree.
\end{proof}

\subsection{Proof of Multiple Cover Formula}

Having fixed suitable models for $S$ and $S\oner$ such that the log DR cycles are represented by the DR-class of twists of the line bundle on the fixed logarithmic model and can in fact computed applying
the universal DR-formula, and thanks to Proposition \ref{prop:push-forward-compatibility}, we can get tautological formulas for each $\DDR^\theta((\L_1^r)^\sigma,(\L_2^r)^\sigma)$. 

\medskip

To do so it is in fact sufficient to: first apply the UniDR formula \cite{bae2023pixton} restricting to the $\theta$-component  of the spin space $\tS_r\oner;$ then compute the push-forward along the map $\tS_r\oner\to\tS_r.$

This was in fact already the strategy adopted in \cite{blommecarocci2025DR} to prove the correlated refinement of the DR-formula with target varieties.

\medskip

The multiple cover formula for the DDR-class is then a consequence  of the degree computation  of the projection performed in Proposition~\ref{prop:degreecomputation}; in particular it does not rely on the specific form of Pixton's formula; we only need to know that the latter is given as a sum of decorated strata classes in the sense of \cite{janda2017DRcurves,holmes2025logDR}.

\begin{proof}[Proof of Theorem \ref{theo:MCF-FFR}]

\textbf{Step 1: Expression for $\mathbf{DDR}(\R_1,\R_2)$.}   By the almost twistability of $(\widetilde{C}^\sigma_{r,j}\to \tS_r\oner, \R_j,0)$ proved in Proposition~\ref{prop:almosttwistable}, the classes    $\DR(\R_1)$ and $\DR(\R_2)$  obtained applying the universal DR-formula for the representatives $\R_1,\R_2$
provide representatives for the logDR-classes. We then just have to multiply the two expressions to get the formula for the intersection.

Since, as discussed above, we do not need any properties of the explicit decoration coming from Pixton's universal formula, we write
the expression in the following concise form:
    \[\DDR(\R_1,\R_2) = \sum_{\Lambda,\phi_1,\phi_2} \left[ \tS_{r,\Lambda,\phi_1,\phi_2}\oner\right]\cap\operatorname{Dec}_{\Lambda,\phi_1,\phi_2},\]
    where, as explained in Section~\ref{sec-stratification-degree-computation}, $\Lambda$ indexes the strata of $\tS_r$, and $(\Lambda,\phi_1,\phi_2)$ indexes the strata of $\tS_r\oner$. The ``$\operatorname{Dec}$'' stands for ``Decoration''. These can be computed explicitly applying the UniDR formula \cite{bae2023pixton}; they  are cohomological insertions pulled back from the boundary of $\mathfrak{Pic}_{g,n,0}$ (involving in particular $\psi$-classes at the markings and at the half edges and Chern classes of the universal line bundle), with coefficients depending on $\Lambda,\phi_1,\phi_2$, namely on the topology of the dual graph $\Gamma$ and on the multidegree of $\R_1,\R_2$, fixed by $\phi_1,\phi_2$.

    \medskip
    
    In particular, since the decorations do not depend on the order of the log $r$-torsion bundle associated to $\R_1,\R_2$ nor on their Weil pairing,
 after restriction to any of the substacks $\tS_r^{k,\theta}\oner$ described in Section \ref{sec-stratification-degree-computation} we obtain the following formula for the :
    $$\DDR^{k,\theta}(\R_1,\R_2) = \sum_{\Lambda,\phi_1,\phi_2} \left[ \tS^{k,\theta}_{r,\Lambda,\phi_1,\phi_2} \oner\right]\cap\operatorname{Dec}_{\Lambda,\phi_1,\phi_2},$$
    where in the sum we only take the $\phi_j$ whose associated torsion divisor $T_j$ is of $(r/k)$-torsion.

\medskip

    \textbf{Step 2: Projection.}
    We push-forward to $\tS_r$ restricting to each stratum in the previous expressions as follows: the decoration $\operatorname{Dec}_{\Lambda,\phi_1,\phi_2}$ is unchanged since the cohomological insertions are actually pulled-back from $\tS_r$ (only the coeffecients depend on the multidegrees determined by $\phi_1,\phi_2$).
    
    Therefore, we only need to compute the degree of the projection restricted to each stratum; this has been done in Proposition \ref{prop:degreecomputation}. Since $\DDR^{k,\theta}(\R_1,\R_2)$ pushes forward to $\operatorname{logDDR}^\theta(\L_1^{r/k},\L_2^{r/k})$ by Propositions \ref{prop:change-order} and \ref{prop:push-forward-compatibility}, we get
    $$\operatorname{logDDR}^\theta(\L_1^{r/k},\L_2^{r/k}) = \sum_{\Lambda,\phi_1,\phi_2} d^{k,\theta}_{r,\Lambda,\phi_1,\phi_2}\cdot [\tS_{r,\Lambda}]\cap\operatorname{Dec}_{\Lambda,\phi_1,\phi_2}.$$
    Making the sum in the group algebra over all correlators $\theta$, and using Proposition \ref{prop:degreecomputation} computing the degree $d^{k,\theta}_{r,\Lambda,\phi_1,\phi_2}$, we get
    $$\mathbf{logDDR}(\L_1^{r/k},\L_2^{r/k}) = \sum_{\Lambda,\phi_1,\phi_2} \bfG_{b_1(\Gamma),2g}(\frac{r}{k},\omega_1,\omega_2)\cdot [\tS_{r,\Lambda}]\cap\operatorname{Dec}_{\Lambda,\phi_1,\phi_2},$$
    where $\omega_1$ and $\omega_2$ are the orders of the tropical divisors associated to $\phi_1$ and $\phi_2$, and $\bfG$ is understood to be $0$ whenever $\omega_1,\omega_2$ do not divide $r/k$.

    \medskip
    
    \textbf{Step 3: Conclusion.} Using linearity and summing over all $(\Lambda,\phi_1,\phi_2)$, the MCF is deduced from the algebraic $2g$-MCF for $l\mapsto\bfG_{b_1(\Gamma),2g}(l,\omega_1,\omega_2)$, which the content of Proposition \ref{prop:expression_G}.
\end{proof}

Notice that as a consequence of Proposition~\ref{prop:push-forward-compatibility}, and implicit in the proof of Theorem~\ref{theo:MCF-FFR} we have that:

\begin{coro}\label{cor:tautological}
    The correlated (logarithmic) DDR cycle $\operatorname{logDDR}^\theta(\L_1^{r},\L_2^{r})$ is tautological.
\end{coro}
\section{Multiple Cover Formula for toric surfaces}\label{sec:toric}
The goal of this section is to define the correlated refinement of the log GW invariants for toric surfaces and prove that these satisfy the multiple cover formula, giving evidence of Conjecture~\ref{conj:generalizedTakahashi}. We achieve this by adapting the methods of \cite{ranganathan2024logarithmic} to correlated logarithmic Gromov--Witten theory of toric surfaces (Definition \ref{def:corGWcycle}).

\subsection{Moduli space of log maps to ($K$-rubber) $\Gmlog^2$}
Let $\Sigma$ be a complete fan in $\mathbb{R}^2$ and $X=X_\Sigma$ the associated projective toric surface with dense torus $\Gm^2$.
We fix two tangency profiles $\bfa=(a_1, ..., a_n)$ and $\bfb=(b_1, ..., b_n)$. Let $\M_{g,\bfa,\bfb}(X)$ denote the moduli space of stable log maps to $X$ with discrete data given by $(g, \bfa,\bfb).$ A key observation in \cite{ranganathan2024logarithmic} is the utility of (virtual) birational models of $\M_{g,\bfa,\bfb}(X)$. Among them lies a distinguished (although not representable by an algebraic stack with log structure) minimal model given by the moduli space of stable log maps to $\Gmlog^2$.

\medskip

Let $K$ be a proper subgroup of $\Gmlog^2$ and let us denote by $K_{\mathrm{\mathrm{trop}}}$ the image under the morphism $\Gmlog^2\to \Gmtrop^2$. In practice, $K$ is  either $\Gmlog$, $\Gmlog^2$ or a finite subgroup of $\Gm^2$ (including the trivial group).
\begin{defi}
The moduli space of stable $K$-rubber maps to $\Gmlog^2$, denoted $\mathrm{R}^{K}_{g, \bfa, \bfb}(\Gmlog^2)$, is the fibered category over $\mathrm{LogSch^{fs}}$ whose fiber over $S$ are pairs 
$
\left( C\xrightarrow{\pi}S, \ (\widetilde{\alpha},\widetilde{\beta}) \right),
$
where $\pi$ is a stable log curve and $(\widetilde{\alpha},\widetilde{\beta})\in H^0(S,(\pi_{*}M_{C}^{\gp})/K )$, such that the induced PL map $\Gamma\xrightarrow{(\alpha,\beta)}\mathbb{R}^2$ (well-defined up to translation by $K_{\mathrm{\mathrm{trop}}}$) has  fixed slope  along the unbounded legs given by the vectors $(a_1,b_1), ..., (a_n,b_n)$.
\end{defi}
For $K$ trivial  we denote this space simply by $\M_{g, \bfa, \bfb}(\Gmlog^2)$ (\cite[Definition 1.5.1]{ranganathan2024logarithmic}) and for $K=\Gmlog^2$ we write $\mathrm{R}_{g, \bfa, \bfb}(\Gmlog^2)$ (\cite[Definition 1.6.1]{ranganathan2024logarithmic}). The first one is not representable by a stack with log structure and only exists as a fibered category over  $\mathrm{LogSch^{fs}}$. In contrast, the second one is representable by a stack with log structure and it coincides with the double double ramification locus $\mathrm{DDRL}(\mathcal{O}(\bfa), \mathcal{O}(\bfb))$ (See Example \ref{ex:rublog}). The relation between these two spaces is at the heart of the arguments in \cite{ranganathan2024logarithmic}.

For $K_1\subseteq K_2$ proper subgroups of $\Gmlog^2$ we get the natural map:

\[
\epsilon_{K_1, K_2}: \mathrm{R}^{K_1}_{g, \bfa, \bfb}(\Gmlog^2)\longrightarrow \mathrm{R}^{K_2}_{g, \bfa, \bfb}(\Gmlog^2),
\]
which can be shown, proceeding as in \cite[Section~1.7]{ranganathan2024logarithmic}, to be a $K_2 / K_1$-torsor
\medskip

\begin{rem}
The moduli space of stable $K$-rubber maps to $\Gmlog^2$ with $K \neq 1, \Gmlog^2$ were not defined in \cite{ranganathan2024logarithmic}. We have done so in this paper to correct a minor oversight in \cite{ranganathan2024logarithmic}, which is explained in Remark \ref{rem:torsion}. For the statements about the torsor structure of $\epsilon_{K_1,K_2}$ as above, and the virtual structure on these spaces we have chosen to still reference the propositions in \cite{ranganathan2024logarithmic} as the proofs only require straightforward modifications.
\end{rem}

The moduli space of stable $K$-rubber maps to $\Gmtrop^2$, denoted $\mathrm{R}^{K}_{g, \bfa, \bfb}(\Gmtrop^2)$, is defined in the same way. We just replace the section of $(\pi_{*}M_{C}^{\gp})^{\oplus 2} / K$ with a section  of $(\pi_{*}\overline{M}_{C}^{\gp})^{\oplus 2} / K_{\mathrm{\mathrm{trop}}}$. The latter can be viewed as a PL-map  $ \Gamma\to\mathbb{R}^2;$ we require it to be balanced.
Just as before for $K=1$ we denote this space by $\M_{g, \bfa, \bfb}(\Gmtrop^2)$ and for $K=\Gmlog^2$ we write $\mathrm{R}_{g, \bfa, \bfb}(\Gmtrop^2)$.
\subsubsection{The tropicalization of the DDR-locus}
In \cite{molcho2024case}, the authors introduce a combinatorial cone stack, which we denote by $\mathrm{DDR}_g^{\mathrm{\mathrm{trop}}}(\bfa,\bfb)$; in the notation of \cite{molcho2024case}, this stack is written $\mathrm{TC}_{g}(\bfa,\bfb)^{\mathrm{\mathrm{trop}}}$.
 Remark \ref{rem:tropDR}
We now describe its tropical moduli theoretic interpretation (for the exact construction we refer to \cite{molcho2024case}[Section 4.4]).

\medskip
An point in a cone of this cone stack  $p\in \mathrm{DDR}_g^{\mathrm{\mathrm{trop}}}(\bfa,\bfb)$ parametrizes the following data:
\begin{itemize}
     \item A metrized dual graph $\Gamma_p$ of a $n$-marked genus $g$ stable tropical curve.
     \item A balanced piece-wise linear map $f_p:\Gamma_p\to \mathbb{R}^2$, well-defined up to global $\mathbb{R}^2$-action, such that the slopes of the two projections along the unbounded legs are given by $\bfa$ and $\bfb$.
 \end{itemize} 

We write $[f_p:\Gamma_p\to \mathbb{R}^2]$ for the equivalence class of PL-maps parametrized by $p$, with $f_p$ a representative of that equivalence class. For a fixed stable tropical curve such a map, when it exists, this representative is unique up to $\mathbb{R}^2$-translation. 

The cones are indexed by the discrete data, i.e. the dual graph $\Gamma$ and the slope along the edges of the piece-wise linear function $f_p;$ the description of the cones is completely analogous to the one given in  Remark \ref{rem:tropDR} for the usual DR.

It follows from the modular description  that $\mathrm{DDR}_g^{\mathrm{\mathrm{trop}}}(\bfa,\bfb)$ is a substack of $\M^{\mathrm{\mathrm{trop}}}_{g,n}$. The relationship with $\mathrm{R}_{g,\bfa,\bfb}(\Gmtrop^2)$ is expressed by the following cartesian diagram of fs log stacks:
\[
\begin{tikzcd}
\mathrm{R}_{g, (\bfa, \bfb)}(\Gmtrop^2) \arrow[r,] \arrow[d,swap] \arrow[rd, phantom, "\square"]
& \mathrm{DDR}_g^{\mathrm{\mathrm{trop}}}(\bfa,\bfb) \arrow[d,] \\
\overline{\M}_{g,n} \arrow[r]
& \M^{\mathrm{\mathrm{trop}}}_{g,n}.
\end{tikzcd}
\]
where we implicitly used the equivalence of combinatorial cone stacks with the 2-category of Artin fans \cite{cavalieri2020moduli}.
In particular, $\mathrm{DDR}_g^{\mathrm{\mathrm{trop}}}(\bfa,\bfb)$ is the tropicalization of $\mathrm{R}_{g, (\bfa, \bfb)}(\Gmtrop^2)$. The rank $1$ version of the above cartesian square already appeared in Remark \ref{rem:tropDR}.

\subsubsection{Representable subdivisions and virtual structure}
\label{subsub:repsub}
There is a natural morphism 
\[
r_K:\mathrm{R}^{K}_{g, (\bfa, \bfb)}(\Gmlog^2)\to \mathrm{R}^{K}_{g, (\bfa, \bfb)}(\Gmtrop^2),
\]
which by \cite[Proposition 1.5.2]{ranganathan2024logarithmic} we know to be strict and representable.

Let $\mathrm{R}^{K, \dagger}_{g, (\bfa, \bfb)}(\Gmtrop^2)\to\mathrm{R}^{K}_{g, (\bfa, \bfb)}(\Gmtrop^2)$ be a logarithmic modification, where the source is represented by a stack over schemes with log structure. The  fiber product of fs log stracks
\[
\begin{tikzcd}
\mathrm{R}^{K, \dagger}_{g, (\bfa, \bfb)}(\Gmlog^2) \arrow[r, "r_K'"] \arrow[d,swap] \arrow[rd, phantom, "\square"]
& \mathrm{R}^{K, \dagger}_{g, (\bfa, \bfb)}(\Gmtrop^2) \arrow[d,] \\
\mathrm{R}^{K}_{g, (\bfa, \bfb)}(\Gmlog^2) \arrow[r,"r_K"]
&  \mathrm{R}^{K}_{g, (\bfa, \bfb)}(\Gmtrop^2),
\end{tikzcd}
\]
is represented by a stack over schemes with log structure. We call any stack constructed this way a \textit{representable logarithmic modification} of $\mathrm{R}^{K}_{g, (\bfa, \bfb)}(\Gmlog^2)$.

\medskip

The morphism $r_K$ is endowed with a relative perfect obstruction theory (see \cite[Section 3.3.2]{ranganathan2024logarithmic}). This endows any representable logarithmic modification of $\mathrm{R}^{K}_{g, (\bfa, \bfb)}(\Gmlog^2)$ with a virtual class. All these virtual classes are compatible under pushforward. Furthermore, we will consider representable logarithmic modifications of source and target of $\epsilon_{K_1, K_2}$:
\[
\tilde{\epsilon}_{K_1, K_2}\colon\mathrm{R}^{K_1, \dagger}_{g, (\bfa, \bfb)}(\Gmlog^2)\longrightarrow \mathrm{R}^{K_2, \dagger}_{g, (\bfa, \bfb)}(\Gmlog^2),
\]
such that $\tilde{\epsilon}_{K_1, K_2}$ is flat. In this case, we have the following compatibility of virtual classes (See \cite[Theorem 3.3.2]{ranganathan2024logarithmic}):
\begin{equation}
\label{eq:virtcomp}
(\tilde{\epsilon}_{K_1, K_2})^*([\mathrm{R}^{K_2, \dagger}_{g, (\bfa, \bfb)}(\Gmlog^2)]^{\mathrm{vir}})=[\mathrm{R}^{K_1, \dagger}_{g, (\bfa, \bfb)}(\Gmlog^2)]^{\mathrm{vir}}.
\end{equation}

\begin{rem}
In \cite[Proposition 1.5.5]{ranganathan2024logarithmic}, the authors prove that $\overline{\mathcal{M}}_{g, (\bfa, \bfb)}(X)$ is a representable logarithmic modification of $\overline{\mathcal{M}}_{g, (\bfa, \bfb)}(\Gmlog^2)$. 
\end{rem}

\subsubsection{Correlated refinement}
 Let $r$ be a positive integer dividing both $\bfa$ and $\bfb$. In Section \ref{subsub:corrhdr}, we defined open and closed components of $\mathrm{R}_{g, (\bfa, \bfb)}(\Gmlog^2)$ (the DDR-locus) indexed by $r$-th roots of unity. Consider the following composition:
 \[
\epsilon'\colon \overline{\mathcal{M}}_{g, (\bfa, \bfb)}(X) \to \overline{\mathcal{M}}_{g, (\bfa, \bfb)}(\Gmlog^2) \to \mathrm{R}_{g, (\bfa, \bfb)}(\Gmlog^2).
 \]
We can then define open and closed components of $\overline{\mathcal M}_{g, (\bfa, \bfb)}(X)$ as follows.

\begin{defi}
\label{def:corrloggw}
 Let $\theta$ be a $r$-th root of unity. We define
 \[
 \overline{\mathcal M}^{\theta}_{g, (\bfa, \bfb)}(X):=\epsilon'^{-1}(\mathrm{R}^{\theta}_{g, (\bfa, \bfb)}(\Gmlog^2)).
 \]
\end{defi}

Note that the same definition gives us correlated refinements of any representable logarithmic modification $\mathrm{R}^{K, \dagger}_{g, (\bfa, \bfb)}(\Gmlog^2)$ of $\mathrm{R}^{K}_{g, (\bfa, \bfb)}(\Gmlog^2)$, which we denote by $\mathrm{R}^{K, \dagger, \theta}_{g, (\bfa, \bfb)}(\Gmlog^2)$ for a fixed correlator $\theta$.

Furthermore, by restriction any representable subdivision $\mathrm{R}^{K, \dagger, \theta}_{g, (\bfa, \bfb)}(\Gmlog^2)$ carries a virtual class. Just as in the uncorrelated case these virtual classes on different representable subdivisions are related by proper pushforward. Similarly, the compatibility of virtual classes in \eqref{eq:virtcomp} also holds for the correlated components:
\begin{equation}
\label{eq:virtcomp2}
\tilde{\epsilon}_{K_1, K_2, \theta}^*([\mathrm{R}^{K_2, \dagger, \theta}_{g, (\bfa, \bfb)}(\Gmlog^2)]^{\mathrm{vir}})=[\mathrm{R}^{K_1, \dagger, \theta}_{g, (\bfa, \bfb)}(\Gmlog^2)]^{\mathrm{vir}}.
\end{equation}
Here $\tilde{\epsilon}_{K_1, K_2, \theta}$ denotes the restriction of $\tilde{\epsilon}_{K_1, K_2}$ to the correlated component.

\subsection{Evaluation Spaces}

\subsubsection{Rigid evaluation spaces}
From now on we assume that all $v_i=(a_i, b_i)$
are contained in rays of the fan $\Sigma;$ we also assume that for some $m\geq 0$ all $v_i$ for $i=1,..., m$ are the zero vector and the remaining $n-m$ vectors are nonzero. This can be achieved subdviding the fan $\Sigma$ and permuting the markings.

If $v_i=(a_i, b_i)\neq 0$, it determines a unique toric divisor $D_i$ of $X$ for $i=m+1, ..., n$. Of course we can have $D_i= D_j$ for $i\neq j.$

\begin{defi}
    The rigid evaluation space is defined as follows:
    \[
    \mathrm{Ev}_{\bfa,\bfb}(X):= X^m\times D_{m+1} \times\dots\times D_n.
    \]
\end{defi}
The moduli space $\overline{\mathcal{M}}_{g, (\bfa, \bfb)}(X)$ comes with an evaluation map
\[
\mathrm{ev}:\overline{\mathcal{M}}_{g, (\bfa, \bfb)}(X)\to \mathrm{Ev}_{\bfa,\bfb}(X).
\]
We consider the lattice $L:= (\mathbb{Z}^2)^{\oplus m}\oplus \bigoplus_{i=m+1}^n \mathbb{Z}^2 / \mathbb{Z} u_{i}$, where $u_i$ is the primitive integral vectors parallel to $v_i$ for $i=m+1, ..., n$. We can then descend the evaluation map to the moduli space of stable log maps to $\Gmlog^2$ (See \cite[Section 2.2]{ranganathan2024logarithmic}):
\[
\begin{tikzcd}
\overline{\M}_{g, (\bfa, \bfb)}(X) \arrow[r, "\mathrm{ev}"] \arrow[d,swap] 
& \mathrm{Ev}_{\bfa,\bfb}(X) \arrow[d,] \\
\overline{\M}_{g, (\bfa, \bfb)}(\Gmlog^2) \arrow[r,"\ev"]
&  L\otimes \Gmlog.
\end{tikzcd}
\]

\subsubsection{Correlated logarithmic Gromov--Witten cycles}\label{sec:correlatedlogGW}

We use the same notation as before for the restriction of $\mathrm{ev}$ to a correlated component  $\overline{\mathcal{M}}^{\theta}_{g, (\bfa, \bfb)}(X)$. Similarly we consider the forgetful map 
\[
\pi\colon \overline{\mathcal{M}}^{\theta}_{g, (\bfa, \bfb)}(X)\to \overline{\mathcal{M}}_{g,n}.
\]

\begin{defi}
\label{def:corGWcycle}
Let $\underline{\gamma}\in \mathrm{CH}^*(\mathrm{Ev}_{\bfa,\bfb}(X))$ and $\alpha\in \mathrm{CH}^*(\overline{\mathcal{M}}_{g,n})$. Fix a correlator $\theta\in\mu_r$. We define the associated correlated logarithmic Gromov-Witten cycle of $X$ as follows:
\[
\langle \alpha ; \underline{\gamma} \rangle_{g,(\bfa, \bfb)}^{\theta}:= \pi_{*}([\overline{\mathcal{M}}^{\theta}_{g, (\bfa, \bfb)}(X)]^{\mathrm{vir}}\cap( \pi^{*}\alpha\cup\ev^*\underline{\gamma})).
\]
\end{defi}
The recipe in \cite[Section 3.2]{ranganathan2024logarithmic} tells us how to lift this cycle class to $\mathrm{logCH}(\overline{\M}_{g,n})$. We denote this lift by the same notation.

\subsubsection{Rubber evaluation spaces}
We recall that $\Gm^2$ acts on the evaluation space $\mathrm{Ev}_{\bfa,\bfb}(X)$ via the following diagonal map on cocharacter lattices:
\[
N:=\mathbb{Z}^2\xrightarrow{ \ \phi \ } L= (\mathbb{Z}^2)^{\oplus m}\oplus \bigoplus_{i=m+1}^n \mathbb{Z}^2 / \mathbb{Z} u_{i},
\]
where each component is either the identity or the canonical projection.

\medskip

\paragraph{\bf{Assumption}}
For the remainder of the paper, we assume $\phi$ is injective. 
Indeed, setting $K:=\ker(\phi\otimes\Gmlog)$, the evaluation map factors as follows:

\[
 \begin{tikzcd}
 \overline{\M}_{g, (\bfa, \bfb)}(X)\arrow[r, "\ev"] \arrow[rd, "\epsilon_K"] & \mathrm{Ev}_{\bfa,\bfb}(X)\\    
 & \mathrm{R}^{K, \dagger}_{g, (\bfa, \bfb)}(\Gmlog^2) \arrow[u],
 \end{tikzcd}
\]
for some representable log modification $\mathrm{R}^{K,\dagger}_{g, (\bfa, \bfb)}(\Gmlog^2)$. 

If $\phi$ is not injective, then $K$ contains a logarithmic torus. We deduce \[(\epsilon_K)_{*}([\overline{\mathcal{M}}_{g, (\bfa, \bfb)}(X)]^{\mathrm{vir}})=0.\]
In this case all correlated logarithmic Gromov-Witten cycles are therefore zero. Under the injectivity assumption, $K$ is thus a finite group isomorphic to the torsion of $\mathrm{coker}(\phi)$. It is non-trivial when the primitive vectors of the ray do not span $\ZZ^2$.
\begin{remark}
\label{rem:torsion}
In \cite{ranganathan2024logarithmic}, the authors additionally assume $\mathrm{coker}(\phi)$ is torsion-free, noting that torsion does not affect the formulas. However, Proposition \ref{prop:rigidrubberform} shows that a global factor must be included in the presence of torsion.
\end{remark}
The quotient homomorphism $\psi\colon L \longrightarrow L / \phi(N)$ induces a morphism
\[
\psi\otimes \Gmlog: \mathrm{Ev}_{\bfa,\bfb}(X)\longrightarrow  (L / \phi(N) )\otimes \Gmlog.
\]
\begin{defi}\cite[Definition 2.4.2]{ranganathan2024logarithmic}
\label{def:rubev}
We call \emph{rubber evaluation space} $\mathrm{Ev}^{\mathrm{rub}}_{\bfa,\bfb}(X)$ a representable smooth subdivision of $(L / \phi(N) )\otimes \Gmlog$ (i.e. any smooth proper toric variety compactifying $(L / \phi(N) )\otimes \Gm$)  such that there exists a smooth subdivision $\mathrm{Ev}^{\dagger}_{\bfa,\bfb}(X)$ of $\mathrm{Ev}_{\bfa,\bfb}(X)$ with a weakly semistable morphism $q\colon \mathrm{Ev}^{\dagger}_{\bfa,\bfb}(X)\to \mathrm{Ev}^{\mathrm{rub}}_{\bfa,\bfb}(X)$ such that the following diagram commutes:
\begin{equation*}
\begin{tikzcd}
\mathrm{Ev}^{\dagger}_{\bfa,\bfb}(X) \arrow[r, ] \arrow[d,"q"] 
& \mathrm{Ev}_{\bfa,\bfb}(X) \arrow[d,] \\
\mathrm{Ev}^{\mathrm{rub}}_{\bfa,\bfb}(X) \arrow[r,]
&  (L / \phi(N) )\otimes \Gmlog.
\end{tikzcd}
\end{equation*}
\end{defi}

\begin{rem}
The following construction from \cite[Remark 2.4.3]{ranganathan2024logarithmic} shows that rubber evaluation spaces always exist: Start with
any representable subdivision $Q$ of $(L / \phi(N) )\otimes \Gmlog$. We then consider $B := Q \times_{(L / \phi(N) )\otimes \Gmlog} \mathrm{Ev}_{\bfa, \bfb}(X)$ with its
morphism $b: B \to Q$. Applying semistable reduction, we can assume $b$ is weakly semistable with smooth source and target.
\end{rem}

For the remainder of the section, we fix a choice of rubber evaluation space $\mathrm{Ev}^{\mathrm{rub}}_{\bfa,\bfb}(X)$ with a choice of $q\colon \mathrm{Ev}^{\dagger}_{\bfa,\bfb}(X)\to \mathrm{Ev}^{\mathrm{rub}}_{\bfa,\bfb}(X)$ as in Definition \ref{def:rubev}.

\begin{defi}
For any $\underline{\gamma}\in \mathrm{CH}^*(\mathrm{Ev}_{\bfa,\bfb}(X))$, we define the associated rubber insertion $\gamma_{\mathrm{rub}}\in \mathrm{CH}^*(\mathrm{Ev}^{\mathrm{rub}}_{\bfa,\bfb}(X))$ by $\gamma_{\mathrm{rub}}:=q_*(\underline{\tilde{\gamma}})$, where $\underline{\tilde{\gamma}}$ is the pullback of $\underline{\gamma}$ to $\mathrm{Ev}^{\dagger}_{\bfa,\bfb}(X)$.
\end{defi}

In particular, we have that
$$\deg\gamma_\mathrm{rub} = \deg\underline{\gamma}-2.$$

\subsubsection{Rubber evaluation maps}
\label{subsub:rubevm}
In \cite[Section 2.4]{ranganathan2024logarithmic}, the authors construct a rubber evaluation map:
\[
\ev_{\mathrm{rub}}: \mathrm{R}_{g, (\bfa,\bfb)}(\Gmlog^2)\longrightarrow (L / \phi(N) )\otimes \Gmlog.
\]
This map induces a map at the level of $\Gmtrop$:
\[
\mathrm{ev}'_{\mathrm{rub}}: \mathrm{R}_{g, (\bfa,\bfb)}(\Gmtrop^2)\longrightarrow (L / \phi(N) )\otimes \Gmtrop.
\]
By the definition, a morphism to $\Gmtrop$ is equivalent to a section of $\bar{M}^{\gp}_{\mathrm{R}_{g, (\bfa,\bfb)}(\Gmtrop^2)}$, which in turn is
 equivalent to a strict PL-map:
\[
\mathrm{ev}^{\mathrm{trop}}_{\mathrm{rub}}: \mathrm{DDR}^{\mathrm{\mathrm{trop}}}_g(\bfa,\bfb)\longrightarrow (L / \phi(N) )\otimes \mathbb{R}=\left( (\mathbb{R}^2)^{\oplus m}\oplus \bigoplus_{i=m+1}^n \mathbb{R}^2 / \mathbb{R} u_{i}\right) / \phi(\mathbb{R}^2).
\]

The map $\mathrm{ev}^{\mathrm{trop}}_{\mathrm{rub}}$ can be described explicitly. Let $p\in \mathrm{DDR}_{g}^{\mathrm{trop}}(\bfa, \bfb)$. Recall that $p$ corresponds to an equivalence class (with respect to the translation action) of PL maps. We choose a representative $f_p:\Gamma\to \mathbb{R}^2$ of this equivalence class, for example by fixing the image of one of the vertices of $\Gamma$ to be the origin in $\mathbb R^2.$

\medskip

For every marking $i$ with $i=1, ..., m$, (these are the marking carrying trivial logarithmic structure, i.e. those where the contact order vector $v_i$ is trivial) we have that $f_p$ contracts the unbounded leg corresponding to $i$ to a point in $\mathbb{R}^2$. Similarly for every marking $i=m+1, ..., n$, under the composition
\[
\Gamma\longrightarrow\mathbb{R}^2\longrightarrow \mathbb{R}^2 / \mathbb{R} u_i,
\]
the unbounded leg corresponding to $i$ is contracted to a point in $\mathbb{R}^2 / \mathbb{R} u_i$. Put together, we have associated to $f_p$ a point in $L\otimes\mathbb{R}$. Quotienting by the translation action translations, we obtain a map  
\[
\mathrm{DDR}^{\mathrm{\mathrm{trop}}}_g(\bfa,\bfb)\longrightarrow \left( (\mathbb{R}^2)^{\oplus m}\oplus \bigoplus_{i=m+1}^n \mathbb{R}^2 / \mathbb{R} u_{i}\right)/\phi(\mathbb{R}^2).
\]

This is in fact an explicit description of the map $\mathrm{ev}^{\mathrm{trop}}_{\mathrm{rub}}$. Our fixed choice of rubber evaluation space induces by fs fiber product a rubber evaluation map:
\[
\mathrm{ev}_{\mathrm{rub}}: \mathrm{R}^{\dagger}_{g, (\bfa,\bfb)}(\Gmlog^2)\longrightarrow \mathrm{Ev}^{\mathrm{rub}}_{\bfa,\bfb}(X).
\]

\subsection{Rigid to Rubber}


\begin{prop}\cite[Proposition 2.4.5]{ranganathan2024logarithmic}
\label{prop:cartdiag}
Let $K=\ker(\phi\otimes\Gmlog)$. There exists a representable subdivision $\mathrm{R}^{K, \dagger}_{g, (\bfa, \bfb)}(\Gmlog^2)$ such that the following diagram is cartesian:
\begin{equation*}
\begin{tikzcd}
\mathrm{R}^{K, \dagger}_{g, (\bfa, \bfb)}(\Gmlog^2) \arrow[r, "\ev"] \arrow[d,"\tilde{\epsilon}"] \arrow[rd, phantom, "\square"]
& \mathrm{Ev}^{\dagger}_{\bfa,\bfb}(X) \arrow[d,"q"] \\
\mathrm{R}^{\dagger}_{g, (\bfa,\bfb)}(\Gmlog^2) \arrow[r, "\ev_{\mathrm{rub}}"]
&  \mathrm{Ev}^{\mathrm{rub}}_{\bfa,\bfb}(X),
\end{tikzcd}
\end{equation*}
where $\tilde{\epsilon}$ is a representable subdivision of $\epsilon$ and $\mathrm{ev}$ is a representable subdivision of 
\[
\mathrm{R}^{K}_{g, (\bfa, \bfb)}(\Gmlog^2)\to L\otimes \Gmlog.
\]
\end{prop}
Let $\underline{\gamma}\in \mathrm{CH}^*(\mathrm{Ev}_{\bfa,\bfb}(X))$. Since $ \mathrm{Ev}^{\mathrm{rub}}_{\bfa,\bfb}(X)$ is a smooth toric variety, the associated rubber insertion $\gamma_{\mathrm{rub}}$ has a piecewise polynomial lift which we denote by $\gamma^{\mathrm{trop}}_{\mathrm{rub}}$ (See Remark \ref{rem:toricpp}).

This piecewise polynomial is homogeneous of degree $\deg\gamma_\mathrm{rub}=\deg\underline{\gamma}-2$.

In the next proposition we transfer the calculation of our correlated logarithmic GW invariant to an intersection theory problem on the moduli space of curves. To do so we will make use of extensions of $\gamma^{\mathrm{trop}}_{\mathrm{rub}}\circ \ev^{\mathrm{trop}}_{\mathrm{rub}}$ to $\mathcal{M}^{\mathrm{trop}}_{g,n}$ which do exist by \cite{ranganathan2024logarithmic}[Remark 2.4.3].
\begin{prop}
\label{prop:rigidrubberform}
Let $\alpha\in \mathrm{CH}^*(\overline{\M}_{g,n})$, $\underline{\gamma}\in\mathrm{CH}^*(\mathrm{Ev}_{\bfa,\bfb}(X))$ and $\theta\in\mu_r$ a correlator for $r$ the g.c.d. of $\bfa$ and $\bfb$. Let $\phi_{\bfa,\bfb}$ be an extension of the piecewise polynomial $\gamma^{\mathrm{trop}}_{\mathrm{rub}}\circ \ev^{\mathrm{trop}}_{\mathrm{rub}}$ to $\mathcal{M}^{\mathrm{trop}}_{g,n}$ 
 and $K=\ker(\phi\otimes\Gmlog)$. We have the following formula:
\[
\langle \alpha ; \underline{\gamma} \rangle_{g,(\bfa, \bfb)}^{\theta}=|K|\cdot\mathrm{DDR}^{\theta}_{g}(\bfa,\bfb)\cap  \phi_{\bfa,\bfb}.
\]
\end{prop}

\begin{proof}
For any subdivision of $\epsilon_{1,K}$:
\[
\tilde{\epsilon}_{1,K}\colon \overline{\M}^{\dagger}_{g,\bfa,\bfb}(\Gmlog^2)\to \mathrm{R}_{g,\bfa,\bfb}^{K,\dagger}(\Gmlog^2)
\]
is a $K$-torsor. We conclude by \eqref{eq:virtcomp}
\[
(\tilde{\epsilon}_{1,K})_{*}([\overline{\M}_{g,\bfa,\bfb}^{\dagger}(\Gmlog^2)]^{\mathrm{vir}})=|K|\cdot [\mathrm{R}_{g,\bfa,\bfb}^{K,\dagger}(\Gmlog^2)]^{\mathrm{vir}}.
\]
For the remainder of the proof assume $K$ is trivial. 
We leave the necessary modifications of the proof for $|K|>1$ to the reader. Consider
\[
\pi: R^{\dagger}_{g, (\bfa,\bfb)}(\Gmtrop^2)\longrightarrow \overline{\M}_{g,n}
\]
By replacing domain and codomain of $\pi$ with log modifications, we can assume this morphism is weakly semistable. We do this to ensure that we determine the lift to $\mathrm{logCH}(\overline{\M}_{g,n})$ of the cycle $ \langle \alpha;\underline{\gamma} \rangle_{g,(\bfa, \bfb)}^{\theta}$.  From Proposition \ref{prop:cartdiag}, and its obvious restriction to the correlated components 
we have the following diagram:
\begin{equation*}
\begin{tikzcd}
\overline{\M}^{\dagger, \theta}_{g, (\bfa, \bfb)}(\Gmlog^2) \arrow[r] \arrow[d,"\tilde{\epsilon_{\theta}}"] \arrow[rd, phantom, "\square"]
&  \overline{\M}^{\dagger}_{g, (\bfa, \bfb)}(\Gmlog^2) \arrow[r, "\ev"] \arrow[d,"\tilde{\epsilon}"] \arrow[rd, phantom, "\square"]
&  \mathrm{Ev}^{\dagger}_{\bfa,\bfb}(X) \arrow[d,"q"]\\
R^{\dagger,\theta}_{g, (\bfa,\bfb)}(\Gmlog^2) \arrow[rdd, swap,"\pi_{\theta}"] \arrow[r]
& R^{\dagger}_{g, (\bfa,\bfb)}(\Gmlog^2) \arrow[d, "k"] \arrow[r, "\ev_{\mathrm{rub}}"]
&  \mathrm{Ev}^{\mathrm{rub}}_{\bfa,\bfb}(X).\\
&  R^{\dagger}_{g, (\bfa,\bfb)}(\Gmtrop^2) \arrow[d, "\pi"]\\
&  \overline{\M}_{g,n}^{\dagger}
\end{tikzcd}
\end{equation*}

Let $\overline{\M}^{\dagger}_{g, (\bfa, \bfb)}(X)$ be a representable logarithmic modification of $\overline{\M}^{\dagger}_{g, (\bfa, \bfb)}(\Gmlog^2)$ with a proper map $p\colon \overline{\M}^{\dagger}_{g, (\bfa, \bfb)}(X)\longrightarrow \overline{\M}^{\dagger}_{g, (\bfa, \bfb)}(\Gmlog^2)$ which identifies virtual classes under proper pushforward. We denote its restriction to a correlated component by $p_{\theta}$. Then, by the definition of correlated logarithmic Gromov-Witten cycle, we have:
\[
\langle \alpha ; \underline{\gamma} \rangle_{g,(\bfa, \bfb)}^{\theta}=(\pi_{\theta} \circ \tilde{\epsilon}_{\theta} \circ p_{\theta})_{*}([\overline{\M}^{\dagger, \theta}_{g, (\bfa, \bfb)}(X)]^{\mathrm{vir}} \cap (\ev_{\theta}\circ p_{\theta})^{*}(\tilde{\gamma})),
\]
where $\tilde{\gamma}$ is the pullback of $\underline{\gamma}$ to $\mathrm{Ev}^{\dagger}_{\bfa,\bfb}(X)$. We now proceed exactly the same as in the proof in \cite[Section 3.4]{ranganathan2024logarithmic}. By projection formula we deduce:
\[
(p_{\theta})_{*}([\overline{\M}^{\dagger, \theta}_{g, (\bfa, \bfb)}(X)]^{\mathrm{vir}} \cap (\ev_{\theta}\circ p_{\theta})^{*}(\tilde{\gamma}))=[\overline{\M}^{\dagger, \theta}_{g, (\bfa, \bfb)}(\Gmlog^2)]^{\mathrm{vir}} \cap \ev_{\theta}^{*}(\tilde{\gamma}).
\]
Then we apply compatibilty of pushforward and pullback in cartesian squares and equation \eqref{eq:virtcomp}:
\[
(\tilde{\epsilon}_{\theta})_{*}([\overline{\M}^{\dagger, \theta}_{g, (\bfa, \bfb)}(\Gmlog^2)]^{\mathrm{vir}} \cap\ev_{\theta}^{*}(\tilde{\gamma}))=[\mathrm{R}^{\dagger,\theta}_{g, (\bfa,\bfb)}(\Gmlog^2)]^{\mathrm{vir}}\cap \ev_{\theta,\mathrm{rub}}^{*}(\gamma_{\mathrm{rub}}).
\]
Let $\mathcal{A}_E$ denote the Artin fan of $\mathrm{Ev}^{\mathrm{rub}}_{\bfa,\bfb}(X)$. Recall from Section \ref{subsub:rubevm} that the rubber evaluation map factors as follows:
\begin{equation*}
\begin{tikzcd}
\mathrm{R}^{\dagger}_{g, (\bfa,\bfb)}(\Gmlog^2) \arrow[r, "\ev_{\mathrm{rub}}" ] \arrow[d,"\mathrm{Trop}"] 
& \mathrm{Ev}^{\mathrm{rub}}_{\bfa,\bfb}(X) \arrow[d,] \\
\mathrm{R}^{\dagger}_{g, (\bfa,\bfb)}(\Gmtrop^2) \arrow[r, "\ev'^{\dagger}"]
&  \mathcal{A}_E.
\end{tikzcd}
\end{equation*}
We then deduce
\[
\mathrm{Trop}^*(\gamma_{\mathrm{rub}}^{\mathrm{trop}}\circ\ev_{\mathrm{rub}}^{\mathrm{trop}})=\ev_{\mathrm{rub}}^*(\gamma_{\mathrm{rub}}).
\]
By projection formula and since $\mathrm{Trop}^*(\gamma_{\mathrm{rub}}^{\mathrm{trop}}\circ ev_{\mathrm{rub}}^{\mathrm{trop}})=\pi^*(\phi_{\bfa,\bfb})$ we conclude:
\begin{gather*}
(\pi_{\theta})_{*}([\mathrm{R}^{\dagger,\theta}_{g, (\bfa,\bfb)}(\Gmlog^2)]^{\mathrm{vir}}\cap \ev_{\theta,rub}^{*}(\gamma_{\mathrm{rub}}))=(\pi_{\theta})_{*}([\mathrm{R}^{\dagger,\theta}_{g, (\bfa,\bfb)}(\Gmlog^2)]^{\mathrm{vir}}\cap (\pi_{\theta})^{*}(\phi_{\bfa,\bfb}))=\\
\DDR_{g}^{\theta}(\bfa,\bfb)\cap \phi_{\bfa,\bfb}.
\end{gather*}
\end{proof}

\begin{theo}\label{thm:toriccorrelated}
Let $\alpha\in \mathrm{CH}^*(\overline{\M}_{g,n})$, $\underline{\gamma}\in\mathrm{CH}^*(\mathrm{Ev}_{\bfa,\bfb}(X))$ and $\theta\in\mu_r$ a correlator for $r$ the g.c.d. of $\bfa$ and $\bfb$. The associated correlated logarithmic Gromov-Witten invariants satisfy the following multiple cover formula:
\[
\langle \alpha ; \underline{\gamma} \rangle_{g,(\bfa, \bfb)}^{\theta}=\sum_{k|\bfa,\bfb,\theta} k^{3g-3+n-\deg\alpha}\langle \alpha ; \underline{\gamma} \rangle_{g,(\frac{\bfa}{k}, \frac{\bfb}{k})}^{\theta_{\mathrm{prim}}}.
\]
\end{theo}
\begin{proof}
Let $k|\bfa,\bfb$. We construct a morphism $i_k\colon \mathrm{DDR}_{g}^{\mathrm{trop}}(\frac{\bfa}{k},\frac{\bfb}{k})\hookrightarrow \mathrm{DDR}_{g}^{\mathrm{trop}}(\bfa,\bfb)$ as follows. Let $p=[f_p:\Gamma_p\to \mathbb{R}^2]\in \mathrm{DDR}_{g}^{\mathrm{trop}}(\frac{\bfa}{k},\frac{\bfb}{k})$.
The following composition:
\[
\Gamma_p\xrightarrow{f_p} \mathbb{R}^2\xrightarrow{k\cdot} \mathbb{R}^2, 
\]
we denote by $k\cdot f_p$. We then clearly have $[k\cdot f_p]\in\mathrm{DDR}_{g}^{\mathrm{trop}}(\bfa,\bfb)$. The morphism $i_k$ this defines commutes with the inclusions into $\mathcal{M}_{g,n}^{\mathrm{trop}}$. From the definitions of the tropical rubber evaluation maps we deduce:
\begin{equation*}
\begin{tikzcd}
\DDR_{g}^{\mathrm{trop}}(\frac{\bfa}{k}, \frac{\bfb}{k}) \arrow[r, "\ev_{rub,k}^{\mathrm{trop}}" ] \arrow[d, hook, "i_k"] 
& (L / \phi(N) )\otimes \mathbb{R} \arrow[d, "k\cdot"] \\
\DDR_{g}^{\mathrm{trop}}(\bfa, \bfb) \arrow[r, "\ev_{rub,1}^{\mathrm{trop}}"]
&  (L / \phi(N) )\otimes \mathbb{R}.
\end{tikzcd}
\end{equation*}
Let $\phi_{\bfa,\bfb}$ be an extension of the piecewise polynomial $\gamma^{\mathrm{trop}}_{\mathrm{rub}}\circ \ev^{\mathrm{trop}}_{\mathrm{rub},1}$ to $\mathcal{M}^{\mathrm{trop}}_{g,n}$. From the diagram we deduce $k^{-\deg\gamma_\mathrm{rub}}\phi_{\bfa,\bfb}$ is an extension of the piecewise polynomial $\gamma^{\mathrm{trop}}_{\mathrm{rub}}\circ \ev^{\mathrm{trop}}_{\mathrm{rub},k}$ to to $\mathcal{M}^{\mathrm{trop}}_{g,n}$. Using Theorem \ref{theo:MCF-FFR} and Proposition \ref{prop:rigidrubberform}, we get:
\begin{gather*}
\langle \alpha ; \underline{\gamma} \rangle_{g,(\bfa, \bfb)}^{\theta}=|K|\cdot\mathrm{DDR}^{\theta}_{g}(\bfa,\bfb)\cap \phi_{\bfa,\bfb}=\\
\sum_{k|\bfa,\bfb,\theta}k^{2g}|K|\DDR^{\theta_\mathrm{prim}}(\frac{\bfa}{k},\frac{\bfb}{k})\cap \phi_{\bfa,\bfb}=\sum_{k|\bfa,\bfb,\theta}k^{2g+\deg\gamma_\mathrm{rub}}\langle \alpha ; \underline{\gamma} \rangle_{g,(\frac{\bfa}{k}, \frac{\bfb}{k})}^{\theta_\mathrm{prim}}.
\end{gather*}
Finally, since the insertions match the rank of the virtual class, we have:
$$n+g-1 = \deg\alpha+\deg\underline{\gamma}.$$
Using that $\deg\gamma_\mathrm{rub}=\deg\underline{\gamma}-2$, we get that
$$2g+\deg\gamma_\mathrm{rub} = 3g-3+n-\deg\alpha.$$
\end{proof}

\begin{color}{red}

\end{color}

\bibliographystyle{alpha}
\bibliography{biblio} 

\end{document}